\documentclass[10pt,a4paper, openany,  oldfontcommands]{memoir}
                                                      \usepackage[utf8]{inputenc}
                                                      \usepackage[english]{babel}
                                                      \usepackage{amsmath}
                                                      \usepackage{amsthm}
                                                      \usepackage{thmtools}
                                                      \usepackage{mathtools}
                                                      \usepackage{amsfonts}
                                                      \usepackage{amssymb}
                                                      \usepackage{makeidx}
                                                      \usepackage{graphicx}
                                                      \usepackage{tikz-cd}
                                                      \usepackage{float}
                                                      \usepackage{authblk}

                                                      \usepackage[linguistics]{forest}

                                                      \usepackage{hyperref}
                                                      \usepackage{doi}

                                                      \usepackage{xcolor}

                                                      \usepackage{stmaryrd}

             \title{A non-sequential arithmetical theory  with pairing}

                                                        \author{Juvenal Murwanashyaka}
                                                        \affil{Institute of Mathematics,  Czech Academy of Sciences,   Czech Republic}

                                                      \newtheorem{theorem}{Theorem} 
                                                      \newtheorem{lemma} [theorem] {Lemma}
                                                       \newtheorem{corollary} [theorem] {Corollary}
                                                      \newtheorem{definition} [theorem] {Definition} 
                                                       
                                                       \newtheorem{open problem} [theorem] {Open Problem}

                                                           \newtheorem{open question} [theorem] {Open Question}

                                                      \setsecnumdepth{subsubsection} 
                                                      \settocdepth{subsubsection}

                 \renewcommand{\mathrm}{\mathsf}

\begin{document}

\maketitle

\begin{abstract}
Albert Visser 
has shown that Robinson's  $  \mathsf{Q} $ and Gregorczyk's $  \mathsf{TC} $  are not sequential by showing that these theories are not even poly-pair theories, which, in a strong sense, means these theories lack pairing.   
In this paper, we  use   Ehrenfeucht-Fra\"iss\'e games  to  show  that  the theory $  \mathsf{Q} + \Theta  $    we obtain by extending   Robinson's $  \mathsf{Q} $  with an axiom  $  \Theta $  which says that the map     $  \pi (x, y ) = (x+y)^2 + x $  is a pairing function 
is not sequential; in fact,   we show that this theory  is not   even  a Vaught theory.
As a corollary,  we get that the tree theory  $  \mathsf{T} $  of  [Kristiansen \& Murwanashyaka,  2020] is also not a Vaught theory. 
\end{abstract}

\section{Introduction}

The concept of  sequential theories was introduced by   Pudl\' ak  \cite{Pudlak1983} in the  study of  degrees of   local  interpretability. 
 Sequential  theories are   theories with a coding machinery, for  all objects in the domain of the theory,    sufficient for developing 
partial  satisfaction predicates for formulas of bounded depth-of-quantifier-alternations  (see Visser \cite{Visser2019}). 
Formally, a first-order theory  $T$ is  \emph{sequential}  if  it   directly interprets  adjunctive set theory  $  \mathsf{AS} $; 
in this paper,   direct interpretations are 1-dimensional unrelativized interpretations  with absolute equality.
A related but strictly weaker notion is the class of \emph{Vaught theories}, that is, 
theories that directly interpret  a weak set theory  $  \mathsf{VS}   $  of   Vaught    \cite{Vaught1967}, 
a non-finitely axiomatizable fragment of  $  \mathsf{AS} $. 
See Figure  \ref{AxiomsOfAS}  for the axioms of  $  \mathsf{AS} $  and   $  \mathsf{VS} $.

\begin{figure} 

\[
\begin{array}{r l  c  c r l  }
&{\large \textsf{The Axioms of } \mathsf{AS} }
&
\\
\\
\mathsf{AS}_1 
& 
 \exists x  \;  \forall y   \;   [    \   y \not\in x    \   ]   
\\
\mathsf{AS}_2 
& 
 \forall x y   \;  \exists z  \;    \forall w   \;   [    \   w \in z  \leftrightarrow   (    \    w \in x   \;  \vee   \;    w  =  y     \    )      \      ]   
 \\
 \\
 &{\large \textsf{The Axioms of } \mathsf{VS} }
&
\\
\\
\mathsf{VS}_0 
& 
  \exists x  \;  \forall y   \;   [       \   y \not\in x    \       ]  
\\
\mathsf{VS}_n 
& 
 \forall  x_1  \ldots  x_n    \;  \exists  y   \;   \forall z    \;   [     \   z \in y  \leftrightarrow  \bigvee_{ i = 1 }^{ n }   z = x_i       \        ] 
 \   \    \mbox{ for  }  0 < n <  \omega     
 \end{array}
\] 

\caption{
Non-logical axioms of the first-order theories   $ \mathsf{AS} $  and  $ \mathsf{VS} $.
}
\label{AxiomsOfAS}
\end{figure}

\begin{definition}
Let  $S  \in  \lbrace   \mathsf{VS} ,  \mathsf{AS} \rbrace   $. 
A first-order theory  $T$  \emph{directly interprets} $S$  if there exist  first-order formulas  
$  \alpha (u_0,   \ldots  ,  u_k  )  $  and  
$  \phi (x, y  ,  u_0,  \ldots , u_k ) $  in the language of  $T$, with all free variables displayed, 
such that  the following holds  for all models  $  \mathcal{M}  \models   T $  with universe $M$: 
\begin{enumerate}
\item    $  \mathcal{M} \models    \exists    \vec{u}    \left[  \,   \alpha (u_0,   \ldots  ,  u_k  )    \,    \right]   $

\item   for   all  $  p_0,  \ldots , p_k  \in  M  $  such that 
$  \mathcal{M}  \models  \alpha (  \vec{p}  )  $   
\[
 \left(  M,  R_{  \vec{p} }    \right)  \models  S 
 \     \   \mbox{  where  }    \     \  
  R_{  \vec{p} }   = 
   \big\{    ( x, y  )  \in  M^2  :   \     \mathcal{M} \models    \phi ( x,  y  ,  \vec{p} )    \,     \big\}      
 \         .
 \]
\end{enumerate}
\end{definition}

Examples of    sequential theories are   adjunctive set theory $ \mathsf{AS} $ (see  Mycielski  et al.  \cite[Appendix III]{AMSVolume426Year1990}),  
the theory of discretely ordered commutative semirings with a least element $ \mathsf{PA^-} $  (see Jeřábek  \cite{Jerabek2012}), 
the theory   $ \mathsf{I \Delta_0 } $ (see Hájek \&  Pudlák  \cite{HajekandPudlak2017}  Section V3b), 
Peano arithmetic $ \mathsf{PA} $, 
Zermelo-Fraenkel set theory $ \mathsf{ZF} $.
Sequential theories are Vaught theories  since $   \mathsf{AS} $ contains $  \mathsf{VS} $,  but the converse does not hold. 
For example,   $  \mathsf{VS} $   is not sequential: 
sequential theories  have finitely axiomatizable  essentially  undecidable subtheories since  they  interpret   $  \mathsf{AS} $, 
which is  essentially undecidable, 
 but   finite subtheories  of  $ \mathsf{VS} $ are  not essentially undecidable; 
in fact, Jeřábek  \cite{Jerabek2022} shows that the minimal models of finite fragments of $ \mathsf{VS} $, 
that  is,  the structures $  \left(   H_{  \kappa  }     ,    \in  \right)   $  of sets   hereditarily of cardinality $  <  \kappa  $,    for $  \kappa $ a non-zero finite cardinal, 
have decidable first-order theories with  transparent computable axiomatization. 
Kurahashi \& Visser   \cite{KurahashiVisser}  have      shown that there   exist  finitely axiomatizable   Vaught theories  that are   not sequential; 
furthermore,  there exist such theories  that are  same-language extensions of    $  \mathsf{VS} $.

Vaught  theories are  essentially undecidable  since they interpret the very weak theory   $  \mathsf{R} $  of  
Robinson, Tarski \& Mostowski   \cite{tarski1953}; 
  $  \mathsf{R} $   is not a Vaught theory since it is locally finitely satisfiable while Vaught theories are not. 
Other examples of non-Vaught theories are 
Robinson's  $ \mathsf{Q} $  (see    Visser \cite{Visser2017}) 
and   Gregorczyk's    theory of concatenation $ \mathsf{TC} $ (see   Visser   \cite{Visser2009NDJFL}  Section 5). 
Visser shows  that $ \mathsf{Q} $  and  $ \mathsf{TC} $ are not sequential by showing that these theories are not even poly-pair theories, 
that is, there is no many-dimensional direct interpretation in these theories 
of the fragment of  $  \mathsf{VS} $  given by  $  \mathsf{VS}_0  $  and  $  \mathsf{VS}_2  $. 
These two examples show that Vaughtness and sequentiality  are  not preserved by interpretations since 
$  \mathsf{Q} ,  \mathsf{TC},  \mathsf{AS},  \mathsf{PA}^- $  and  $ \mathsf{I \Delta_0 } $ are  mutually interpretable 
(see  Ferreira \& Ferreira \cite{Ferreira2013} for a survey of  theories  interpretable in $  \mathsf{Q} $).
These properties are,  though,  preserved by   direct interpretations
(see the definition above). 

The main result of this paper is that there exists a non-sequential arithmetical theory with pairing, in contrast to $  \mathsf{Q} $  and  $  \mathsf{TC} $ which   are  not    even poly-pair theories;  
an arithmetical theory is a theory in the language of arithmetic that extends   Robinson's   $\mathsf{Q} $
(see Figure  \ref{AxiomsOfQ}  for the axioms of  $  \mathsf{Q} $). 
Let  
\[
\Theta :=   \forall xy z w  \left[  \,   (x+y)^2 + x = (z+w)^2 + z \rightarrow x = z  \;  \wedge  \;  y = w  \,  \right] 
\]
This sentence  says that the map  $  \pi  (x, y ) = (x+y)^2 + x $ is one-to-one. 
In this paper,  we show that  the theory  $  \mathsf{Q} + \Theta $  is not a Vaught theory and hence not a sequential theory. 
The theory  $  \mathsf{Q} + \Theta $  is a fragment of  $  \mathsf{PA}^- $,   and the map $ \pi $ is one of the ingredients used by 
Emil   Jeřábek \cite{Jerabek2012} to show that    $ \mathsf{PA^-} $  is  a  sequential theory. 
Cantor's pairing function cannot be used since   $  \mathsf{PA}^- $  does not prove
 $ \forall x y  \,  \exists z   \left[  \,   2z = (x+y) (x+y+1) \,   \right]  $; 
for example, in the model $  \mathbb{Z}  \left[ X \right]^+ $, the non-negative part of the polynomial ring   $ \mathbb{Z}  \left[ X \right] $, 
monic polynomials of degree $  \geq 1  $ are neither even nor odd.

\begin{figure}

\[
\begin{array}{r l  c  c r l  }
& {\large \textsf{The Axioms of } \mathsf{Q} }
\\
\\
 \mathsf{Q_1}  
& \forall x   y \;  [ \  x \neq y \rightarrow \mathrm{S}  x \neq \mathrm{S}  y    \ ]   
\\
\mathsf{Q_2} 
& \forall x \;   [ \   \mathrm{S}  x \neq 0 \  ] 
\\
 \mathsf{Q_3}  
& \forall x  \;   [ \  x=0 \vee \exists y \;  [ \ x = \mathrm{S}  y \  ]     \ ]   
\\
 \mathsf{Q_4}  
&  \forall x  \;   [ \  x+0 = x  \ ]   
\\
 \mathsf{Q_5}  
& \forall x  y \;   [ \  x+ \mathrm{S}  y = \mathrm{S}  ( x+y)   \ ] 
\\
 \mathsf{Q_6} 
&  \forall x  \;   [ \  x \times 0 = 0  \ ]
\\
 \mathsf{Q_7} 
&  \forall x  y  \;  [ \  x \times  \mathrm{S}  y =   x \times y   + x   \ ]  
\end{array}
\]

\caption{
Non-logical axioms of Robinson`s $ \mathsf{Q} $.
}
\label{AxiomsOfQ}
\end{figure}

Our main tool in  proving that    $  \mathsf{Q} + \Theta  $   is  not a  Vaught theory  is the characterization of elementary equivalence in terms of   Ehrenfeucht-Fra\"iss\'e games.  
Let   $  \mathcal{M} $ and  $  \mathcal{N} $ be first-order  structures in a finite  language  $L$. 
For $ 0 < n  <  \omega $,   the   Ehrenfeucht-Fra\"iss\'e game  $  \mathsf{EF}_n  \left[    \mathcal{M}  ,  \mathcal{N}   \right] $  
is a game with $n$ rounds between two  players,   $  \forall $  and  $  \exists $. 
On the $i$-th round, player $  \forall $   picks one of the two structures $  \mathcal{M} $, $  \mathcal{N} $ and   chooses an element from this structure, 
then  player $  \exists $ chooses an element from the other structure. 
The play is the   sequence $  (a_1, b_1 ) ,   (a_2, b_2 ) ,  \ldots ,   (a_n , b_n ) $   which consists of the pairs chosen by the players at each round, 
the $ a_i$'s are elements of  $ \mathcal{M}  $ and the $ b_i$'s are elements of $ \mathcal{N}  $.
Player $  \exists $ wins the game if and only if there there exists  a partial embedding  $  F:    \mathcal{M}  \to   \mathcal{N}  $  
such that  $  F(a_i)  =  b_i $  for all $  i  \in  \lbrace  0,1  ,  \ldots ,  n  \rbrace  $.

Recall that $  \mathcal{M} $ and $   \mathcal{N} $ are elementary equivalent  if and only if  for each  $ 0 < n  <  \omega $, 
player $  \exists $  has a winning strategy for the game
$  \mathsf{EF}_n  \left[    \mathcal{M}  ,  \mathcal{N}   \right] $.
Furthermore,  player  $  \exists  $  has a winning strategy for the game   $  \mathsf{EF}_n  \left[    \mathcal{M}  ,  \mathcal{N}   \right] $
if and only if   $  \mathcal{M}  \equiv_n   \mathcal{N} $, that is, 
 for each first-order   $L$-sentence   $  \phi $, 
if each atomic formula that occurs in    $  \phi $ is unnested and the quantifier depth of $  \phi $  is  at most  $    n $, 
then $  \mathcal{M} \models  \phi $  if and only if    $  \mathcal{N} \models  \phi $. 
An atomic formula is unnested it is of the form: 
(i)  $  x = c   $  where $x$  is a variable  symbol and $c$ is a constant symbol; 
(ii) $  f (  \vec{x} )  =  y  $  where  $  \vec{x} , y $  are variable symbols  and $f$ is a function symbol; 
(iii) $  R (  \vec{x} )  $  where  $  \vec{x} $  are variable symbols and  $R$ is a relation symbol. 
 The quantifier depth of  $  \phi $ is defined by recursion as follows: 
 \[
 \mathsf{qd} \left(  \phi  \right)   =   \begin{cases}
 0              &   \mbox{  if  }    \phi   \mbox{ is an atomic formula; }
 \\
  \mathsf{qd} \left(  \psi  \right)       &   \mbox{  if  }    \phi   \mbox{ is  }  \neg  \psi \,  ;  
   \\
\max \lbrace  \mathsf{qd} \left(  \psi_0   \right)   ,    \mathsf{qd} \left(  \psi_1   \right)  \rbrace      &   \mbox{  if  }    \phi   \mbox{ is  }   \psi_0  \;  \square  \;  \psi_1 \,    \   \mbox{  and  }  \square  \in  \lbrace \vee,  \wedge ,  \rightarrow,  \leftrightarrow  \rbrace   ;  
   \\
1 + \mathsf{qd} \left(  \psi   \right)        &   \mbox{  if  }    \phi   \mbox{ is  }     Q x  \left[  \,   \psi  \,  \right]     
\   \mbox{  and  }     Q  \in  \lbrace  \forall ,  \exists   \rbrace
 \,          .
 \end{cases}
 \]
See, for example,  Section 2.4 of  Marker  \cite{Marker2002} for more details on    Ehrenfeucht-Fra\"iss\'e games.

\section{Main Theorem}

Recall that  $  \mathsf{Q} +  \Theta $  is the theory we obtain by extending Robinson's  $  \mathsf{Q} $  with the axiom 
\[
\Theta :=   \forall xy z w  \left[  \,   (x+y)^2 + x = (z+w)^2 + z \rightarrow x = z  \;  \wedge  \;  y = w  \,  \right] 
\]
which  says that the map  $  \pi  (x, y ) = (x+y)^2 + x $ is one-to-one. 
We prove the following theorem.

\begin{theorem}  \label{QPLUSCANTORPAIRING}
$  \mathsf{Q} +  \Theta  $ has a  model $  \mathcal{M} $  
such that for each $ n \in  \omega $,   there exists   a    finite set  $  A_n  $  of elements of  $  \mathcal{M} $
   such that for each  element  $w $  of   $  \mathcal{M} $  there exist  $  a  \in  A_n  $  and  $ b  \not\in A_n  $   such that 
\[
  \left( \mathcal{M}   \, , \,  w  ,    a    \right)  \equiv_n    \left( \mathcal{M}   \, , \,  w  ,    b     \right)    
  \     .
  \]
\end{theorem}

\begin{corollary}    \label{QPLUSCANTORPAIRINGCorollary}
$  \mathsf{Q} +  \Theta  $  is not a Vaught theory.
\end{corollary}
\begin{proof}
Assume for the sake of a contradiction $ \mathsf{Q} +  \Theta  $ directly interpret  $  \mathsf{VS} $.
Let  $  \mathcal{M} $  be the model of  $ \mathsf{Q} +  \Theta  $  given by Theorem  \ref{QPLUSCANTORPAIRING}. 
Since  $ \mathsf{Q} +  \Theta  $ directly interpret  $  \mathsf{VS} $, there exists a formula 
$  \phi (x, y  ,  u_0,  \ldots , u_k ) $  in the language of  $ \mathsf{Q} $, with all free variables displayed, 
 and   elements  $  p_0,  \ldots , p_k $  of   $  \mathcal{M} $  such that   for all  $  n  \in  \omega  $
\[
  \mathcal{M}   \models    
\forall  x_0,  x_1,  \ldots , x_n   \,   \exists   y   \,  \forall  z  \left[   \, 
  \phi ( z , y  ,  p_0,  \ldots , p_k ) \leftrightarrow   \bigvee_{ i = 0 }^{ n }  z = x_i   \,   \right] 
\          . 
\]
In particular, for each $ n  \in  \omega $, we have an $(n+1)$-ary function $  F_n $  on the universe of  $  \mathcal{M} $   such that 
\[
  \mathcal{M}   \models    
\forall  x_0,  x_1,  \ldots , x_n     \,  \forall  z  \left[   \, 
  \phi ( z,   F_n  \left(  \vec{x} \right)    ,  p_0,  \ldots , p_k ) \leftrightarrow   \bigvee_{ i = 0 }^{ n }  z = x_i   \,   \right] 
\          . 
\]

Let   $  \pi (x, y )   $ denote the term  $  (x+y)^2  + x  $.
By recursion, let    $   \pi_0 (x_0,  x_1  )   :=  \pi (x_0 , x_1 )   $   and  
 \[
    \pi_{ n+1}  (x_0,  x_1,  \ldots , x_{ n+2}   )   :=   \pi \left(   \pi_{ n}  (x_0,  x_1,  \ldots , x_{ n +1}    )   ,   x_{ n+2}   \right)    
  \       .
\]    
By $  \Theta $, these terms define one-to-one maps. 
Let    $  z  \in^{  \star }   y   $  be shorthand for 
\[
\exists  y_0 ,  u_0, u_1,  \ldots  ,  u_k  \left[    \    y =   \pi_{ k }   \left(  y_0 , u_0, u_1,  \ldots , u_k  \right)       \;  \wedge   \;  
 \phi ( z , y_0  ,  u_0,  \ldots , u_k )    \    \right]  
 \      .
\]
For each  $  n  \in \omega  $,  we define an $  (n+1) $-ary function  $G_n $  as follows: 
\[
G_n  \left(  \vec{x} \right)  =   \pi_k   \left(   F_n  \left(   \vec{x} \right)   ,   p_0, p_1 ,   \ldots , p_k   \right)
\      .
\]
Then 
\[
  \mathcal{M}   \models    
\forall  x_0  \ldots  x_n     \;   \forall z      
\big[        \               z   \in^{  \star }    G_n  \left(  \vec{x} \right)      \leftrightarrow  \bigvee_{ i = 0 }^{ n }   z  = x_i               \            \big] 
\          .
\]
Thus,    for  all $  n  \in  \omega   $
\[
  \mathcal{M}   \models    
\forall  x_0  \ldots  x_n    \;  \exists  y   \;   \forall z      
\big[        \               z   \in^{  \star }   y       \leftrightarrow  \bigvee_{ i = 0 }^{ n }   z  = x_i               \            \big] 
\          .
\tag{*}
\]

Let  $N$ be the quantifier depth of the formula   $  x  \in^{  \star } y $ when written in a logically equivalent form  where  each atomic formula is unnested.
By Theorem \ref{QPLUSCANTORPAIRING}, 
there exists  a finite set $ A $ of elements of   $  \mathcal{M} $  such that for 
 each    element  $w $  of   $  \mathcal{M} $  there exist 
 $  a  \in  A  $  and  $  b \not\in A  $ such that 
$  \left( \mathcal{M}   \, , \,  w  ,    a    \right)  \equiv_N    \left( \mathcal{M}   \, , \,  w  ,    b   \right)    $. 
By (*),   $   \mathcal{M}  $ has an element  $   w $  such that 
\[
\mathcal{M}     \models     \forall  x   \left[   \   
x   \in^{  \star }   w   \leftrightarrow  \bigvee_{  g \in  A }  x =  g
 \      \right]      
\         .
\tag{**}
\]
We also know that there exist  $   a   \in  A  $  and   $    b    \not\in  A  $   such that 
\[
  \left( \mathcal{M}   \, , \,   w   ,  a     \right)  \equiv_N    \left( \mathcal{M}   \, , \, w   ,   b    \right)    
  \       .
  \]
Hence, 
$   \mathcal{M}    \models   a   \in^{  \star }   w   \leftrightarrow    b \in^{  \star }   w    $,  by the definition of  $  \equiv_N    $,  
which contradicts (**)  since   $  b  \not\in  A  $. 
Thus,  $  \mathsf{VS} $ is not directly interpretable in  $\mathsf{Q} +  \Theta   $. 
\end{proof}

It follows from  $  \mathsf{Q} + \Theta $  not being a Vaught theory that the tree theory  $  \mathsf{T} $  of    Kristiansen \& Murwanashyaka   \cite{cie20}  is also not a Vaught theory 
since it can be checked that  there exists an extension by definitions of  $  \mathsf{Q} + \Theta $   that contains    $  \mathsf{T} $
(see the following corollary). 
The language of  $  \mathsf{T} $  
 consists of a constant symbol $ \perp$, a binary function symbol  $ \langle \cdot , \cdot \rangle $
and a binary relation symbol $ \sqsubseteq $. 
The theory  $  \mathsf{T}  $  is given by the following four non-logical axioms 
\[
\begin{array}{r l  c  c r l  }
\mathsf{ T_{1} } 
& \forall x y \;  [     \  \langle x, y \rangle \neq \perp    \         ]   
\\
\mathsf{ T_{2} } 
& \forall x y z w \;    [  \  \langle x, y \rangle  =   \langle z, w \rangle \rightarrow 
(  \  x=z \wedge y=w  \   )    \     ] 
\\
\mathsf{ T_{3} } 
&  \forall x  \; [    \     x \sqsubseteq \perp       \;    \leftrightarrow    \;      x= \perp  \      ]   
\\
\mathsf{ T_{4} } 
&   \forall x y z  \; [      \       x \sqsubseteq   \langle y , z  \rangle         \leftrightarrow     
(  \  x =  \langle y , z  \rangle  \vee x \sqsubseteq y \vee x \sqsubseteq z  \  )       \       ] 
\end{array}
\] 
The minimal  model of  $  \mathsf{T}  $   is the  term algebra of finite full binary trees  extended with the subtree relation.
Kristiansen \& Murwanashyaka  \cite{cie20}  show that $  \mathsf{T} $ interprets  Robinson's   $  \mathsf{Q} $, 
and   Damnjanovic  \cite{Damnjanovic2021}   has shown  that   $  \mathsf{T}  $  is interpretable in  $  \mathsf{Q} $.

\begin{corollary}
$  \mathsf{T} $  is not a Vaught theory. 
\end{corollary}
 \begin{proof}
 
It  suffices to show that there exists an extension by definitions of  $  \mathsf{Q} + \Theta $  that contains  $  \mathsf{T} $; 
if  $  \mathsf{VS} $ were directly interpretable in  $  \mathsf{T} $, 
 then it would also be directly  interpretable in $    \mathsf{Q} + \Theta $, 
which would contradict   Corollary  \ref{QPLUSCANTORPAIRINGCorollary}.
 We   can use the following translation of the non-logical symbols of  $  \mathsf{T}  $: 
 \begin{enumerate}
\item   We translate  $  \perp  $  as  $0$.

\item     We  translate  $  \langle  \cdot  ,  \cdot   \rangle  $  as  $ G(x,y) =    (x+y)^2 + x +1  $.
By  $  \Theta  $  and the axioms of  $  \mathsf{Q} $,  $ G$ is one-to-one and $0$  is not in the image  of  $G$.

\item By the method of shortening of cuts, pick a definable cut  $  J   \models  \mathsf{PA}^- $
(see, for example,  sections 7-8  of   Murwanashyaka \cite{MurwanashyakaAML2024}).
Since   $  \mathsf{PA}^- $ is a sequential theory (see Jeřábek  \cite{Jerabek2012}), 
we can code  sequences in  $  J$ using some indexing domain  $ N  \subseteq  J  $. 
We translate   $  u  \sqsubseteq  v  $     as follows:  $  u, v  \in J  $  and there exists a  sequence  $w   \in J  $  such  that  
$   \left(  w  \right)_0  =  u  $  and there exists  $  \ell \in  N  $  such that   $   \left(  w   \right)_{  \ell }  =  v  $ 
 and for all $  j  <  \ell  $  there exists  $  g  \in J  $  such that   $   \left(  w   \right)_{  j +1  }  = G \left(  g,   \left(  w  \right)_j  \right)  $  or  
$   \left(  w   \right)_{  j +1  }  = G \left(  \left(  w  \right)_j    ,   g \right)  $.   \qedhere
 \end{enumerate}
 \end{proof}

The rest of the paper is devoted to  proving   Theorem   \ref{QPLUSCANTORPAIRING}.

\section{Tree models of  $  \mathsf{Q}  + \Theta$}

We proceed to prove Theorem  \ref{QPLUSCANTORPAIRING}. 
The idea  is   to construct a non-standard model    $  \mathcal{M}  \models   \mathsf{Q} + \Theta $ 
where   addition  $  +^{  \mathcal{M} }  $ is  expressive but manageable,  and   multiplication  $  \times^{  \mathcal{M} }  $  is somewhat generic  on non-standard elements. 
It will be the case that neither one of  $  +^{  \mathcal{M} }  $  and  $  \times^{  \mathcal{M} }  $   is commutative. 
The model  $  \mathcal{M} $  will be such that   the non-standard elements can be viewed as labeled  full binary trees with respect to     $+^{  \mathcal{M}  }  $. 
More precisely, consider the theory in the language $  \lbrace  \langle  \cdot ,  \cdot   \rangle   \rbrace $  given by the following  axioms: 
\begin{enumerate}
\item    $  \langle  \cdot ,  \cdot   \rangle   $  is a one-to-one  binary function.

\item  Infinitely many axioms which say that   $   \langle \cdot , \cdot \rangle   $  is acyclic: 
for each  term   $ t (  x_0,  \ldots , x_n ) $ which is not  a variable  and which is  such that   the displayed variables are all the variables that occur in    $ t ( \vec{x} ) $, 
we have the axiom 
$  \forall  \vec{x} \left[  \,  t (  x_0,  \ldots , x_n )   \neq  x_i   \,   \right]  $  for each  $  i \in  \lbrace 0, 1,  \ldots , n  \rbrace  $.
\end{enumerate}
The set of all the non-standard elements of  $   \mathcal{M}   $  will  form a model of this theory when we interpret 
the function symbol  $  \langle  \cdot ,  \cdot   \rangle  $  as   $  +^{   \mathcal{M}   }   $. 
The next lemma shows that  $  \Theta  $  holds in  any model of  $   \mathsf{Q}   $  with this special  property of    $  +^{   \mathcal{M}  } $
and the additional property that 
 $  m  +^{  \mathcal{M}  }  w  =  w   +^{  \mathcal{M}  }   m  $ for each standard element $m $  and each non-standard element $w $. 
We write  $  M  $  for universe of  $  \mathcal{M}  $, and we write  $  \mathbb{N} $ for the standard part of   $  \mathcal{M}  $.

\begin{theorem} \label{SectionQTreeModelTheorem}
Assume  $  \mathcal{M}  $ is a non-standard model of  $  \mathsf{Q} $  which is such that:
\begin{enumerate}
\item The restriction   $ +^{  \mathcal{M}  }  :     \left( M  \setminus  \mathbb{N} \right)^2  \to    M  \setminus  \mathbb{N} $
is  one-to-one. 

\item   Let   $ t ( x_0, x_1,  \ldots , x_{  \ell  }  ) $ be a term   in the language  $  \lbrace + \rbrace $  which is not a variable and  the displayed variables are all the variables that occur in     $ t ( \vec{x} ) $. 
For any tuple  $    \vec{w}  $ in $  M  \setminus  \mathbb{N} $,   we have 
$ t^{  \mathcal{M}  } ( w_0, w_1,  \ldots , w_{ \ell } )   \neq  w_i $  for each  $ i  \in  \lbrace 0, 1,  \ldots ,  \ell  \rbrace $.

\item   $  m  +^{  \mathcal{M}  }  w  =  w   +^{  \mathcal{M}  }   m  $   
for  all  $  m \in  \mathbb{N} $  and  all    $ w  \in M \setminus  \mathbb{N}   \,  $. 
\end{enumerate}
Then,  $   \mathcal{M}     \models  \Theta   $.
\end{theorem}
\begin{proof}

We start with an observation.

  \begin{quote} {\bf (Claim)} \;\;\;\;\;\;    
$  \mathcal{M}  \models  \forall x  \left[   \,  \mathrm{S}^k  x  \neq   x    \,   \right]    $  for all $  0 <  k <  \omega   $.
 \end{quote}

Assume for the sake of a contradiction  there exist  $  0 < k  <  \omega $  and  $ g  \in  M  $  such that 
   $     \left( \mathrm{S}^{  \mathcal{M} }  \right)^{k}   g  =  g  $. 
Then, $ g$ is  certainly non-standard. 
By the axioms of  $  \mathsf{Q} $
\[
g  \times^{   \mathcal{M}   }  g  =   g    \times^{   \mathcal{M}   }      \left( \mathrm{S}^{  \mathcal{M} }  \right)^{k}   g    =  
\left(   \ldots 
\left( 
\left(  g  \times^{   \mathcal{M}   }        g   \right)  +^{   \mathcal{M}  }   g    \right)    +^{   \mathcal{M}  }   \ldots 
  \right)    +^{   \mathcal{M}  }    g 
\]
where we have   $  k$  occurrences of  $   +^{   \mathcal{M}  }  $  on the right-hand side. 
But this contradicts clause (2) since  $g$  and  $  g  \times^{   \mathcal{M}   }  g  $  are both  non-standard. 
This proves the claim.

We proceed to show that   $   \mathcal{M}     \models  \Theta   $.
Let     $   \pi   :  M^2  \to  M  $  be defined  as follows
\[
  \pi  (x,y) =  \left(    \,   (x +^{  \mathcal{M} }   y)    \times^{  \mathcal{M}  }  ( x +^{  \mathcal{M}  } y)    \,  \right)   +^{  \mathcal{M}  } x  
  \      .
  \]
  We need to show that this map is one-to-one. 
Assume $   \pi ( g_0, g_1) =  \pi ( h_0, h_1) $.  
We need to show that $ g_0 = h_0 $  and $ g_1 = h_1  $.
We already know that  the restriction  $   \pi   :  \mathbb{N} \times   \mathbb{N}  \to   \mathbb{N} $  is one-to-one. 
We may  thus assume at least  one of  $  g_0, g_1 $ is non-standard; 
and hence that  at least  one of $h_0, h_1 $  is also non-standard   since  $  \mathcal{M}   \models  \mathsf{Q} $.
Let  
\[
 w :=      g_0 +^{  \mathcal{M} }   g_1
 \,  ,   \   \   
  w^2 := w  \times^{  \mathcal{M}  }   w 
   \,  ,   \   \   
  v :=      h_0 +^{  \mathcal{M} }   h_1 
   \,  ,   \   \   
 v^2 := v  \times^{  \mathcal{M}  }   v  
  \      .
\]
Clearly,  $ w, w^2 , v  , v^2 \not\in  \mathbb{N} $. 
Hence, by the axioms of  $  \mathsf{Q} $
\[
w^2  =  \left( w_0   +^{  \mathcal{M} }  w   \right)     +^{  \mathcal{M} }   w 
\     \mbox{ where  }    \   
w_0 :=  w  \times^{  \mathcal{M}  }   \left(  \mathrm{P}^{   \mathcal{M} }  \right)^2  w  
\]
and 
\[
v^2  =  \left( v_0   +^{  \mathcal{M} }  v   \right)     +^{  \mathcal{M} }   v 
\     \mbox{ where  }    \   
v_0 :=  v  \times^{  \mathcal{M}  }   \left(  \mathrm{P}^{   \mathcal{M} }  \right)^2  v  
\      .
\]
We write   $   \mathrm{P}^{   \mathcal{M} } $ for  the predecessor function. 
Without loss of generality, since at least one of $ g_0, g_1 $  is non-standard and  at   least one of $ h_0, h_1 $  is non-standard, 
we  may assume one of the following holds: 
\begin{itemize}
\item[(I)]   $  g_0, h_0  \in   \mathbb{N} $  and    $  g_1, h_1  \not\in   \mathbb{N} $; 

\item[(II)]    $  g_0, h_0  \not\in   \mathbb{N} $;

\item[(III)]    $ g_0  \in  \mathbb{N} $  and  $  g_1,   h_0  \not\in  \mathbb{N}  $. 
\end{itemize}

We consider case (I):  $  g_0, h_0  \in   \mathbb{N} $  and    $  g_1, h_1  \not\in   \mathbb{N}      \,   $. 
By clause (3)  
\[
  w   =  g_0 +^{  \mathcal{M} }   g_1 =   \left(  \mathrm{S}^{   \mathcal{M} }  \right)^{ g_0 }  g_1  
\   \mbox{  and  }    \   
v  =       h_0 +^{  \mathcal{M} }   h_1   =    \left(  \mathrm{S}^{   \mathcal{M} }  \right)^{ h_0 }  h_1  
\        .
\]
By the axioms of  $  \mathsf{Q} $ 
\begin{align*}
 \left( w_0   +^{  \mathcal{M} }  w   \right)     +^{  \mathcal{M} }     \left(    \mathrm{S}^{  \mathcal{M} }  \right)^{ g_0 }   w   
 & =  \left( \left( w_0   +^{  \mathcal{M} }  w   \right)     +^{  \mathcal{M} }   w   \right) +^{  \mathcal{M} }  g_0
 \\
 & =  \pi  ( g_0, g_1)  
\\
& =  \pi  (h_0, h_1) 
\\
&=
\left( \left( v_0   +^{  \mathcal{M} }  v   \right)     +^{  \mathcal{M} }   v   \right) +^{  \mathcal{M} } h_0  
\\
&= 
 \left( v_0   +^{  \mathcal{M} }  v   \right)     +^{  \mathcal{M} }     \left(    \mathrm{S}^{  \mathcal{M} }  \right)^{ h_0 }   v   
\      .
\end{align*}
By clause (1) 
\[
w  =  v     \    \mbox{  and  }     \     \left(    \mathrm{S}^{  \mathcal{M} }  \right)^{ g_0 }   w     =     \left(    \mathrm{S}^{  \mathcal{M} }  \right)^{ h_0 }   v   
\         .
\]
By the claim above, we must have  $  g_0  =  h_0     \,  $. 
Hence 
\[
 \left(  \mathrm{S}^{   \mathcal{M} }  \right)^{ g_0 }  g_1    =   w  =  v  =    \left(  \mathrm{S}^{   \mathcal{M} }  \right)^{ h_0 }  h_1    
 =   \left(  \mathrm{S}^{   \mathcal{M} }  \right)^{ g_0 }  h_1
 \       .  
\]
Since    $  \mathrm{S}^{   \mathcal{M} }   $  is one-to-one,  $  g_1  =  h_1   \,     $.

We consider case (II):   $  g_0, h_0  \not\in  \mathbb{N} $.
By clause (1),  from 
\[
  w^2  +^{  \mathcal{M} }   g_0 =
\pi   ( g_0, g_1)  
= 
\pi   (h_0, h_1) 
=   
    v^2  +^{  \mathcal{M} }   h_0   
  \]
we get that  $  w^2  = v^2  $  and  $  g_0 = h_0  $. 
Again, from  $  w^2  = v^2  $  we get   $  w  =  v $. 
That is 
\[
 g_0 +^{  \mathcal{M} }   g_1   =   w  =  v  =  h_0 +^{  \mathcal{M} }   h_1   =  g_0 +^{  \mathcal{M} }   h_1 
 \          .   \tag{*}
\]
If    $  g_1, h_1  \not\in   \mathbb{N} $,  then by clause (1),  (*) gives  $  g_1  =  h_1  $. 
If  $  g_1, h_1   \in   \mathbb{N} $, then     (*)    gives 
\[
\left( \mathrm{S}^{  \mathcal{M}  }  \right)^{ g_1 }   g_0  =   
 g_0 +^{  \mathcal{M} }   g_1   =    g_0 +^{  \mathcal{M} }   h_1   = 
 \left( \mathrm{S}^{  \mathcal{M}  }  \right)^{ h_1 }   g_0  
\]
which implies   $  g_1  =  h_1  $  by the claim above.
So, assume for the sake of a contradiction one of $g_1, h_1$ is standard while the other is non-standard. 
Without loss of generality, we may assume  $  g_1  \in  \mathbb{N}  $  and     $  h_1   \not\in  \mathbb{N}  $. 
By the axioms of  $  \mathsf{Q} $ and the assumption that $h_1 $ is non-standard,   from (*)   we get  
\begin{align*}
 \left(  \mathrm{S}^{   \mathcal{M} }  \right)^{ g_1 }   g_0   
 &= g_0 +^{  \mathcal{M} }   g_1 
\\
 & =     g_0   +^{  \mathcal{M} }  h_1
 \\
& =  \left(  \mathrm{S}^{   \mathcal{M} }  \right)^{ g_1 }   \left(  g_0    +^{  \mathcal{M} }      \left(  \mathrm{P}^{   \mathcal{M} }  \right)^{ g_1 }   h_1    \right)
 \     .
\end{align*}
Since  $   \mathrm{S}^{   \mathcal{M} }    $  is one-to-one,  this  implies  $ g_0 = g_0    +^{  \mathcal{M} }      \left(  \mathrm{P}^{   \mathcal{M} }  \right)^{ g_1 }   h_1    $, 
which contradicts clause (2). 
Thus, case (II) gives  $g_0 = h_0 $  and $  g_1 = h_1 $.

Finally, we  show that  case (III) leads to a contradiction. 
By assumption,  $ g_0  \in  \mathbb{N} $  and  $   g_1, h_0  \not\in  \mathbb{N}  $. 
We have two cases: (IIIa) $  h_1  \in  \mathbb{N} $; 
(IIIb)    $  h_1    \not\in  \mathbb{N} $.
We consider case (IIIa). 
By clause (3) 
\[
w  =   g_0  +^{  \mathcal{M} }   g_1  =    \left(  \mathrm{S}^{   \mathcal{M} }  \right)^{ g_0 }    g_1  
\    \mbox{  and  }    \   
v  =   h_0  +^{  \mathcal{M} }   h_1  =    \left(  \mathrm{S}^{   \mathcal{M} }  \right)^{ h_1 }    h_0  
\      .
\]
By the axioms of  $  \mathsf{Q} $
\begin{align*}
    \left( w_0   +^{  \mathcal{M} }  w   \right)     +^{  \mathcal{M} }    \left(  \mathrm{S}^{   \mathcal{M} }  \right)^{ g_0 }  w 
&=    \pi  ( g_0, g_1)  
\\
&  =   \pi   (h_0, h_1) 
\\
&=   \left( \left( v_0   +^{  \mathcal{M} }  v   \right)     +^{  \mathcal{M} }   v   \right)    +^{  \mathcal{M} }   h_0  
\      .
\end{align*}
By clause (1),   we get 
 \[
 w  =  v    
 \    \mbox{  and   }    \   
  \left(  \mathrm{S}^{   \mathcal{M} }  \right)^{ g_0 }  w    =   h_0 
  \     .
 \]
 Hence 
\begin{align*}
h_0 
&=    \left(  \mathrm{S}^{   \mathcal{M} }  \right)^{ g_0 }  w  
\\
&  =   \left(  \mathrm{S}^{   \mathcal{M} }  \right)^{ g_0 }  v  
\\
&=   \left(  \mathrm{S}^{   \mathcal{M} }  \right)^{ h_1 + g_0  }    h_0  
\end{align*}
By the claim above, we must then have $ h_1  = g_0 = 0 $. 
But then, since  $  w  = v $  we  get  
\[
w^2  =   w^2  +^{   \mathcal{M} }  g_0  =  \pi   (g_0, g_1)  =  \pi   ( h_0, h_1)  =   v^2    +^{   \mathcal{M} }  h_0 
=   w^2    +^{   \mathcal{M} }  h_0 
\]
which contradicts clause (2)  since  $ h_0$ is non-standard.

Finally, we consider case (IIIb):    $ g_0  \in  \mathbb{N} $  and  $   g_1, h_0 , h_1  \not\in  \mathbb{N}  $. 
We have  
\begin{align*}
    \left( w_0   +^{  \mathcal{M} }  w   \right)     +^{  \mathcal{M} }    \left(  \mathrm{S}^{   \mathcal{M} }  \right)^{ g_0 }  w 
&=    \pi   ( g_0, g_1)  
\\
&  =   \pi   (h_0, h_1) 
\\
&=   \left( \left( v_0   +^{  \mathcal{M} }  v   \right)     +^{  \mathcal{M} }   v   \right)    +^{  \mathcal{M} }   h_0  
\      .
\end{align*}
By clause (1),  this implies  
\[
 w  = v  
 \    \mbox{  and  }  
 \left(  \mathrm{S}^{   \mathcal{M} }  \right)^{ g_0 }  w   =  h_0 
 \     .
\]
Hence 
\begin{align*}
h_0   &= 
 \left(  \mathrm{S}^{   \mathcal{M} }  \right)^{  g_0 }   w  
 \\
  &= 
 \left(  \mathrm{S}^{   \mathcal{M} }  \right)^{  g_0 }   v
 \\
 &=    \left(  \mathrm{S}^{   \mathcal{M} }  \right)^{  g_0 }   \left(  h_0  +^{  \mathcal{M}  }   h_1   \right) 
 \\
 &= 
 h_0      +^{  \mathcal{M}  }     \left( \mathrm{S}^{   \mathcal{M} }  \right)^{  g_0 }    h_1
\end{align*}
which contradicts clause (2). 
This completes the proof. 
\end{proof}

We proceed to construct a non-standard model  $  \mathcal{M} $  of  $  \mathsf{Q} $  that satisfies clauses (1)-(3) of  
Theorem  \ref{SectionQTreeModelTheorem}, and is thus a model of  $  \mathsf{Q} +  \Theta  $,  
and has the additional property that multiplication by non-standard elements is somewhat generic. 
Let  $L$ denote the first-order language that consists of a binary function symbol  $  \langle  \cdot , \cdot   \rangle   $  
and for each  $  (n, m )  \in   \omega   \times  \mathbb{Z} $ a unique constant symbol  $  \mathsf{d}^m_n $.
Let  $  \mathcal{B} $  be an $  \aleph_0$-saturated model of the $L$-theory given by the following axioms: 
\begin{enumerate}

\item  $  \langle   \cdot  ,  \cdot   \rangle  $  is one-to-one. 

\item For each  $  (n, m )  \in   \omega   \times  \mathbb{Z} $, we have the axiom  
$  \forall x  y  \left[  \,   \langle x, y  \rangle   \neq       \mathsf{d}^m_n      \,   \right]    \,   $. 

\item  We have infinitely many axioms which say that   $   \langle \cdot , \cdot \rangle   $  is acyclic: 
for each  term   $ t (  x_0,  \ldots , x_n ) $ which is not  a variable  and which is  such that   the displayed variables are all the variables that occur in    $ t ( \vec{x} ) $, 
we have the axiom 
$  \forall  \vec{x} \left[  \,  t (  x_0,  \ldots , x_n )   \neq  x_i   \,   \right]  $  for each  $  i \in  \lbrace 0, 1,  \ldots , n  \rbrace  $.
\end{enumerate} 
We  write $  \mathbb{B} $   for the universe of     $  \mathcal{B}   $.
We write  $  \mathsf{d}^m_n   $  for     $  \left(   \mathsf{d}^m_n    \right)^{   \mathcal{B} }   $. 
The non-standard part of  $  \mathcal{M} $  will be a  ``nice" subset of    $  \mathbb{B}  $  that  will allow us  to have some control over  how multiplication on $  \mathcal{M} $  is defined.

We proceed to define the model  $  \mathcal{M} $. 
We will ensure  that  $  \mathcal{M}  \models  \mathsf{Q} $ and then invoke  Theorem  \ref{SectionQTreeModelTheorem}   to conclude  that we also have 
$     \mathcal{M}   \models  \Theta  $. 
With this in mind,  we will interpret $+$   as follows: 
 $   +^{  \mathcal{M}  }      $   and  $  \langle  \cdot  ,  \cdot   \rangle^{  \mathcal{B}  }     $  
 will agree on  $    \left(   M \setminus \mathbb{N}   \right)^2     $. 
The next lemma gives us two families  of  one-to-one maps 
\[
 \mathcal{R}_m :   \mathbb{B} \to    \mathbb{B}         \    \       \mbox{ for }    m  \in \mathbb{Z}  
\   \    \mbox{  and  }     \      \   
 \mathcal{A}_n :   \mathbb{B} \to    \mathbb{B}         \    \    \mbox{ for }     n  \in  \omega  \setminus  \lbrace 0   \rbrace  
\]
that will be used to define $  \times^{  \mathcal{M}  }  $: 
\begin{enumerate}
\item   If  $  w  $  is non-standard and  $  n  \in  \mathbb{N} \setminus 0  \rbrace  $, 
we will set  $   w    \times^{  \mathcal{M}  }   n  :=      \mathcal{A}_n   \left(  w  \right)      \,    $.

\item  If  $  w, v  $  are both non-standard,  we will set   
 $   w    \times^{  \mathcal{M}  }   n  :=      \mathcal{R}_{  \chi \left(  v  \right)  }    \left(  w  \right)       $
 where  $  \chi :  M  \setminus  \mathbb{N}   \to   \mathbb{Z} $  is a  map that will be defined  later. 
\end{enumerate}
The non-standard part of  $  \mathcal{M} $, which will also  be denoted  $  \mathbb{T} $, 
  will be the closure in   $   \mathbb{B}   $  
of  $  \big\{  \mathsf{d}^m_n  :   \   (n, m  )  \in  \omega  \times  \mathbb{Z}  \,   \big\}  $  under 
$  \langle  \cdot  ,  \cdot  \rangle^{  \mathcal{B}  }        $
and the maps    $   \mathcal{R}_m      \,   $.

\begin{lemma}
There exist  two sequences 
\[
  \left(  \mathcal{R}_m :    \   m  \in \mathbb{Z}  \,  \right)  
\   \    \mbox{  and  }     \      \   
  \left(  \mathcal{A}_n :    \   n  \in  \omega  \setminus  \lbrace 0  \rbrace  \,  \right)  
  \]
 such that the following holds: 
 \begin{enumerate}
 \item  Each  $  \mathcal{G}  \in  \lbrace  \mathcal{R}_m :    \   m  \in \mathbb{Z}  \,   \rbrace  \cup    \lbrace     \mathcal{A}_n :    \   n  \in  \omega  \setminus  \lbrace 0  \rbrace   \rbrace   $  
 is a one-to-one function  $  \mathcal{G}     :   \mathbb{B} \to  \mathbb{B}       \,     $.

 \item   Any two distinct maps in   $  \lbrace  \mathcal{R}_m :    \   m  \in \mathbb{Z}  \,   \rbrace  \cup    \lbrace     \mathcal{A}_n :    \   n  \in  \omega  \setminus  \lbrace 0 , 1  \rbrace   \rbrace   $    have disjoint images.

 \item     $  \mathcal{R}_m   \left( \cdot  \right)     =  \langle    \mathcal{R}_{ m-1}   \left( \cdot  \right)      \,  ,  \,  \cdot   \rangle^{  \mathcal{B}  }  $  for all  $  m  \in  \mathbb{Z} $.

 \item   $  \mathcal{A}_1  $  is the identity map. 
 
  \item     $  \mathcal{A}_{n}   \left( \cdot  \right)     =  \langle    \mathcal{A}_{ n-1}   \left( \cdot  \right)      \,  ,  \,  \cdot   \rangle^{  \mathcal{B}  }  $  for all  $  n  \in   \omega  \setminus  \lbrace 0 , 1  \rbrace      $. 
 \end{enumerate}
\end{lemma}
\begin{proof}

First, we define the sequence    $  \left(  \mathcal{A}_n :    \   n  \in  \omega  \setminus  \lbrace 0  \rbrace  \,  \right)    $
by recursion:  for  $  n  \in  \omega  \setminus  \lbrace  0  \rbrace  $  and  $  w  \in  \mathbb{B}  $
\[
\mathcal{A}_n  \left(  w  \right)    =  \begin{cases}
w     &     \mbox{  if  }   n  =  1  
\\
\langle   \mathcal{A}_{n-1}  \left(  w  \right)   \,   ,  \,  w   \rangle^{   \mathcal{B} }       &   \mbox{  if   }  n  >  1  
\       .
\end{cases}
\]
We show that each   $  \mathcal{A}_n   $  is one-to-one. 
This is obvious for    $  n = 1  $. 
We consider the case  $  n >  1  $. 
Let  $  w,  v  \in  \mathbb{B}  $  be such that  $  \mathcal{A}_{n}  \left(  w  \right)    =  \mathcal{A}_{n}  \left(  v  \right)    $.
By the definition of    $  \mathcal{A}_{n}   $  
\[
\langle    \mathcal{A}_{n-1}  \left(  w  \right)   \,   ,  \,  w   \rangle^{   \mathcal{B} } 
  =   \mathcal{A}_{n}  \left(  w  \right)    
 =  \mathcal{A}_{n}  \left(  v  \right) 
 = 
 \langle    \mathcal{A}_{n-1}  \left(  v  \right)   \,   ,  \,  v   \rangle^{   \mathcal{B} } 
\    .
\]
Since  $  \langle  \cdot  ,  \cdot  \rangle^{  \mathcal{B} }  $  is one-to-one,  this implies    $  w  =  v  $. 
Thus,    $  \mathcal{A}_n   $  is one-to-one.

Next, we show that  for any two distinct  $  n, m  \in     \omega  \setminus  \lbrace  0, 1   \rbrace  $, 
the maps $    \mathcal{A}_n   $  and  $   \mathcal{A}_m $  have disjoint images. 
Without loss of generality, we may assume  $ 1 <   n <  m  $. 
Assume for the sake of a contradiction $    \mathcal{A}_n   $  and  $   \mathcal{A}_m $  do not have disjoint images. 
Then, there exist  $  w,  v  \in  \mathbb{B}  $  such that  
$     \mathcal{A}_{n}  \left(  w  \right)     =  \mathcal{A}_{m}  \left(  v  \right)     \,   $. 
Since   $  n,m >  1  $, this means that 
\[
\langle    \mathcal{A}_{n-1}  \left(  w  \right)   \,   ,  \,  w   \rangle^{   \mathcal{B} } 
  =   \mathcal{A}_{n}  \left(  w  \right)    
 =  \mathcal{A}_{m}  \left(  v  \right) 
 = 
 \langle    \mathcal{A}_{m-1}  \left(  v  \right)   \,   ,  \,  v   \rangle^{   \mathcal{B} } 
\    .
\]
Since  $  \langle  \cdot  ,  \cdot  \rangle^{  \mathcal{B} }  $  is one-to-one,  this implies    
\[
\mathcal{A}_{n-1}  \left(  w  \right)    =     \mathcal{A}_{m-1}  \left(  v  \right)    
\   \mbox{  and  }     \     w  =  v  
\      .
\]
This means that  
\[
 1  \leq  n-1   <  m-1  
 \   \mbox{  and  }    \       \mathcal{A}_{n-1}  \left(  w  \right)     =  \mathcal{A}_{m-1}  \left(  w  \right)    
 \      .
\]
If  $  n-1 =  1  $,  then  
\[
w  =   \mathcal{A}_{1}  \left(  w  \right)     =  \mathcal{A}_{m-1}  \left(  w  \right)     =
  \langle    \mathcal{A}_{m-2}  \left(  w  \right)    \,  ,  \,  w  \rangle^{  \mathcal{B} }
\]
which contradicts the fact that  $   \langle  \cdot ,  \cdot   \rangle^{  \mathcal{B}  }  $  is acyclic. 
Assume $  n-1  >  0  $.   Then 
\[
  \langle    \mathcal{A}_{n-2}  \left(  w  \right)    \,  ,  \,  w  \rangle^{  \mathcal{B} }
  =  
 \mathcal{A}_{n-1}  \left(  w  \right)     =  \mathcal{A}_{m-1}  \left(  w  \right)    
 = 
   \langle    \mathcal{A}_{m-2}  \left(  w  \right)    \,  ,  \,  w  \rangle^{  \mathcal{B} }
\       .
\]
Since    $  \langle  \cdot  ,  \cdot  \rangle^{  \mathcal{B} }  $  is one-to-one,  this implies    
\[
 1  \leq  n-2   <  m-2  
 \   \mbox{  and  }    \       \mathcal{A}_{n-2}  \left(  w  \right)     =  \mathcal{A}_{m-2}  \left(  w  \right)    
 \      .
\]
By continuing this way, we get that 
\[
1  <  m -n+1   
 \   \mbox{  and  }    \       \mathcal{A}_{1}  \left(  w  \right)     =  \mathcal{A}_{m-n+1}  \left(  w  \right)    
 \      .
\]
Hence
\[
w  =   \mathcal{A}_{1}  \left(  w  \right)        =   \mathcal{A}_{m-n+1}  \left(  w  \right)     = 
    \langle    \mathcal{A}_{m-n}  \left(  w  \right)    \,  ,  \,  w  \rangle^{  \mathcal{B} } 
\]
which contradicts the fact that  $   \langle  \cdot ,  \cdot   \rangle^{  \mathcal{B}  }  $  is acyclic. 
Thus, any two distinct maps in  $  \lbrace  \mathcal{A}_{n}  :   \   n  \in  \omega  \setminus  \lbrace 0, 1  \rbrace    \  \rbrace $  
have disjoint images.

Next,  we define the sequence   $  \left(  \mathcal{R}_m :    \   m  \in \mathbb{Z}  \,  \right)    \,  $.
We use the fact that  $  \mathcal{B}  $  is $  \aleph_0$-saturated to define  $  \mathcal{R}_0 $  
and then define the other maps  using    $  \mathcal{R}_0 $. 
For each $g  \in   \mathbb{B}  $,  we define a sequence   of formulas  in the expanded language where we have a constant symbol for $g$: 
 \[
   \Psi_0^g  (x ,  v  )   :=    v  =    \langle   x ,  g   \rangle   
   \      \mbox{  and   }      \   
   \Psi_{  n+1  }^g  (x ,  v  )   :=     \exists y   \left[   \,    \Psi_n^g  ( y ,  v )   \;  \wedge   \;     y   =    \langle   x  ,  g   \rangle     \,   \right]   
   \         .
\]
Clearly, for each  $n \in  \omega  $,   there exists  $ w_n   \in   \mathbb{B}   $    such that  
$  \mathcal{B}   \models \bigwedge_{ j =  0 }^{ n }    \exists  x  \left[  \,  \Psi_j^g  ( x,  w_n )  \,   \right]   \,   $: 
by recursion,  set  $  w_0 :=  \langle  g,  g  \rangle^{  \mathcal{B}  }  $  and  
$  w_{ n+1}  :=  \langle  w_n ,  g  \rangle^{  \mathcal{B}  }         \,   $.
Let     $  W_g $  denote the set of all elements of  $  \mathcal{B}  $  that realize  the $1$-type 
 $  \lbrace    \exists  x   \left[  \,   \Psi_n^g  (x, v )   \,   \right]      :   \   n \in  \omega   \,   \rbrace  $    over  $  \lbrace  g  \rbrace  $. 
Since     $  \mathcal{B}  $   is  $  \aleph_0 $-saturated,   $  W_g  \neq   \emptyset    $. 
 Observe that  for each   $  w  \in  W_g $,   we  also  have  $  w^{  \prime }  \in  W_g  $  such that  $  w  =  \langle  w^{  \prime } ,  g  \rangle   \,   $. 
 Thus, since  $  \langle  \cdot  ,  \cdot  \rangle^{  \mathcal{B} }  $  is acyclic,   $  W_g  $  is infinite; 
 this will be used to define  $  \mathcal{R}_m  $  for $  m \in  \mathbb{Z}_{ < 0  }     \,   $.
 Since we have a family $  \lbrace   W_g  :   \   g  \in  \mathbb{B}   \,  \rbrace  $  of non-empty sets, 
 the axiom of choice gives us  a map   $  \mathcal{R}_0     :        \mathbb{B}   \to   \mathbb{B}   $  
such that   $  \mathcal{R}_0  \left(    g  \right)      \in  W_g  $  for all $  g  \in    \mathbb{B}   $.

Let us check that     $  \mathcal{R}_0  $  is one-to-one. 
Assume we have  $  w, v  \in  \mathbb{B}  $  such that   $    \mathcal{R}_0    w  =  \mathcal{R}_0  v  $. 
We need to show that  $  w  = v  $. 
Since   for each   $  w  \in  W_g $,   we  also  have  $  w^{  \prime }   \in  W_g  $  such that  $  w  =  \langle  w^{  \prime } ,  g  \rangle     $, 
  there exist  $  g_0,  g_1  \in  \mathbb{B}  $  such that  
\[
\langle  g_0  ,  w  \rangle^{  \mathcal{B}  }   
=    \mathcal{R}_0    w  =  \mathcal{R}_0  v   =  
\langle  g_1  ,  v  \rangle^{  \mathcal{B}  }        
 \      .
\]
Since  $  \langle  \cdot   ,  \cdot     \rangle^{  \mathcal{B}  }  $  is one-to-one, 
this implies  $    w  =  v  $. 
Thus,   $     \mathcal{R}_0   $ is one-to-one.

We define  $  \mathcal{R}_m $  for  $  m  \in  \mathbb{Z} \setminus  \lbrace 0  \rbrace  $  by recursion: 
\begin{enumerate}
\item  Assume first  $  m > 0  $.   We define   $  \mathcal{R}_m :    \mathbb{B} \to  \mathbb{B}   $  using 
$  \mathcal{R}_{m-1} :    \mathbb{B} \to  \mathbb{B}   $
\[
  \mathcal{R}_m   \left(  w  \right)     =    \langle      \mathcal{R}_{m-1}   \left(  w  \right)     \,  ,   \,   w    \rangle^{  \mathcal{B}  }  
  \     \mbox{  for all }    w  \in  \mathbb{B}
  \         .
\]

\item  Next, assume  $  m  <  0  $.   We define   $  \mathcal{R}_m :    \mathbb{B} \to  \mathbb{B}   $  using 
$  \mathcal{R}_{m+1} :    \mathbb{B} \to  \mathbb{B}   $. 
We  assume     $  \mathcal{R}_{ m+1}   \left(    g  \right)      \in  W_g  $  for all $  g  \in    \mathbb{B}   $. 
Let  $  g   \in  \mathbb{B} $. 
We proceed to define    $   \mathcal{R}_{ m}   \left(    g  \right)    $. 
Since      $  \mathcal{R}_{ m+1}   \left(    g  \right)      \in  W_g  $,  there exists  $  w_g  \in  W_g  $  such that  
$  \mathcal{R}_{ m+1}   \left(    g  \right)     =  \langle  w_g  ,  g  \rangle^{  \mathcal{B}  }  $; 
this  $  w_g  $  is unique   since    $  \langle  \cdot  ,  \cdot   \rangle^{  \mathcal{B}  }  $  is one-to-one. 
We set   $      \mathcal{R}_{ m}   \left(    g  \right)    :=  w_g     \,    $. 
\end{enumerate}
The  maps  $  \mathcal{R}_m $   are one-to-one since  $  \langle  \cdot  ,  \cdot   \rangle^{  \mathcal{B}  }  $  is one-to-one. 
Furthermore,  they have disjoint images  since     $  \langle  \cdot  ,  \cdot   \rangle^{  \mathcal{B}  }  $  is acyclic
and     $  \mathcal{R}_m   \left( \cdot  \right)     =  \langle    \mathcal{R}_{ m-1}   \left( \cdot  \right)      \,  ,  \,  \cdot   \rangle^{  \mathcal{B}  }  $  for all  $  m  \in  \mathbb{Z} $. 
The arguments are similar to the arguments above for the maps  $  \mathcal{A}_n $.

Let  $  m  \in  \mathbb{Z}  $  and let  $  n  \in  \omega  \setminus  \lbrace 0, 1  \rbrace  $. 
It remains to show that    $  \mathcal{R}_m   $  and  $   \mathcal{A}_n  $  have disjoint images. 
Assume for the sake of a contradiction $    \mathcal{R}_m   $  and  $   \mathcal{A}_m $  do not have disjoint images. 
Then, there exist  $  w,  v  \in  \mathbb{B}  $  such that  
$     \mathcal{R}_{m}  \left(  w  \right)     =  \mathcal{A}_{n}  \left(  v  \right)     \,   $. 
This means that 
\[
\langle    \mathcal{R}_{m-1}  \left(  w  \right)   \,   ,  \,  w   \rangle^{   \mathcal{B} } 
  =   \mathcal{R}_{m}  \left(  w  \right)    
 =  \mathcal{A}_{n}  \left(  v  \right) 
 = 
 \langle    \mathcal{A}_{n-1}  \left(  v  \right)   \,   ,  \,  v   \rangle^{   \mathcal{B} } 
\    .
\]
Since  $  \langle  \cdot  ,  \cdot  \rangle^{  \mathcal{B} }  $  is one-to-one,  this implies    
\[
\mathcal{R}_{m-1}  \left(  w  \right)    =     \mathcal{A}_{n-1}  \left(  v  \right)    
\   \mbox{  and  }     \     w  =  v  
\      .
\]
If  $  n-1 = 1  $, then  
\[
\langle    \mathcal{R}_{m-2}  \left(  w  \right)    \,   ,  \,  w   \rangle^{   \mathcal{B} }     
= 
\mathcal{R}_{m-1}  \left(  w  \right)    =     \mathcal{A}_{n-1}  \left(  w  \right)      =  w  
\]
which contradicts the fact that  $   \langle  \cdot ,  \cdot   \rangle^{  \mathcal{B}  }  $  is acyclic. 
Assume  $  n-1 >  1  $. 
Then
\[
\langle    \mathcal{R}_{m-2}  \left(  w  \right)    \,   ,  \,  w   \rangle^{   \mathcal{B} }     
= 
\mathcal{R}_{m-1}  \left(  w  \right)    =     \mathcal{A}_{n-1}  \left(  w  \right)      =  
\langle    \mathcal{A}_{n-2}  \left(  w  \right)    \,   ,  \,  w   \rangle^{   \mathcal{B} }     
\      .
\]
Since    $   \langle  \cdot ,  \cdot   \rangle^{  \mathcal{B}  }  $ is one-to-one,  this implies  
$     \mathcal{R}_{m-2}  \left(  w  \right)     =     \mathcal{A}_{n-2}  \left(  w  \right)        \,   $.
By continuing this way, we get 
\[
\langle    \mathcal{R}_{m-n}  \left(  w  \right)    \,   ,  \,  w   \rangle^{   \mathcal{B} }     
= 
\mathcal{R}_{m-n+1}  \left(  w  \right)    =     \mathcal{A}_{1}  \left(  w  \right)      =  w  
\]
which contradicts the fact that  $   \langle  \cdot ,  \cdot   \rangle^{  \mathcal{B}  }  $  is acyclic. 
Thus,  $  \mathcal{R}_m  $  and  $  \mathcal{A}_n $  have disjoint images. 
\end{proof}

We are now ready to define the set  $  \mathbb{T} $ of non-standard elements of    $  \mathcal{M} $.
For $  n  \in \omega $, let  $  \mathbb{T}_0 :=   \big\{  \mathsf{d}^m_n :   \   (n,m)  \in  \omega  \times  \mathbb{Z} \,  \big\}  $
and 
\[
 \mathbb{T}_{n+1}  :=   \mathbb{T}_n  
 \cup 
  \lbrace \langle w, v  \rangle^{  \mathcal{B} }  :   \    w,  v  \in   \mathbb{T}_n  \,  \rbrace
 \cup  
\lbrace    \mathcal{R}_m w  :   \    w  \in   \mathbb{T}_n  \,   ,   \    m  \in  \mathbb{Z}  \,    \rbrace
  \      .
\]
Let  
\[
\mathbb{T} :=  \bigcup_{  n  \in  \omega  }   \mathbb{T}_n 
\      .
\]
Since   $  \mathbb{T}  $  is inductively  defined,  we will be able to define operations on   $ \mathbb{T}   $ by recursion 
with respect to the following rank function 
\[
\mathsf{rk} \left(  w  \right)  =  \min  \lbrace  n   \in  \omega   :   \    w  \in   \mathbb{T}_n    \,    \rbrace 
\    .
\]
Observe that by clause (3) of the preceding lemma,    
for all $  w  \in   \mathbb{T}  $,  either   $  w  \in  \big\{  \mathsf{d}^m_n :   \   (n,m)  \in  \omega  \times  \mathbb{Z} \,  \big\}  $  
or there exist  $  w_0,  w_1  \in  \mathbb{T}  $  such that   $  w  =  \langle  w_0,  w_1  \rangle^{  \mathcal{B}  }  $. 
Clearly,   $   \mathsf{rk} \left(   \langle  w_0 ,  w_1  \rangle^{  \mathcal{B} }  \right)  >  \mathsf{rk} \left(  w_1  \right)    $, 
but we may  have  $   \mathsf{rk} \left(   \langle  w_0 ,  w_1  \rangle^{  \mathcal{B} }  \right)  =  \mathsf{rk} \left(  w_0  \right)     \,     $.

We are ready to define the universe   $M$ of   $  \mathcal{M} $.
We assume  $  \mathbb{N} :=   \omega  $  and    $   \mathbb{T}   $  are disjoint. 
Let   $  M  :=  \mathbb{N}  \cup   \mathbb{T}  $.

We proceed to define the arithmetical operations on  $M$. 
We start by defining  the successor function on  $M$  by recursion: 
\[
\mathrm{S}^{  \mathcal{M}  }   g   =   \begin{cases}
g +1         &   \mbox{  if  }     g    \in  \mathbb{N}
\\
\mathsf{d}^{ m+ 1}_n        &   \mbox{  if  }     g  =  \mathsf{d}^m_n 
\\
\langle  g_0  \,  ,   \,    \mathrm{S}^{  \mathcal{M}  }   g_1   \rangle^{  \mathcal{B}   }         
                              &   \mbox{  if  }   g  \in   \mathbb{T}       \mbox{  and  }     g  =  \langle  g_0,  g_1  \rangle^{  \mathcal{B}    }
\         .
\end{cases}
\]
The next lemma shows that  axioms $  \mathsf{Q}_1 $,  $  \mathsf{Q}_2 $,  $  \mathsf{Q}_3  $   will hold in the model $  \mathcal{M} $  we are defining.

\begin{lemma}
$  \mathrm{S}^{  \mathcal{M}  }    :  M  \to   M  \setminus  \lbrace 0  \rbrace  $  is a one-to-one map. 
Furthermore,  there exists a map $  \mathrm{P}^{  \mathcal{M}  }    :  M  \to   M   $  such that  
$   \mathrm{S}^{  \mathcal{M}  }     \mathrm{P}^{  \mathcal{M}  }  g = g  $  for all $  g \in  M \setminus  \lbrace 0 \rbrace  $.
\end{lemma}
\begin{proof}

We have  $  \mathrm{S}^{  \mathcal{M}  }  g  \neq 0 $  for all $  g \in  M  $ since: 
(i)   $  \mathrm{S}^{  \mathcal{M}  }   m   \in  \mathbb{N}_{ >0 }  $  for all $ m \in \mathbb{N} $; 
(ii)  $  \mathrm{S}^{  \mathcal{M}  }   g  \in  \mathbb{T}  $  for all $  g \in  \mathbb{T} $. 
 
We verify that     $  \mathrm{S}^{  \mathcal{M}  }   $  is one-to-one. 
It suffices to show that  $  \mathrm{S}^{  \mathcal{M}  }     \upharpoonright  \mathbb{T} $  is one-to-one. 
Pick  $  g  \in  \mathbb{T} $. 
We prove  by induction on the complexity of  $  g $ that for all $  h  \in  \mathbb{T} $,  
we have  $  g  = h  $ whenever $   \mathrm{S}^{  \mathcal{M}  }  g   =  \mathrm{S}^{  \mathcal{M}  }  h  $. 
For the base case, assume  $  g  =  \mathsf{d}^m_n $  where  $  (n, m )  \in  \omega  \times   \mathbb{Z} $. 
Assume $    \mathrm{S}^{  \mathcal{M}  }  g   =  \mathrm{S}^{  \mathcal{M}  }  h  $. 
Then 
\[
\mathsf{d}^{m+1}_n  =   \mathrm{S}^{  \mathcal{M}  }  g   =  \mathrm{S}^{  \mathcal{M}  }  h
\]
and from this  it follows that we must have  $  h  =  \mathsf{d}^m_n  $.

For the inductive case,   assume $  g  =  \langle   w,  v  \rangle^{  \mathcal{B} }  $  where  $  w, v  \in  \mathbb{T} $. 
Assume    $   \mathrm{S}^{  \mathcal{M}  }  g   =  \mathrm{S}^{  \mathcal{M}  }  h  $. 
Then 
\[
\langle   w  \, ,  \,     \mathrm{S}^{  \mathcal{M}  }   v  \rangle^{  \mathcal{B} }  
=     \mathrm{S}^{  \mathcal{M}  }  g   =  \mathrm{S}^{  \mathcal{M}  }  h   
\     .
\]
It  follows that  we cannot have  $  h \in  \big\{  \mathsf{d}^m_n :   \  (n,m)  \in  \omega  \times  \mathbb{Z}  \,   \big\}  $. 
There thus exist  $  h_0, h_1  \in  \mathbb{T}  $  such that  
$  h  =  \langle  h_0,  h_1  \rangle^{  \mathcal{B} }  $. 
Hence 
\[
\langle   w  \, ,  \,     \mathrm{S}^{  \mathcal{M}  }   v  \rangle^{  \mathcal{B} }  
= \mathrm{S}^{  \mathcal{M}  }  h   
= 
\langle   h_0   \, ,  \,     \mathrm{S}^{  \mathcal{M}  }   h_1  \rangle^{  \mathcal{B} }  
\       .
\]
Since  $  \langle  \cdot  ,  \cdot    \rangle^{  \mathcal{B} }    $  is one-to-one,  
$  w  = h_0  $  and   $  \mathrm{S}^{  \mathcal{M}  }   v  =    \mathrm{S}^{  \mathcal{M}  }   h_1  $. 
By the induction hypothesis, the latter gives  $  v = h_1  $. 
Thus,  $  g  =   \langle  w,  v   \rangle^{  \mathcal{B} }  =      \langle  h_0,  h_1  \rangle^{  \mathcal{B} } =  h  $.

Thus, by induction,  $   \mathrm{S}^{  \mathcal{M}  }   $  is one-to-one.

Finally, we define the predecessor function by recursion  
\[
\mathrm{P}^{  \mathcal{M}  }   g   =   \begin{cases}
0        &    \mbox{  if  } g = 0 
\\
g -1         &   \mbox{  if  }     g    \in  \mathbb{N}_{ > 0 } 
\\
\mathsf{d}^{ m- 1}_n        &   \mbox{  if  }     g  =  \mathsf{d}^m_n 
\\
\langle  g_0  \,  ,   \,    \mathrm{P}^{  \mathcal{M}  }   g_1   \rangle^{  \mathcal{B}    }         
                              &   \mbox{  if  }    g  \in  \mathbb{T}      \mbox{  and  }     g  =  \langle  g_0,  g_1  \rangle^{  \mathcal{B}   }
\         .
\end{cases}
\]
This completes the proof of the lemma.
\end{proof}

We define $  +^{  \mathcal{M} }  $   as follows 
\[
g_0 +^{  \mathcal{M}  }   g_1   =   \begin{cases}
g_0 +  g_1          &     \mbox{  if  }    g_0,  g_1   \in  \mathbb{N} 
\\
\langle  g_0,  g_1  \rangle^{  \mathcal{B}  }      &     \mbox{  if   }  g_0,  g_1   \in   \mathbb{T} 
\\
  \left(    \mathrm{S}^{  \mathcal{M}  }  \right)^{ g_0 }   g_1                     
               & \mbox{  if  }     g_0  \in  \mathbb{N}       \    \mbox{  and    }    \   g_1  \in  \mathbb{T}    
\\
  \left(    \mathrm{S}^{  \mathcal{M}  }  \right)^{ g_1 }   g_0          
                     & \mbox{  if  }     g_1  \in  \mathbb{N}      \    \mbox{  and    }    \          g_0  \in  \mathbb{T}    
\         .
\end{cases}
\]
The next lemma shows that   $  \mathsf{Q}_4,  \mathsf{Q}_5  $  will hold in   the model $  \mathcal{M} $ we are defining. 
Furthermore, by Theorem  \ref{SectionQTreeModelTheorem}, we will also have  $  \mathcal{M}  \models  \Theta   $.

\begin{lemma}
The following holds:
\begin{enumerate}
\item The restriction   $ +^{  \mathcal{M}  }  :     \left( M  \setminus  \mathbb{N} \right)  \times    \left( M  \setminus  \mathbb{N} \right)   \to    M  \setminus  \mathbb{N} $
is  one-to-one. 

\item   Let   $ t ( \vec{x} ) $ be a term   in the language  $  \lbrace + \rbrace $  which is not a variable and  the displayed variables are all the variables that occur in     $ t ( \vec{x} ) $. 
For any tuple  $    \vec{w}  $ in $  M  \setminus  \mathbb{N} $,   we have 
$ t^{  \mathcal{M}  } ( w_0, w_1,  \ldots , w_{ \ell } )   \neq  w_i $  for each  $ i  \in  \lbrace 0, 1,  \ldots ,  \ell  \rbrace $.

\item   $  m  +^{  \mathcal{M}  }  g  =  g   +^{  \mathcal{M}  }   m  $   for  all   $  m \in  \mathbb{N} $  and  all   $ g  \in M $.

\item  $  g +^{  \mathcal{M}  }   0  =  g  $  for all $  g \in  M  $. 

\item   $ g_0   +^{  \mathcal{M}  }    \mathrm{S}^{ \mathcal{M}  }  g_1  =   \mathrm{S}^{ \mathcal{M}  }   \left(    g_0   +^{  \mathcal{M}  }    g_1  \right)  $ 
for all $  g_0, g_1  \in M $.
\end{enumerate}
\end{lemma}
\begin{proof}
Clause (1) holds since $  \langle \cdot  , \cdot  \rangle^{   \mathcal{B}  }  $  is one-to-one. 
Clause (2) holds since  $  \langle \cdot  , \cdot  \rangle^{   \mathcal{B}  }  $ is acyclic. 
Clauses (3) and (4) hold by how  $  +^{  \mathcal{M}  }   $ is defined. 
It remains to verify that (5) holds. 
We have 
\begin{align*}
g_0   +^{  \mathcal{M} }    \mathrm{S}^{  \mathcal{M}  }  g_1 
&=   
\begin{cases}
\left(    \mathrm{S}^{  \mathcal{M}  }  \right)^{ g_1 +1  }   g_0         &    \mbox{ if  }  g_1  \in  \mathbb{N}
\vspace*{0.08cm}
\\
\left(    \mathrm{S}^{  \mathcal{M}  }  \right)^{ g_0   }    \left(    \mathrm{S}^{  \mathcal{M}  }   g_1    \right) 
                             &    \mbox{ if  }  g_0  \in  \mathbb{N}   
\vspace*{0.08cm}
 \\
\langle  g_0   \,  ,   \,    \mathrm{S}^{  \mathcal{M} } g_1     \rangle^{  \mathcal{B}    }       &    \mbox{ if  }     g_0,  g_1  \in  \mathbb{T} 
\end{cases}
\\
\\
&=   
\begin{cases}
 \mathrm{S}^{  \mathcal{M}  }   \left(  \left(    \mathrm{S}^{  \mathcal{M}  }  \right)^{ g_1  }   g_0     \right)       &    \mbox{ if  }  g_1  \in  \mathbb{N}
\vspace*{0.08cm}
\\
 \mathrm{S}^{  \mathcal{M}  }   \left(  \left(    \mathrm{S}^{  \mathcal{M}  }  \right)^{ g_0  }   g_1     \right)        
              &    \mbox{ if  }  g_0  \in  \mathbb{N}  
\vspace*{0.08cm}
 \\
 \mathrm{S}^{  \mathcal{M} }  \left(  \langle  g_0   \,  ,   \,    g_1     \rangle^{  \mathcal{B}   }   \right)      &    \mbox{ if  }     g_0,  g_1  \in  \mathbb{T} 
\end{cases}
\\
\\
&=   
 \mathrm{S}^{  \mathcal{M} }  \left(     g_0   +^{  \mathcal{M} }    g_1   \right) 
\end{align*}
which shows that  (5) holds. 
\end{proof}

All that remains it to define   $  \times^{  \mathcal{M}  } $. 
First,   we define the map  $   \chi :  \mathbb{T}  \to   \mathbb{Z}  $  by recursion 
\[
\chi  \left( g  \right)    \begin{cases}
m                 &      \mbox{  if   }  g =  \mathsf{d}^m_n     
\\
\chi  \left( g_1  \right)     
&    \mbox{  if   }    \left(  \exists  g_0  \in  \mathbb{T}  \right)  \left[   \,   g =  \langle g_0, g_1  \rangle^{  \mathcal{B}   }   \,    \right]   
\      .
\end{cases}
\]
Observe that     $  \chi  \left(  \mathrm{S}^{  \mathcal{M}  } g  \right)       =    \chi  \left( g  \right)      +1  $
and   $  \chi  \left(  \mathrm{P}^{  \mathcal{M}  } g  \right)       =    \chi  \left( g  \right)      - 1  $
  for all $  g  \in  \mathbb{T}  $.

We define    $   \times^{  \mathcal{M}  }   $  as follows 
\[
g_0 \times^{  \mathcal{M}  }   g_1   =   \begin{cases}
g_0 \times  g_1       &    \mbox{  if  }   g_0, g_1  \in  \mathbb{N} 
\\
\mathsf{d}_0^{   g_0   \times  \chi \left(  g_1  \right)  }          &    \mbox{  if  }   g_0  \in  \mathbb{N}   \mbox{  and  }    g_1  \in  \mathbb{T}
\\ 
0             &  \mbox{  if  }    g_0  \in  \mathbb{T}    \mbox{  and   }   g_1  =  0 
\\
\mathcal{A}_{ g_1  }  \left(  g_0  \right)    &    \mbox{  if  }    g_0  \in  \mathbb{T}    \mbox{  and   }    g_1    \in    \mathbb{N}_{ >0 }
 \\
 \mathcal{R}_{   \chi  \left( g_1  \right)  }   \left(   g_0    \right)                  & \mbox{  if  }     g_0, g_1  \in  \mathbb{T}
 \         .
\end{cases}
\]

\begin{lemma}
The following holds:
\begin{enumerate}
\item  $  g \times^{  \mathcal{M}  }   0  =  0  $  for all $  g \in  M  $. 

\item   $ g_0   \times^{  \mathcal{M}  }    \mathrm{S}^{ \mathcal{M}  }  g_1  =   \left( g_0 \times^{  \mathcal{M}  }   g_1  \right)   +^{  \mathcal{M}  }    g_0  $ 
for all $  g_0, g_1  \in M $.
\end{enumerate}
\end{lemma}
\begin{proof}

Clause (1) holds  by how   $  \times^{  \mathcal{M}  }    $  is defined. 
We check that (2) holds. 
We have two cases: 
(2a)  $  g_1  \in  \mathbb{N} $; 
(2b) $  g_1  \in  \mathbb{T} $.
We consider (2a). 
We have 
\begin{align*}
 g_0  \times^{  \mathcal{M}  }    \mathrm{S}^{ \mathcal{M}  }  g_1    
 &  =  \begin{cases}
 g_0  \times  (g_1+1)          &   \mbox{  if  }    g_0   \in  \mathbb{N}
 \\
 \mathcal{A}_{ g_1+1  }  \left(  g_0   \right)      &  \mbox{  if  }   g_0   \in  \mathbb{T}
 \end{cases}
 \\
 &  =  \begin{cases}
 \left(  g_0  \times  g_1   \right) + g_0           &   \mbox{  if  }    g_0   \in  \mathbb{N}
 \\
  \mathcal{A}_{ 1  }  \left(  g_0   \right)   = 0 +^{  \mathcal{M}  }  g_0           
                            &   \mbox{  if  }    g_0   \in  \mathbb{T}       \mbox{  and  }   g_1  =  0 
 \\
 \langle     \mathcal{A}_{ g_1  }  \left(  g_0   \right)   \,  ,   \,   g_0     \rangle^{   \mathcal{B}  }      
      &  \mbox{  if  }   g_0   \in  \mathbb{T}   \mbox{  and  }   g_1  \in  \mathbb{N}_{ > 0 }
 \end{cases} 
 \\
 & =   \left( g_0 \times^{  \mathcal{M}  }   g_1  \right)   +^{  \mathcal{M}  }    g_0 
 \     .
\end{align*}

Finally,  we consider (2b). 
We have 
\begin{align*}
 g_0  \times^{  \mathcal{M}  }    \mathrm{S}^{ \mathcal{M}  }  g_1    
 &  =  \begin{cases}
\mathsf{d}_0^{   g_0   \times  \chi \left(   \mathrm{S}^{ \mathcal{M}  } g_1  \right)  }         &   \mbox{  if  }    g_0   \in  \mathbb{N}
 \\
 \mathcal{R}_{ \chi \left(   \mathrm{S}^{ \mathcal{M}  } g_1  \right)   }  \left(  g_0   \right)      &  \mbox{  if  }   g_0   \in  \mathbb{T}
 \end{cases}
 \\
 &  =  \begin{cases}
\mathsf{d}_0^{   g_0   \times  \left(  \chi \left(  g_1  \right)  + 1  \right)   }         &   \mbox{  if  }    g_0   \in  \mathbb{N}
 \\
 \mathcal{R}_{ \chi \left(    g_1  \right)  +1  }  \left(  g_0   \right)      &  \mbox{  if  }   g_0   \in  \mathbb{T}
 \end{cases}
  \\
 &  =  \begin{cases}
 \left(   \mathrm{S}^{ \mathcal{M}  }  \right)^{ g_0 }   \left(  \mathsf{d}_0^{   g_0   \times    \chi \left(  g_1  \right)    }   \right)    
                 &   \mbox{  if  }    g_0   \in  \mathbb{N}
 \\
 \langle    \mathcal{R}_{ \chi \left(    g_1  \right)   }  \left(  g_0   \right)    \,  ,  \,  g_0   \rangle^{  \mathcal{B}  }
                     &  \mbox{  if  }   g_0   \in  \mathbb{T}
 \end{cases}
 \\
 & =   \left( g_0 \times^{  \mathcal{M}  }   g_1  \right)   +^{  \mathcal{M}  }    g_0 
 \     .    \qedhere
\end{align*}
\end{proof}

Theorem  \ref{SectionQTreeModelTheorem}  and  the lemmas we have proved show that the structure $  \mathcal{M} $ we have defined is a model of   $  \mathsf{Q} +  \Theta     $.

\begin{theorem}
$  \mathcal{M}  \models   \mathsf{Q} +  \Theta    \,   $.
\end{theorem}

\section{EF  games on $  \mathbb{T} $}

 Recall that to prove  Theorem  \ref{QPLUSCANTORPAIRING},   we need to show    that    certain expansions  of  $  \mathcal{M}    $  are  $n$-elementary equivalent. 
 We will prove this   using the characterization of  $n$-elementary equivalence in terms of   Ehrenfeucht-Fra\"iss\'e games of length  $n$. 
In this section,  
we  prove a certain extension lemma that will be used show that player  $  \exists  $  has a winning strategy.

First,  we define the class of maps we will work with. 
Let  $  X  \subseteq   \mathbb{T} $. 
We will say that a map    $  F :     X  \to   \mathbb{T}   $   is  $  \mathsf{Good}   $ if the following holds: 
\begin{enumerate}
\item  $  \left(   \forall   w  \in X  \right)   \left[  \,   F(w)  \neq  w   \rightarrow   w,  F(w)  \not\in   \lbrace  \mathsf{d}^m_0  :   \   m  \in  \mathbb{Z}  \,    \rbrace    \,    \right]   $.
\item  $  \chi \left(  w  \right)   =     \chi \left( F( w ) \right)   $  for all   $  w  \in  X   $.
\end{enumerate}
These properties have to do with how  $  \times^{  \mathcal{M}  }  $  is defined.

We will mainly be concerned with subsets of   $  \mathbb{T} $ that are closed under  $  \mathrm{S}^{ \mathcal{M} }  $  but not  $  +^{  \mathcal{M} }  $  and  $  \times^{  \mathcal{M} }  $. 
We will view these sets as   $ L_{  \mathbb{T} } $-structures where    $  L_{  \mathbb{T} } $    is  a  language where we have: 
(i)  a unary  function   symbol which is interpreted  as    $ \mathrm{S}^{  \mathcal{M} }    $; 
(ii) a ternary relation symbol which is interpreted as the graph of   $  \langle  \cdot  ,  \cdot  \rangle^{  \mathcal{B}  }  $; 
(iii) for each  $  n  \in  \omega \setminus  \lbrace 0, 1  \rbrace $,  a binary relation  symbol  which is interpreted  as    the graph of $  \mathcal{A}_n $; 
(iv) for each  $m  \in  \mathbb{Z} $,  a binary relation  symbol  which is interpreted  as    the graph of $  \mathcal{R}_m $.

The following lemma follows from  how the arithmetical operations are defined on  $   \mathcal{M}  $. 

\begin{lemma}
Let   $ X  \subseteq  \mathbb{T}  $ be closed under   $  \mathrm{S}^{  \mathcal{M}  }    $. 
Assume   $  F :    X \to  \mathbb{T} $  is a  \textsf{Good}  $ L_{  \mathbb{T} } $-embedding. 
Let  $  G  :  X  \cup  \mathbb{N} \to  M  $  be the extension of  $F$ that fixes  $ \mathbb{N}   $. 
Then,  $G$ is a partial embedding  into   $  \mathcal{M} $. 
\end{lemma}
\begin{proof}

We consider  $  \mathrm{S}^{  \mathcal{M}  }  $. 
Since    $F$ is an $  L_{  \mathbb{T}  }  $-homomorphism and  $G$  fixes  $  \mathbb{N} $  
\begin{align*}
G  \left(  \mathrm{S}^{  \mathcal{M}  }   v   \right)    &=    \begin{cases} 
G  \left(  v +  1 \right)          &  \mbox{  if  }   v  \in  \mathbb{N}
 \\
 F  \left(  \mathrm{S}^{  \mathcal{M}  }    v   \right)    &   \mbox{ if  }   v  \in X  
  \end{cases} 
  \\
  & = 
   \begin{cases} 
G  \left(  v  \right)  +  1      &       \mbox{  if  }   v  \in  \mathbb{N}
 \\
  \mathrm{S}^{  \mathcal{M}  }   F (v)    &   \mbox{ if  }   v  \in X  
  \end{cases} 
  \\
  & =    \mathrm{S}^{  \mathcal{M}  }   G(v) 
  \        .
\end{align*}

We consider  $  +^{  \mathcal{M}  }   $. 
Let  $  w, u,  v  \in    X  \cup  \mathbb{N}  $. 
We need to show that  $  w  =  u   +^{  \mathcal{M}  }   v  $  if and only if  $  G(w)  =G(u)     +^{  \mathcal{M}  }  G(v)   \,   $.
It  suffices to consider the case   $  w  \in  X  $. 
We have three cases:  
(i)  $  u  \in  X  $  and  $  v  \in  \mathbb{N} $; 
(ii) $  u  \in  \mathbb{N}  $  and  $  v  \in  X  $; 
(iii)  $  u,  v  \in  X  $. 
We consider (i). 
Since  $F$ is an  $ L_{  \mathbb{T} } $-embedding 
\[
 w  =     u   +^{  \mathcal{M}  }   v  =   \left(  \mathrm{S}^{ \mathcal{M}  }  \right)^v  u  
 \   \leftrightarrow   \    
F ( w )  =
\left(  \mathrm{S}^{ \mathcal{M}  }  \right)^v    F(u)  
= 
F(u)  +^{  \mathcal{M}  }  v
\   . 
\]
Case (ii) holds by similar reasoning. 
We consider case (iii). 
Since  $F$ is an  $ L_{  \mathbb{T} } $-embedding
\[
 w  =  u   +^{  \mathcal{M}  }   v  =  \langle  u,  v   \rangle^{  \mathcal{B}  } 
 \     \leftrightarrow    \  
F(w)   =   \langle  F( u) ,   F (v)    \rangle^{  \mathcal{B}  }   =  F(u)   +^{  \mathcal{M}  }    F( v) 
\       .
\]

We consider  $  \times^{  \mathcal{M}  }  $. 
Let  $  w, u,  v  \in    X  \cup  \mathbb{N}  $. 
We need to show that  $  w  =  u   \times^{  \mathcal{M}  }   v  $  if and only if  $  G(w)  =G(u)     \times^{  \mathcal{M}  }  G(v)   \,   $.
It  suffices to consider the case   $  w  \in  X  $. 
We have two cases:  
(1)  $  u  \in  X  $  and  $  v  \in  \mathbb{N} $; 
(2) $  u  \in  \mathbb{N}  $  and  $  v  \in  X  $; 
(3)  $  u,  v  \in  X  $. 
We consider (1). 
It  suffices to consider the case  $  v  \in \mathbb{N}_{  >  0 }   $.
Recall that  $  \mathcal{A}_1  $ is the identity map.
Since  $F$ is an  $ L_{  \mathbb{T} } $-embedding 
\[
 w  =     u   \times ^{  \mathcal{M}  }   v  =    \mathcal{A}_v  \left(  u \right) 
 \   \leftrightarrow   \    
F ( w )  =
  \mathcal{A}_v  \left(    F(u) \right) 
= 
F(u)  \times^{  \mathcal{M}  }  v
\   . 
\]

We consider (2). 
Since  $F$  is   \textsf{Good}, 
$   \chi \left( v  \right)  =    \chi \left(  F(v)  \right)  $  and  
$F$  fixes  $  \big\{  \mathsf{d}^m_0  :   \   m  \in  \mathbb{Z}  \,   \big\}    \cap  X  \,   $. 
Since  $F$ is also one-to-one 
\[
 w  =     u   \times ^{  \mathcal{M}  }   v  =    \mathsf{d}_0^{  u  \times  \chi \left( v  \right) }   
 \   \leftrightarrow   \    
F ( w )  =   \mathsf{d}_0^{  u  \times  \chi \left( v  \right) }     =  
  \mathsf{d}_0^{  u  \times  \chi \left(   F(v)  \right) }   
= u  \times^{  \mathcal{M}  }    F(v)
\   . 
\]

Finally,  we consider (3). 
Since  $F$ is a  \textsf{Good}  $ L_{  \mathbb{T} } $-embedding 
\begin{multline*}
 w  =     u   \times ^{  \mathcal{M}  }   v  =    \mathcal{R}_{  \chi  \left(  v  \right)  }  \left(  u \right) 
\     \   \leftrightarrow   \    
 \\
F ( w )  =
 \mathcal{R}_{  \chi  \left(  v  \right)  }  \left(   F(u)  \right)   = 
  \mathcal{R}_{  \chi  \left(  F(v)  \right)   } \left(   F(u)  \right)    = 
F(u)  \times^{  \mathcal{M}  }    F(v)
\       .       \qedhere
\end{multline*}
\end{proof}

We define the subsets of  $  \mathbb{T}  $  we will be concerned with. 
Let  $  X  \subseteq  \mathbb{T}  $. 
First,  we  define  the    sets    $  \mathcal{V}_k     \left(  X \right)   $   by recursion  on $  k \in  \omega  $:  
$  \mathcal{V}_0     \left(   X \right)   =  X $  and 
\begin{multline*}
\mathcal{V}_{ k+1}     \left( X  \right)    =  \mathcal{V}_k    \left(  X \right)  
 \cup     
 \\
 \Big\{  u   \in  \mathbb{T}   :   \     
      \left(  \exists   w  \in  \mathcal{V}_k  \left(   X  \right)    \right)    \left(  \exists  v  \in  \mathbb{T}  \right)  
      \left[  \,   w  =  \langle  u,  v  \rangle^{  \mathcal{B}  }     \;  \vee   \;    w  =  \langle  v,  u  \rangle^{  \mathcal{B}  }     \,   \right]  
      \    \Big\} 
     \        .
\end{multline*}
Let  
\[
 \mathcal{V}     \left(   X \right)    :=  \bigcup_{ k  \in  \omega  }      \mathcal{V}_k     \left(  X \right)   
 \         .
\]

The sets   $  \mathcal{V}_k     \left(  X \right)   $   are rather inconvenient to work with since they are not closed  under 
 $  \mathrm{S}^{  \mathcal{M}  }  $   and    $  \mathrm{P}^{  \mathcal{M}  }  $. 
We will work with the  sets  $  \mathcal{W}_k     \left(  X \right)   $ defined as follows:  for  $  k  \in  \omega  $
\[
  \mathcal{W}_k     \left(  X \right)   =  
   \Big\{   \left(  \mathrm{S}^{  \mathcal{M}  }  \right)^m  w  :   \   
  m  \in   \mathbb{Z}  \,  ,    \     w  \in     \mathcal{V}_k     \left(  X \right)      \Big\}
\]
(If  $  m  <  0  $,  then    $  \left(  \mathrm{S}^{  \mathcal{M}  }  \right)^m  :=    \left(  \mathrm{P}^{  \mathcal{M}  }  \right)^{-m}  $.)

We  introduce a partial order  $  \prec   $  on  $ \mathbb{T}  $  with respect to which we will use  well-founded recursion. 
For  $  w,  v  \in  \mathbb{T} $
\[
v  \prec  w      \     \      \Leftrightarrow_{  \mathsf{def}  }        \       \ 
 \left(  \exists  r  \in  \mathbb{Z}  \right)   \left[   \    
 \left(  \mathrm{S}^{  \mathcal{M}  }  \right)^r  v  \in    \mathcal{V}     \left(  w   \right)    \setminus \lbrace w  \rbrace   \   \right] 
\        .
\]
We write  $  \preceq  $  for the reflexive closure of  $  \prec  $.
The next lemma says that the structures  $  \left(    \mathcal{W}_k     \left(  X \right)   ,  \prec   \right)   $  are well-founded.

\begin{lemma}  \label{SectionQWellFoundednessLemma}
Let   $ k \in  \omega  $  and let  $  X  \subseteq  \mathbb{T} $  be finite. 
Let  $ N  =  \vert     \mathcal{V}_k     \left(  X \right)   \vert  $. 
If  $  w_0  \prec  w_1  \prec     \ldots   \prec  w_ {\ell }  $  is an  increasing sequence in   $  \mathcal{W}_k     \left(  X \right)     $, 
then $  \ell  <  N  $.
\end{lemma}
\begin{proof}

We need the following claim. 

  \begin{quote} {\bf (Claim)} \;\;\;\;\;\;    
 $  w  \not\prec      \left(  \mathrm{S}^{  \mathcal{M} }  \right)^m  w   $
 for all  $  w  \in  \mathbb{T}  $  and all   $  m  \in  \mathbb{Z} $. 
 \end{quote}  

We prove the claim by induction on the  complexity of elements of  $  \mathbb{T} $
(recall that   $  \mathbb{T} $  is inductively defined). 
We have the following cases: 
\begin{enumerate}
\item   $  w  \in  \lbrace   \mathsf{d}^m_n  :   \   (n, m )  \in  \omega  \times   \mathbb{Z}  \,  \rbrace $; 

\item     $  w  =  \langle  w_0,  w_1  \rangle^{  \mathcal{B} }  $  and the claim holds for  $  w_0$  and $ w_0  $;

\item    $  w  =  \mathcal{R}_q  \left(  v  \right)  $   where   $  q  \in  \mathbb{Z}   $  and the claim holds for  $  v $. 
\end{enumerate}
Case (1) is clear since  $  \mathcal{V} \left(  \mathsf{d}^m_n  \right)  =  \lbrace    \mathsf{d}^m_n    \rbrace   \,   $. 
 We consider   case  (2):    $  w  =  \langle  w_0,  w_1  \rangle^{  \mathcal{B} }  $  and the claim holds for  $  w_0$  and $ w_0  $. 
 Assume for the sake of a contradiction there exists  $  m  \in  \mathbb{Z}  $  such that     $  w  \prec      \left(  \mathrm{S}^{  \mathcal{M} }  \right)^m  w   $.
   By how    $  \mathrm{S}^{  \mathcal{M} }  $  is defined 
 \[
  \left(  \mathrm{S}^{  \mathcal{M} }  \right)^m  w
  = 
    \langle  w_0   \,  ,   \,     \left(  \mathrm{S}^{  \mathcal{M} }  \right)^m  w_1  \rangle^{  \mathcal{B} }  
    \         .
 \]
 Since   $  w  \prec      \left(  \mathrm{S}^{  \mathcal{M} }  \right)^m  w   $,    
 there exists  $  r  \in  \mathbb{Z}  $    such that  
\begin{align*}
  \left(  \mathrm{S}^{  \mathcal{M} }  \right)^r  w   \in     \mathcal{V}  \left(   \left(  \mathrm{S}^{  \mathcal{M} }  \right)^m  w    \right)   \setminus  \lbrace  w  \rbrace 
  & =  
 \mathcal{V}  \left(   w_0     \right)   \cup    \mathcal{V}  \left(   \left(  \mathrm{S}^{  \mathcal{M} }  \right)^m  w_1    \right) 
\     .
\end{align*}
We cannot have   $    \langle  w_0   \,  ,   \,     \left(  \mathrm{S}^{  \mathcal{M} }  \right)^r  w_1  \rangle^{  \mathcal{B} }   
   \in    \mathcal{V} \left(   w_0     \right)  $
since   $  \langle  \cdot  ,  \cdot  \rangle^{  \mathcal{B} }  $  is acyclic. 
Hence      
\[
     \langle  w_0   \,  ,   \,     \left(  \mathrm{S}^{  \mathcal{M} }  \right)^r  w_1  \rangle^{  \mathcal{B} }  
   =  
      \left(  \mathrm{S}^{  \mathcal{M} }  \right)^r  w   
    \in 
         \mathcal{V}  \left(   \left(  \mathrm{S}^{  \mathcal{M} }  \right)^m  w_1    \right)   
    \        .
    \]
 Sine     $  \langle  \cdot  ,  \cdot  \rangle^{  \mathcal{B} }  $  is acyclic.    from this we get  
\[
  \left(  \mathrm{S}^{    \mathcal{M} }  \right)^r    w_1  \in   \mathcal{V}  \left(   \left(  \mathrm{S}^{  \mathcal{M} }  \right)^m  w_1    \right) 
  \setminus  \lbrace  \left(  \mathrm{S}^{  \mathcal{M} }  \right)^m  w_1    \rbrace    
  \     .
  \]
But this means   $  w_1    \prec   \left(  \mathrm{S}^{  \mathcal{M} }  \right)^m  w_1   $, 
which contradicts the induction hypothesis. 
Thus,    $  w    \not\prec      \left(  \mathrm{S}^{  \mathcal{M} }  \right)^m  w   $  for all  $  m  \in  \mathbb{Z}  $.

 We consider   case  (3):     $  w  =  \mathcal{R}_q  \left(  v  \right)  $   where   $  q  \in  \mathbb{Z}   $  and the claim holds for  $  v $. 
 Assume for the sake of a contradiction there exists  $  m  \in  \mathbb{Z}  $  such that     $  w  \prec      \left(  \mathrm{S}^{  \mathcal{M} }  \right)^m  w   $.
 By how  $  \mathcal{R}_q  $  and    $  \mathrm{S}^{  \mathcal{M} }  $     are   defined 
 \[
  \left(  \mathrm{S}^{  \mathcal{M} }  \right)^m  w
  = 
    \langle    \mathcal{R}_{q-1}  \left(  v  \right)    \,  ,   \,     \left(  \mathrm{S}^{  \mathcal{M} }  \right)^m  v  \rangle^{  \mathcal{B} }  
    \         .
 \] 
  Since   $  w  \prec      \left(  \mathrm{S}^{  \mathcal{M} }  \right)^m  w   $,    there exists    $  r  \in  \mathbb{Z}  $  
    such that  
\[
 \left(  \mathrm{S}^{  \mathcal{M} }  \right)^r  w  \in     \mathcal{V}  \left(   \left(  \mathrm{S}^{  \mathcal{M} }  \right)^m  w    \right)   \setminus  \lbrace   \left(    \mathrm{S}^{  \mathcal{M} }  \right)^m  w   \rbrace 
   =  
 \mathcal{V}  \left(   \mathcal{R}_{q-1}  \left(  v  \right)       \right)   \cup    \mathcal{V}  \left(   \left(  \mathrm{S}^{  \mathcal{M} }  \right)^m  v    \right) 
\      .
\]
We cannot have   
$      \langle    \mathcal{R}_{q-1}  \left(  v  \right)    \,  ,   \,     \left(  \mathrm{S}^{  \mathcal{M} }  \right)^m  v  \rangle^{  \mathcal{B} }  
  \in 
    \mathcal{V}  \left(   \mathcal{R}_{q-1}  \left(  v  \right)       \right)   $
since    $  \langle  \cdot  ,  \cdot  \rangle^{  \mathcal{B}  }  $  is  acyclic. 
Hence 
\[  
    \langle    \mathcal{R}_{q-1}  \left(  v  \right)    \,  ,   \,     \left(  \mathrm{S}^{  \mathcal{M} }  \right)^r  v  \rangle^{  \mathcal{B} }  
=
  \left(  \mathrm{S}^{  \mathcal{M} }  \right)^r  w 
   \in 
          \mathcal{V}  \left(   \left(  \mathrm{S}^{  \mathcal{M} }  \right)^m  v    \right)   
  \       .
  \]
From this we get 
\[
   \left(  \mathrm{S}^{  \mathcal{M} }  \right)^r  v  \in  \mathcal{V}   \left(   \left(  \mathrm{S}^{  \mathcal{M} }  \right)^m  v    \right)  
\setminus   \lbrace   \left(  \mathrm{S}^{  \mathcal{M} }  \right)^m  v    \rbrace   
\     .
\]
But then  $  v  \prec   \left(  \mathrm{S}^{  \mathcal{M} }  \right)^m  v    $,  
which contradicts  the induction hypothesis. 
Thus,    $  w    \not\prec      \left(  \mathrm{S}^{  \mathcal{M} }  \right)^m  w   $  for all  $  m  \in  \mathbb{Z}  $.

Thus, by induction, the claim holds.

We complete the proof of the lemma. 
Assume for the sake of a contradiction there exists    an    increasing sequence 
 $  w_0  \prec  w_1  \prec     \ldots   \prec  w_{ \ell }   $   in   $  \mathcal{W}_k     \left(  X \right)     $
 where  $  \ell  \geq   N  $. 
 By how  $  \mathcal{W}_k     \left(  X \right)     $  is defined, 
 we have a sequence  $  v_0, v_1,   \ldots , v_{ \ell }     $  in   $  \mathcal{V}_k     \left(  X \right)     $
and a sequence  $  m_0, m_1,  \ldots , m_{ \ell }    $  in  $  \mathbb{Z} $  such that 
$  w_i =  \left(    \mathrm{S}^{  \mathcal{M} }  \right)^{  m_i }   v_i  $  for all $  i  \in  \lbrace 0, 1,  \ldots  ,  \ell \rbrace $.
Since  $  \ell   \geq      \vert     \mathcal{V}_k     \left(  X \right)   \vert  $,    there must  exist  two distinct  indexes  
$ p, q   \in      \lbrace 0, 1,  \ldots  ,  \ell \rbrace   $  such that  $  v_p   = v_q  $. 
Without loss of generality, we may assume  $  p <  q  $. 
Since  $  w_p  \prec  w_q  $,  we get that  
\[
   \left(    \mathrm{S}^{  \mathcal{M} }  \right)^{  m_p }   v_p    \prec     \left(    \mathrm{S}^{  \mathcal{M} }  \right)^{  m_q }   v_q    
   = 
      \left(    \mathrm{S}^{  \mathcal{M} }  \right)^{  m_q }   v_p
 \]  
 which contradicts the  claim. 
 This completes the proof of the lemma. 
\end{proof}

To make the proof of the extension lemma more transparent,  
we introduce a notion of    the $n$-th closure  in  $  \mathbb{T} $  of a set   $  Y  \subseteq  \mathbb{T}  $. 
Let   $  Y  \subseteq  \mathbb{T}  $. 
For  $  n  \in  \omega  $,  we define the set $  \mathsf{Cl}_n  \left(   Y  \right)   $  by recursion: 
$  \mathsf{Cl}_0  \left(   Y  \right)  =  Y    $  and  
$  \mathsf{Cl}_{n+1}   \left(   Y  \right)  =  \mathsf{Cl}_n  \left(   Y  \right)    \cup   W_n  $  where  
\begin{multline*}
W_n  =  
\big\{  \langle w,  v  \rangle^{  \mathcal{B}  }   :   \   w,  v   \in     \mathsf{Cl}_n  \left(   Y  \right)     \   \big\} 
\cup  
\bigcup_{  m  \in   \mathbb{Z}  }  
 \big\{    \mathcal{R}_m  w   :   \   w     \in  \mathsf{Cl}_n  \left(   Y  \right)     \   \big\} 
 \;   \cup   \;  
 \\
\bigcup_{  m  \in   \omega  \setminus  \lbrace 0, 1  \rbrace  }  
 \big\{    \mathcal{A}_m      :   \   w     \in   \mathsf{Cl}_n  \left(   Y  \right)     \   \big\} 
\          .
\end{multline*}

Let   $  L_{  \mathbb{T} }^-   :=     L_{  \mathbb{T} }    \setminus  \lbrace  \mathrm{S} \rbrace    \,   $. 
Since   $  \mathsf{Cl}_{n}  \left(  \mathcal{W}_{k}  \left( X  \right)   \right)     $  is not necessarily closed under  
$  \mathrm{S}^{  \mathcal{M}  }  $, 
we  view it as  an     $  L_{  \mathbb{T} }^-  $-structure.

\begin{lemma}   \label{SectionQNthClureLemma}
Let  $X \subseteq   \mathbb{T} $  and   $  k, N  \in  \omega $. 
Assume    $  F :  \mathcal{W}_{ k + N  }  \left(  X  \right)   \to   \mathbb{T}    $  is  a  \textsf{Good}  $  L_{  \mathbb{T} }^- $-embedding.
Then,  there  exists  a   one-to-one  $  L_{  \mathbb{T} }^-$-homomorphism   
 $  L:   \mathsf{Cl}_{N}  \left(  \mathcal{W}_{k}  \left( X  ,  \mathsf{d}^0_0 \right)   \right)   \to   \mathbb{T}   $
that agrees with  $  F  $  on     $  \mathcal{W}_{ k   }  \left(  X  \right)    $
and  fixes    $  \big\{   \mathsf{d}^m_0  :   \    m  \in  \mathbb{Z}   \,   \big\}     \,    $.
\end{lemma}
\begin{proof}

Observe that 
\[
\mathcal{W}_{k}  \left( X  ,  \mathsf{d}^0_0 \right)    
= 
  \mathcal{W}_{k}  \left( X   \right)    \cup     \mathcal{W}_{k}  \left(  \mathsf{d}^0_0   \right)  
= 
  \mathcal{W}_{k}  \left( X   \right)  
\cup   \big\{   \mathsf{d}^m_0  :   \    m  \in  \mathbb{Z}   \,   \big\}
\        .
\]
First,  we  introduce a rank function on   $ \mathsf{Cl}_{N}  \left(  \mathcal{W}_{k}  \left( X ,  \mathsf{d}^0_0   \right)   \right)    $
\[
\mathsf{rk}  \left(  w  \right)   =  \min  
\big\{  m  \in  \lbrace 0, 1,  \ldots ,  N  \rbrace  :   \  
 w  \in  \mathsf{Cl}_{m}  \left(  \mathcal{W}_{k}  \left( X ,   \mathsf{d}^0_0    \right)   \right)     \big\}  
\     .
\]
We define  $L$ by recursion:   for  $  w  \in      \mathsf{Cl}_{N}  \left(  \mathcal{W}_{k}  \left( X  ,   \mathsf{d}^0_0    \right)   \right)     $
\[
L  \left(  w  \right)   =   \begin{cases}
w            &   \mbox{  if   }     w  \in    \big\{   \mathsf{d}^m_0  :   \    m  \in  \mathbb{Z}   \,   \big\} 
\\
F ( w)      &     \mbox{  if   }   w  \in    \mathcal{V}_{ k  }  \left(  X  \right) 
\\
\langle  L(w_0),   L(w_0 )   \rangle^{   \mathcal{M}  }      &    \mbox{  if  }  w  \not\in   \mathcal{V}_{ k  }  \left(  X  \right)   
                     \mbox{  and   }    w  =  \langle  w_0,  w_1    \rangle^{   \mathcal{M}  }  
\\
\mathcal{R}_m  \left(   L(v)  \right)      &    \mbox{  if  }  w  \not\in   \mathcal{V}_{ k  }  \left(  X  \right)   
                     \mbox{  and   }      w  =       \mathcal{R}_m  \left(   v \right)    
\\
\mathcal{A}_n  \left(   L(v)  \right)      &    \mbox{  if  }  w  \not\in   \mathcal{V}_{ k  }  \left(  X  \right)   
                     \mbox{  and   }      w  =       \mathcal{A}_n  \left(   v \right)                                   
\end{cases}
\]
where   $   w_0,  w_1,  v   \in   \mathsf{Cl}_{N}  \left(  \mathcal{W}_{k}  \left( X   ,  \mathsf{d}^0_0   \right)   \right)     $, 
$  m  \in  \mathbb{Z}  $  and  $  n \in  \omega  \setminus  \lbrace 0, 1  \rbrace  $. 
This map is well-defined since: 
(i)   $   \mathcal{R}_m  \left(   v \right)    =  \langle    \mathcal{R}_{m-1}  \left(   v \right)     \,  ,   \,  v  \rangle^{  \mathcal{B} }  $  
and  
 $   \mathcal{A}_v  \left(   v \right)    =  \langle    \mathcal{A}_{n-1}  \left(   v \right)     \,  ,   \,  v  \rangle^{  \mathcal{B} }  $; 
 (ii) distinct maps in    $   \lbrace  \mathcal{R}_m   :  \   m  \in  \mathbb{Z} \,   \rbrace   \cup 
 \lbrace  \mathcal{A}_n :  \   n  \in    \omega  \setminus \lbrace 0, 1   \rbrace  \,   \rbrace       $
 have disjoint images; 
 (iii)  $F$ is  \textsf{Good}.
Recall that    $  \mathcal{A}_1  $  is  the identity map.

To show that $L$ is one-to-one,  we need the following claim.

   \begin{quote} {\bf (Claim)} \;\;\;\;\;\;    
   Let    $   v  \in    \mathsf{Cl}_{N}  \left(  \mathcal{W}_{k}  \left( X  \right)     ,  \mathsf{d}^0_0     \right)   $     and 
   $  w  \in  \mathcal{W}_{k + N- \mathsf{rk}  \left( v  \right)  }  \left( X  \right)    $
   be distinct. 
Then,    $ L(v)   \neq   F(w) $.
 \end{quote}  
 
 Pick   $   v  \in    \mathsf{Cl}_{N}  \left(  \mathcal{W}_{k}  \left( X    ,  \mathsf{d}^0_0    \right)   \right)   $. 
 We prove  by induction on  $  \mathsf{rk}   \left(  v  \right)  $  that  
 $  L(v)  \neq  F(w)  $  for all   $  w    \in \mathcal{W}_{k + N- \mathsf{rk}  \left( v  \right)  }  \left( X  \right)   \setminus  \lbrace  v  \rbrace  $. 
 We consider the base case   $   \mathsf{rk}_N  \left(  v  \right)  = 0  $. 
 We have two cases:   
 (1)  $  v    \in  \mathcal{W}_{k}  \left( X  \right)     $; 
 (2)  $  v  =  \mathsf{d}^m_0  $  where  $  m  \in  \mathbb{Z}  $. 
 In case of  (1),  since   $L$ and $F$ agree on   $   \mathcal{W}_{k}  \left( X  \right)     $ and  $F$  is one-to-one, 
  $  L(v)   =  F(v)   \neq    F(w)  $    for all   $  w    \in \mathcal{W}_{k + N }  \left( X  \right)   \setminus  \lbrace  v  \rbrace  $. 
   We consider (2):   $  v  =  \mathsf{d}^m_0     \,  $.
   We have  $  L(\mathsf{d}^m_0 )  =  \mathsf{d}^m_0   $. 
   Pick    $  w    \in \mathcal{W}_{k + N }  \left( X  \right)   \setminus  \lbrace  v  \rbrace  $. 
   Assume for the sake of a contradiction  $ \mathsf{d}^m_0   =  F(w)  $. 
   Since  $  w  \neq  \mathsf{d}^m_0   =  F(w)  $  and  $  F $  is  \textsf{Good},  
   we must have  $  w,  F(w)  \not\in    \big\{   \mathsf{d}^m_0  :   \    m  \in  \mathbb{Z}   \,   \big\}  $, 
   which contradicts the assumption that    $  \mathsf{d}^m_0   =  F(w)  $. 
   Thus,     $  L(v)   \neq    F(w)  $    for all   $  w    \in \mathcal{W}_{k + N }  \left( X  \right)   \setminus  \lbrace  v  \rbrace  $.

We consider the case      $   \mathsf{rk}  \left(  v  \right)   >  0  $. 
We   have the following cases: 
 \begin{enumerate}
\item  $  v  =  \langle  s,  t  \rangle^{  \mathcal{B}  }  $  where  
$  s,  t  \in    \ \mathsf{Cl}_{N}  \left(  \mathcal{W}_{k}  \left( X   ,   \mathsf{d}^0_0    \right)   \right)   $
and  $  \mathsf{rk} \left(  v  \right)  =  1  +  \max\left(   \mathsf{rk} \left(  s \right)  \,  ,  \,    \mathsf{rk} \left(  t  \right)   \right)    \,   $.

\item  $  v  =   \mathcal{R}_m \left(  s  \right)   $  where  
$  s  \in     \mathsf{Cl}_{N}  \left(  \mathcal{W}_{k}  \left( X    ,   \mathsf{d}^0_0    \right)   \right)   $, 
  $  m  \in  \mathbb{Z}  $
  and  $  \mathsf{rk} \left(  v  \right)  =  1  +  \mathsf{rk} \left(  s \right)      \,   $.

\item  $  v  =   \mathcal{A}_n \left(  s  \right)   $  where  
$  s  \in    \mathsf{Cl}_{N}  \left(  \mathcal{W}_{k}  \left( X     ,   \mathsf{d}^0_0      \right)   \right)    $, 
  $ n  \in  \omega  \setminus  \lbrace 0, 1  \rbrace   $
  and  $  \mathsf{rk} \left(  v  \right)  =  1  +  \mathsf{rk} \left(  s \right)      \,   $. 
 \end{enumerate}
Pick     $  w  \in  \mathcal{W}_{k + N- \mathsf{rk} \left( v  \right)  }  \left( X  \right)    $
such that  $  L(v)  =  F(w)  $.  
We need to show that  $  v  =  w  $. 
We consider  case  (1). 
Since  $L$ is an $L_{  \mathbb{T} }^-  $-homomorphism  and  $F$  is  an   $L_{  \mathbb{T} }^-  $-embedding
\[
\langle  L(s), L(t)   \rangle^{  \mathcal{B}  } 
= L(v)  = F(w)   \in  
 \mathcal{W}_{k + N- \mathsf{rk} \left(   v \right)  }  \left(   F  \left([ X  \right]   \right)   
 \     . 
\]
Since  $   k + N- \mathsf{rk}  \left(   v \right)  +1  \leq  k+N $  as  $   \mathsf{rk}  \left(   v \right)   >  0  $
\[
  L(s), L(t)  \in       
   \mathcal{W}_{ k + N  - \mathsf{rk}  \left(   v \right)  +1  }  \left(     F  \left[  X   \right]   \right)  
  \subseteq   \mathcal{W}_{ k + N  }  \left(     F  \left[  X   \right]   \right)  
  \      .
  \]
Since  $  F$ is an  $L_{  \mathbb{T} }^-  $-embedding, this means that we also have  
\begin{multline*}
 w_0    \in     \mathcal{W}_{ k + N  - \mathsf{rk}  \left(   v \right)  +1  }  \left(     F  \left[  X   \right]   \right)  
\subseteq  
   \mathcal{W}_{ k + N  - \mathsf{rk}  \left(   s \right)   }  \left(     F  \left[  X   \right]   \right)  
\     \mbox{  and   }      \   
\\
 w_1    \in     \mathcal{W}_{ k + N  - \mathsf{rk}  \left(   v \right)  +1  }  \left(     F  \left[  X   \right]   \right)  
\subseteq  
   \mathcal{W}_{ k + N  - \mathsf{rk}  \left(   t \right)   }  \left(     F  \left[  X   \right]   \right)  
 \mathcal{W}_{k + N }  \left( X  \right)    
   \end{multline*}
such that   
\[
   w  =  \langle  w_0, w_1  \rangle^{  \mathcal{B}  }  
\    \mbox{  and   }          \   
\langle  F(w_0) ,   F(w_1)  \rangle^{  \mathcal{B}  }  
= F(w)  =  L(v)  
= 
\langle  L(s), L(t)   \rangle^{  \mathcal{B}  } 
\     .
\]
Since  $  \langle  \cdot  ,  \cdot   \rangle^{  \mathcal{B}  }   $ is one-to-one,  from this  we get 
\[
 F(w_0)  =  L(s)      \    \mbox{  and  }      \         F(w_1)  = L(t)  
 \       .
 \]
By the induction hypothesis,   $  w =  s  $  and  $  w_1  =  t  $,  
and thus  
\[
  w  =  \langle  w_0,  w_1  \rangle^{  \mathcal{B} }   =     \langle  s,  t \rangle^{  \mathcal{B} }     =  v  
  \     .
  \]

We consider  cases    (2)  and  (3)  together since they are very similar. 
Let   
\[
\mathcal{G} :=   \begin{cases}
\mathcal{R}_m      &    \mbox{ in case of  }  (2) 
\\
\mathcal{A}_n      &    \mbox{ in case of  }  (3) 
\end{cases}
\]
Recall that    $  \mathcal{G}   $  is one-to-one. 
Since  $L$ is an $L_{  \mathbb{T} }^-  $-homomorphism  and  $F$  is  an   $L_{  \mathbb{T} }^-  $-embedding
\[
\mathcal{G}  \left(   L(s)  \right)     =   
 L(v)  = F(w)   \in  
 \mathcal{W}_{k + N- \mathsf{rk} \left(   v \right)  }  \left(   F  \left([ X  \right]   \right)   
 \     . 
\]
Since  $   k + N- \mathsf{rk}  \left(   v \right)  +1  \leq  k+N $  as  $   \mathsf{rk}  \left(   v \right)   >  0  $
\[
  L(s)  \in  
   \mathcal{W}_{k + N- \mathsf{rk} \left(   s \right)  }  \left(   F  \left([ X  \right]   \right)   
  \subseteq     \mathcal{W}_{ k + N  }  \left(     F  \left[  X   \right]   \right)  
  \      .
  \]
Since  $  F$ is an  $L_{  \mathbb{T} }^-  $-embedding, this means that we also have  
\[
u     \in   \mathcal{W}_{k + N -  \mathsf{rk} \left(   s \right)  }  \left( X  \right)    
 \   \mbox{ such that  }    \   
 w  =   \mathcal{G}   \left(  u  \right) 
   \      .
   \]
and   
\[
\mathcal{G}  \left(   F(u)  \right)     =   
 F(w)  =  L(v)    = 
\mathcal{G}  \left(   L(s)  \right)     
\         .
\]
Since  $   \mathcal{G} $ is one-to-one, $  F(u)  =  L(s)  $. 
By the induction hypothesis,   $u =  s $,     and thus  
\[
  w  =    \mathcal{G} \left(  u   \right)     =       \mathcal{G}   \left(  s   \right)     =   v  
  \     .
  \]

Thus,  by induction,   the claim holds.

 We use the claim to prove the lemma. 
Pick    $  w,  v    \in   \mathsf{Cl}_{N}  \left(  \mathcal{W}_{k}  \left( X ,  \mathsf{d}^0_0  \right)   \right) $  such that  $  L(w)  =  L(v )  $. 
We prove by induction on  $  \mathsf{rk}  \left( w  \right)  +   \mathsf{rk}  \left( v  \right)   $  that  $  w  = v $. 
First,  we  consider the case  $  \min  \left(   \mathsf{rk}  \left( w  \right)     \,  ,  \,    \mathsf{rk}  \left( v  \right)     \right)  =  0   \,    $.
If    one of  $  w,  v  $  is an element of  $  \big\{  \mathsf{d}^r_0  :   \  r  \in  \mathbb{Z}  \,  \big\}  $, 
then the other must be an element of  $  \mathcal{W}_{k}  \left( X ,  \mathsf{d}^0_0  \right)   $,   
and we get  that     $  w  = v  $   since   $F$ is   \textsf{Good}. 
If  one of     $  w,  v  $   is   an element  of    $  \mathcal{W}_{k}  \left( X  \right)   $,  
then  $   w   =   v  $  by the claim  since  $F$ and $L$ agree on   $  \mathcal{W}_{k}  \left( X  \right)   $.

We consider the case    $   \mathsf{rk}  \left( w  \right)    > 0  $  and     $    \mathsf{rk}  \left( v  \right)    >  0  $. 
We    have the following cases
\begin{enumerate}
\item    $  w  =  \langle  w_0 ,  w_1  \rangle^{  \mathcal{B}  }  $  and  $  v =  \langle  v_0, v_1  \rangle^{  \mathcal{B}  }  $
where  $ w_0, w_1,  v_0, v_1  \in     \mathsf{Cl}_{N}  \left(  \mathcal{W}_{k}  \left( X  \right)   \right)     $, 
 $  \mathsf{rk} \left(  w  \right)  =  1  +  \max\left(   \mathsf{rk} \left(  w_0 \right)  \,  ,  \,    \mathsf{rk} \left(  w_1  \right)   \right)     $ 
and  $  \mathsf{rk} \left(  v  \right)  =  1  +  \max\left(   \mathsf{rk} \left(  v_0 \right)  \,  ,  \,    \mathsf{rk} \left(  v_1  \right)   \right)    \,   $.

\item     $  w  =    \mathcal{G}  \left(  g  \right)  $  and  $  v =  \langle v_0, v_1  \rangle^{  \mathcal{B}  }  $
where  $ g,  v_0, v_1  \in     \mathsf{Cl}_{N}  \left(  \mathcal{W}_{k}  \left( X  \right)   \right)     $
  $  \mathcal{G}     \in  \lbrace  \mathcal{R}_m   :  \   m  \in  \mathbb{Z} \,   \rbrace   \cup 
 \lbrace  \mathcal{A}_n :  \   n  \in    \omega  \setminus \lbrace 0, 1 ,  2   \rbrace  \,   \rbrace      $, 
  $  \mathsf{rk} \left(  w  \right)  =  1  +  \mathsf{rk} \left(  g \right)    $
and  $  \mathsf{rk} \left(  v  \right)  =  1  +  \max\left(   \mathsf{rk} \left(  v_0 \right)  \,  ,  \,    \mathsf{rk} \left(  v_1  \right)   \right)    \,   $.

\item   $  w  =    \mathcal{G}  \left(  g   \right)  $  and  $  v =  \mathcal{H}  \left(  u \right)   $
where  $ u,  g  \in     \mathsf{Cl}_{N}  \left(  \mathcal{W}_{k}  \left( X  \right)   \right)     $
and  $  \mathcal{G} ,   \mathcal{H}      \in  \lbrace  \mathcal{R}_m   :  \   m  \in  \mathbb{Z} \,   \rbrace   \cup 
 \lbrace  \mathcal{A}_n :  \   n  \in    \omega  \setminus \lbrace 0, 1 ,  2    \rbrace  \,   \rbrace      $, 
   $  \mathsf{rk} \left(  w  \right)  =  1  +  \mathsf{rk} \left(  g \right)    $
and    $  \mathsf{rk} \left(  v  \right)  =  1  +  \mathsf{rk} \left(  u \right)       \,   $.
\end{enumerate} 
 We consider case (1). 
Since  $  L $  is an $  L_{  \mathbb{T} }^-  $  homomorphism,  we get  
\[
 \langle  L(w_0) ,  L( w_1)  \rangle^{  \mathcal{B}  }  
 =   L(w)   =  L(v)   =  
  \langle  L(v_0) ,  L( v_1)  \rangle^{  \mathcal{B}  }  
\     .
\]
Since  $  \langle  \cdot  ,  \cdot   \rangle^{  \mathcal{B}  }   $  is one-to-one,  
this implies  $  L(w_0) = L(v_0) $  and  $  L(w_1) = L(v_1) $. 
By the induction hypothesis,   we get $  w_0 = v_0 $  and  $  w_1= v_1 $, 
which implies  $  w = v $.

 We consider case (2). 
 Let  
 \[
  \mathcal{G}^{  \prime  }    =   \begin{cases}
  \mathcal{R}_{ m-1}     &   \mbox{  if  }      \mathcal{G} =  \mathcal{R}_m    \mbox{  where  }   m  \in  \mathbb{Z}
  \\
    \mathcal{A}_{ n-1}     &   \mbox{  if  }      \mathcal{G} =  \mathcal{A}_n    \mbox{  where  }   n  \in  \omega  \setminus  \lbrace 0, 1,  2  \rbrace
\         .
  \end{cases}
 \]
 Since  $  L $  is an $  L_{  \mathbb{T} }^-  $  homomorphism,  we get  
\[
\langle    \mathcal{G}^{  \prime  }      \left(  L(g)  \right)    \,    ,    \,   L(g)   \rangle^{  \mathcal{B}  }  
= 
 \mathcal{G} \left(  L(g)  \right) 
 =   L(w)   =  L(v)   =  
  \langle  L(v_0) ,  L( v_1)  \rangle^{  \mathcal{B}  }  
\     .
\]
Since  $  \langle  \cdot  ,  \cdot   \rangle^{  \mathcal{B}  }   $  is one-to-one,  
this implies
\[
 \mathcal{G}^{  \prime  }      \left(  L(g)  \right)    =      L(v_0)   
 \     \mbox{  and   }       \    
    L(g)          =    L( v_1 )  
    \         .       \tag{*}
\]
Since    $  \mathcal{G}      \left(  g  \right)     \in     \mathsf{Cl}_{N}  \left(  \mathcal{W}_{k}  \left( X  ,  \mathsf{d}^0_0  \right)   \right)   $  
and $  \mathsf{rk}  \left(    \mathcal{G}      \left(  g   \right)   \right)   >  0  $, 
we also have  
 \[
   \mathcal{G}^{  \prime }      \left(   g \right)     \in     \mathsf{Cl}_{N}  \left(  \mathcal{W}_{k}  \left( X   ,  \mathsf{d}^0_0 \right)   \right)   
   \    \     \mbox{  and  }      \        \   
   \mathsf{rk}  \left(    \mathcal{G}      \left(  g \right)   \right)     \geq   
\mathsf{rk}  \left(    \mathcal{G}^{  \prime }       \left(  g  \right)   \right)   
\        .
\]
Hence,  from (*)   we get  
\[
  L   \left(    \mathcal{G}^{  \prime  }      \left(  g   \right)   \right)  
  =     \mathcal{G}^{  \prime  }      \left(  L(g)  \right)    =      L(v_0)   
 \     \mbox{  and   }       \    
    L(g)          =    L( v_1 )  
    \         .       
\]
By the induction hypothesis 
\[
  \mathcal{G}^{  \prime  }      \left(  g   \right)    =  v_0 
 \     \mbox{  and   }       \    
   g        =   v_1 
    \         .       
\]
Thus 
\[
w  =     \mathcal{G}     \left(  g   \right)  
= 
\langle    \mathcal{G}^{  \prime  }      \left( g   \right)    \,    ,    \,  g   \rangle^{  \mathcal{B}  }  
= 
\langle   v_0,  v_1   \rangle^{  \mathcal{B}  }  
=  v  
\          .
\]

 Finally,  we consider case (3). 
 Since  $ L$  is  an $L_{  \mathbb{T}  }^-$-homomorphism 
\[
 \mathcal{G}    \left(  L(g)  \right)      =  L(w)   =  L(v)   =     \mathcal{H}    \left(  L(u)  \right)  
\     .
\]
Since   any two distinct maps in   
$  \big\lbrace   \mathcal{R}_m   :   \   m  \in  \mathbb{Z}   \,   \rbrace    \cup  
 \big\lbrace   \mathcal{A}_n   :   \   n   \in     \omega  \setminus  \lbrace 0,  1  \rbrace   \,   \rbrace         $ 
 have disjoint images,   
 $   \mathcal{G}    =    \mathcal{H}  $. 
 This means  
 \[
  \mathcal{G}    \left(  L(g)  \right)   =   \mathcal{G}    \left(  L(u)  \right) 
  \      .
 \] 
 Since     $    \mathcal{G}    $  is one-to-one,   this implies  
 $  L(g)  =  L(u)  $. 
 By the induction hypothesis,  $  g  =  u  $, and thus 
 \[
 w  =    \mathcal{G}    \left(  g \right)   =   \mathcal{G}    \left(  u \right)     =   v 
 \       .
 \]

Thus,  by induction,  $L$ is one-to-one.
\end{proof}

We are now ready to state and prove the extension lemma.

\begin{lemma}  \label{SectionQFirstExtensionLemma}
Let  $  k  \in  \omega  $. 
Let  $X \subseteq   \mathbb{T} $  be finite,  let  $  a  \in \mathbb{T} $     and  let 
$  N   \geq  2  \vert     \mathcal{V}_{ k   }  \left(  X, a  \right)    \vert  + 1 $. 
Assume     $  F :  \mathcal{W}_{ k + N }  \left(  X  \right)   \to  \mathbb{T}     $   is  a $  \mathsf{Good}  $     $  L_{  \mathbb{T} } $-embedding. 
Then, there exists  a   $  \mathsf{Good}  $   $  L_{  \mathbb{T} }$-embedding 
$
  G  :   \mathcal{W}_{  k  }  \left(  X  \cup  \lbrace a  \rbrace  \right)   \to   \mathbb{T} 
$
that agrees with  $F$ on    $  \mathcal{W}_{ k }  \left(  X  \right)    $.
\end{lemma}

\begin{proof}

We construct a finite  increasing sequence    of   $  L_{  \mathbb{T} }$-embeddings 
\[
\left(  H_j  :  D_j  \to  \mathbb{T}  :   \    j  \in  \lbrace 0, 1  ,  \ldots ,   \ell  \rbrace   \,    \right)
\]
where  $  D_0  =    \mathcal{W}_{ k }  \left(  X  \right)     $, 
$ D_{  \ell  }  =   \mathcal{W}_{  k  }  \left(  X  \cup  \lbrace a  \rbrace  \right) $
and   the followings holds   for each    $   j  \in  \lbrace 0, 1  ,  \ldots ,   \ell  \rbrace   $: 
\begin{enumerate}

\item  Pick   $  w  \in    D_j   \setminus   D_0  $. 
 Let   $  p  \in     \lbrace 0, 1  ,  \ldots ,   j-1   \rbrace  $  be the largest index  $r$ such that  $  w  \not\in  D_r  $. 
 Then,  $  w  $  is  $  \prec  $-minimal in   $   \mathcal{W}_{  k  }  \left(  X  \cup  \lbrace a  \rbrace  \right)   \setminus  D_p  $
 and 
 \[
 D_{ p+1}  =  D_p  \cup
 \big\{     \left(  \mathrm{S}^{  \mathcal{M}  }  \right)^m   w :   \   m   \in  \mathbb{Z}    \,  \big\}   
   \        .
 \]

\item  Pick   $  w  \in    D_j   \setminus   \left( D_0  \cup   \lbrace  \mathsf{d}^m_0  :   \   m   \in  \mathbb{Z}  \,   \rbrace    \right)   $.   
 Let   $  p  \in     \lbrace 0, 1  ,  \ldots ,   j-1   \rbrace  $  be the largest index  $r$ such that  $  w  \not\in  D_r  $. 
Assume  the following holds: 
\begin{enumerate}
\item There does not exist  $  w_0, w_1  \in D_{p}  $  such that  $  w  =  \langle w_0, w_1  \rangle^{  \mathcal{B}  }   $.

\item     There does   not exist  $  m, r  \in  \mathbb{Z}  $  and  $  v  \in  D_{p}   $  such that 
 $  \left(  \mathrm{S}^{  \mathcal{M}  }  \right)^r   w = \mathcal{R}_m  \left( v  \right)  $.

  \item     There   does   not exist  $  n  \in  \omega  \setminus  \lbrace 0, 1, 2  \rbrace $,  $  r  \in  \mathbb{Z} $   and  $  v  \in  D_{p}  $  such that  
  $  \left(  \mathrm{S}^{  \mathcal{M}  }  \right)^r    w   = \mathcal{A}_n  \left( v  \right)  $.
\end{enumerate}
Then
\begin{multline*}
  H_{p+1}   \left[     \big\{     \left(  \mathrm{S}^{  \mathcal{M}  }  \right)^m   w :   \   m   \in  \mathbb{Z}    \,  \big\}     \right] 
\subseteq 
\\
   \big\{    \mathsf{d}^m_n  :   \   (n, m )  \in  \omega  \times  \mathbb{Z}   \, ,   \  n>0   \,   \big\}       \setminus 
\mathcal{V}     \left(   H_{ p}  \left[  D_{p}   \right]   \right)      
\        .
\end{multline*}

\item  $  \left(   \forall   w  \in  D_j   \right)   \left[  \,   F(w)  \neq  w   \rightarrow   w,  F(w)  \not\in   \lbrace  \mathsf{d}^m_0  :   \   m  \in  \mathbb{Z}  \,    \rbrace    \,    \right]   $.

\item  $  \chi \left(  w  \right)   =     \chi \left( F( w ) \right)   $  for all   $  w  \in  D_j    $.
\end{enumerate}
(In the proof, when we write ``by $(i) $" where $  i \in \lbrace 1, 2,  4  \rbrace $, we will be making reference to clause (i) above.)
Clause (1) tells us how the recursion is done. 
Clause (2) is a requirement to ensure that subsequent extensions of  $  H_j $ will be one-to-one.
This requirement can be met since for  any finite  set  $  A  \subseteq  \mathbb{T}  $,  by how  $  \mathbb{T}  $  is defined, 
  there exists   $  q  \in  \omega  $  such that  
\[
 \big\{   \left(  \mathrm{S}^{  \mathcal{M}  }  \right)^m  w  :   \   
  m  \in   \mathbb{Z}  \,  ,    \     w  \in     \mathcal{V}     \left(  A \right)      \big\}
 \cap  
 \big\{  \mathsf{d}^m_n  :   \   (n, m )  \in  \omega  \times   \mathbb{Z}   \,  ,    \,   n \geq   q  \, \big\}
  = 
   \emptyset  
  \     .
  \]
Recall also that  $  \mathrm{S}^{  \mathcal{M}  }  \left(   \mathsf{d}^m_n  \right)  =  \mathsf{d}^{m+1}_n   $. 
Clauses (3) and (4) say that  $  H_j $  is   \textsf{Good} . 
In particular,   $  H_{  \ell  }  $  will be a    \textsf{Good}  $ L_{  \mathbb{T} }  $-embedding. 
Since  $  D_0 =   \mathcal{W}_{ k }  \left(  X  \right)    $,  $  H_0  \subseteq  H_{  \ell }  $  and  $F$ agrees with  $  H_0  $  on  $  D_0  $,  
we just need to set  $G :=   H_{  \ell  }   $  and we will be  done.

Before we construct the sequence    $  \left(  H_j  :   \  j  \in  \lbrace 0, 1,  \ldots  ,  \ell  \rbrace   \,  \right)  $, 
it will be important to have a bound on $  \ell  $. 
It follows from the next claim that  $  2  \ell + 1  \leq  2  \vert   \mathcal{V}_k   \left(     X,  a  \right)    \vert     +1  \leq     N  $. 

   \begin{quote} {\bf (Claim)} \;\;\;\;\;\;    
$   \ell   \leq   \vert   \mathcal{V}_k   \left(     X,  a  \right)    \vert     \,  $.
 \end{quote}  

We define an equivalence relation on  $   \mathcal{W}_k   \left(     X,  a  \right)  $: 
\[
w  \sim  v  
\   \Leftrightarrow_{  \mathsf{def}  }    \   
\big\{  \left(  \mathrm{S}^{  \mathcal{M} }  \right)^m   w :   \   m  \in  \mathbb{Z} \,  \rbrace  
= 
\big\{  \left(  \mathrm{S}^{  \mathcal{M} }  \right)^m   v  :   \   m  \in  \mathbb{Z} \,  \rbrace  
\        .
\]
Since  $    \mathcal{W}_k   \left(     X,  a  \right)   $  is the closure in $  \mathbb{T} $ of  $   \mathcal{V}_k   \left(     X,  a  \right)  $
under   $   \mathrm{S}^{  \mathcal{M} }    $   and  $   \mathrm{P}^{  \mathcal{M} }    $  
\[
\vert  \mathcal{W}_k   \left(     X,  a  \right)    /   \sim  \vert   \leq     \vert   \mathcal{V}_k   \left(     X,  a  \right)    \vert   
\          .
\]
By (1),   
$  D_j  \setminus  D_{ j-1}  $  is a $  \sim$-equivalence class     for  each  $  j  \in  \lbrace 1, 2,  \ldots ,  \ell  \rbrace  $. 
It then follows that  $  \ell $ is at most the number of  $  \sim $-equivalence classes, 
and hence    $   \ell   \leq   \vert   \mathcal{V}_k   \left(     X,  a  \right)    \vert    \,  $.
This completes the proof of the claim.

We need one more observation.

   \begin{quote} {\bf (Claim)} \;\;\;\;\;\;    
   For  all $  w  \in  \mathbb{T} $, there exist  at most one
     $  v  \in  \lbrace   \left(   \mathrm{S}^{  \mathcal{M}  }  \right)^r   w :   \   r  \in  \mathbb{Z}   \,   \rbrace   $  
   such that   $  w  =  \mathcal{G} \left( v  \right)  $  for some  
   $  v  \in  \mathbb{T}  $  and   
   $    \mathcal{G}  \in    \lbrace  \mathcal{R}_m  :   \   m  \in  \mathbb{Z}  \,   \rbrace    \cup  
  \lbrace  \mathcal{A}_n   :   \   n  \in  \omega  \setminus  \lbrace 0,   1   \rbrace   \,   \rbrace        \,     $.
 \end{quote}

 We will use the following property of the model  $  \mathcal{M} $: 
 \[
  \left(   \mathrm{S}^{  \mathcal{M}  }  \right)^r    g   \neq   g  
 \      \       \      \mbox{  for all }   g  \in  \mathbb{T} 
  \mbox{  and    all }   r  \in  \mathbb{Z}
  \         .      \tag{*}  
 \]
Assume  there exist  $  r, s  \in   \mathbb{Z}  $    such that  
\[
 \left(   \mathrm{S}^{  \mathcal{M}  }  \right)^r    w   =     \mathcal{G} \left( u  \right)  
\        \      \mbox{  and  }      \       \  
 \left(   \mathrm{S}^{  \mathcal{M}  }  \right)^s    w   =     \mathcal{H} \left( v  \right)   
\]
where   $  u, v  \in  \mathbb{T}  $  and  
 $    \mathcal{G}  ,    \mathcal{H}    \in    \lbrace  \mathcal{R}_m  :   \   m  \in  \mathbb{Z}  \,   \rbrace    \cup  
  \lbrace  \mathcal{A}_n   :   \   n  \in  \omega  \setminus  \lbrace 0,   1   \rbrace   \,   \rbrace        \,     $.
We need to show that  $  r  =  s  $. 
For  $   \mathcal{F}  \in  \lbrace   \mathcal{G} ,   \mathcal{H}  \rbrace  $,  let  
\[
 \mathcal{F}^{  \prime  }  =  \begin{cases}
 \mathcal{R}_{ m-1}      &     \mbox{  if  }     \mathcal{F}  =  \mathcal{R}_m     \mbox{  where  }  m  \in  \mathbb{Z}
 \\
  \mathcal{A}_{ n-1}      &     \mbox{  if  }     \mathcal{F}  =  \mathcal{A}_n     \mbox{  where  }  n  \in  \omega  \setminus  \lbrace 0, 1  \rbrace
\       .
 \end{cases}
\]
Recall that  $    \mathcal{A}_{ 1}   $  is the identity  map.  
We have 
\[
 \left(   \mathrm{S}^{  \mathcal{M}  }  \right)^r    w   =   
  \langle  \mathcal{G}^{  \prime  }  \left( u  \right)    \,  ,  \,   u  \rangle^{  \mathcal{B}  }
\        \      \mbox{  and  }      \       \  
 \left(   \mathrm{S}^{  \mathcal{M}  }  \right)^s    w   =   
  \langle  \mathcal{H}^{  \prime  }  \left( u  \right)    \,  ,  \,   v  \rangle^{  \mathcal{B}  }
\       .
\]
By how    $   \mathrm{S}^{  \mathcal{M}  }   $  is defined,  
this means that  
\[
  \langle  \mathcal{G}^{  \prime  }  \left( u  \right)    \,  ,  \,   \left(   \mathrm{S}^{  \mathcal{M}  }  \right)^{-r}  u  \rangle^{  \mathcal{B}  }
  = 
  w 
   = 
     \langle  \mathcal{H}^{  \prime  }  \left( u  \right)    \,  ,  \,    \left(   \mathrm{S}^{  \mathcal{M}  }  \right)^{-s}   v  \rangle^{  \mathcal{B}  }
\      .
\]
Since  $  \langle  \cdot  ,  \cdot    \rangle^{  \mathcal{B}  }  $  is one-to-one 
\[
 \mathcal{G}^{  \prime  }  \left( u  \right)    =     \mathcal{H}^{  \prime  }  \left( v  \right)    
 \        \mbox{  and   }     \   
 \left(   \mathrm{S}^{  \mathcal{M}  }  \right)^{-r}  u     =      \left(   \mathrm{S}^{  \mathcal{M}  }  \right)^{-s}   v  
 \     .
\] 
Since  distinct maps in 
$  \lbrace  \mathcal{R}_m  :   \   m  \in  \mathbb{Z}  \,   \rbrace    \cup  
  \lbrace  \mathcal{A}_n   :   \   n  \in  \omega  \setminus  \lbrace 0,   1   \rbrace   \,   \rbrace       $ 
  have disjoint images,    $   \mathcal{G}^{  \prime  }    =   \mathcal{H}^{  \prime  }    \,   $.
  Since  $   \mathcal{G}^{  \prime  }   $  is one-to-one,
  from  $   \mathcal{G}^{  \prime  }  \left( u  \right)    =     \mathcal{G}^{  \prime  }  \left( v  \right)     $  we get    $  u =  v  $.
This means that  
\[
 \left(   \mathrm{S}^{  \mathcal{M}  }  \right)^{-r}  u     =      \left(   \mathrm{S}^{  \mathcal{M}  }  \right)^{-s}   u  
\       .
\]
By (*),  it follows that  we must have  $  r  =  s  $. 
This completes the proof of the claim.

We proceed to construct the sequence of    $ L_{  \mathbb{T} }  $-embeddings. 
First,  let   $  D_0  :=   \mathcal{W}_{ k }  \left(  X  \right)    $
and let    $  H_0  $  be the restriction of  $F$ to $  D_0  $. 
Since  $F$ is a   \textsf{Good}  $ L_{  \mathbb{T} }  $-embedding, 
$  H_0 $  is also a   a   \textsf{Good}  $ L_{  \mathbb{T} }  $-embedding. 
Next,  assume we have defined  
\[
  H_0  \subseteq  H_1  \subseteq   \ldots  \subseteq  H_{ q-1}  
  \    \mbox{  and  }     \  
  D_0  \subseteq  D_1  \subseteq   \ldots    \subseteq  D_{ q-1 } 
  \       .
  \]
If  $   D_{ q-1}  =   \mathcal{W}_{  k }  \left(  X  , a  \right)         $, stop.
Otherwise, we proceed to  construct  $ H_q  $. 
Pick $    \mathsf{g}   \in  \mathcal{W}_{  k }  \left(  X  , a   \right)    \setminus D_{ q-1 }  $ 
that is  $  \prec $-minimal in   $     \mathcal{W}_{  k }  \left(  X  , a   \right)    \setminus D_{ q-1 }     \,   $. 
To simplify matters,  we assume  there does not exist  $  r    \in    \mathbb{Z}   \setminus  \lbrace  0  \rbrace $ and 
$  v  \in  D_{ q-1}  $ that:  
\begin{itemize}
\item[-]   $   \left(  \mathrm{S}^{  \mathcal{M}  }  \right)^r   \mathsf{g}   =   \mathcal{R}_m  \left(  v  \right)  $ for some  $  m \in  \mathbb{Z} $;

\item[-]   $   \left(  \mathrm{S}^{  \mathcal{M}  }  \right)^r   \mathsf{g}   =   \mathcal{A}_n \left(  v  \right)  $ 
for some  $  n \in  \omega  \setminus  \lbrace 0, 1,  2  \rbrace $.
\end{itemize} 
We can do this since   there  exist   at most one
     $  w   \in  \lbrace   \left(   \mathrm{S}^{  \mathcal{M}  }  \right)^r    \mathsf{g} :   \   r  \in  \mathbb{Z}   \,   \rbrace   $  
   such that   $  w  =  \mathcal{G} \left( w_1   \right)  $  for some  
   $  w_1  \in  \mathbb{T}  $  and  
   $    \mathcal{G}  \in    \lbrace  \mathcal{R}_m  :   \   m  \in  \mathbb{Z}  \,   \rbrace    \cup  
  \lbrace  \mathcal{A}_n   :   \   n  \in  \omega  \setminus  \lbrace 0,   1   \rbrace   \,   \rbrace        \,     $.
 Set  
\[
D_q  :=  D_{ q-1 }  \cup  
  \big\{  \left(   \mathrm{S}^{  \mathcal{M}  }  \right)^r   \mathsf{g}   :   \  r  \in  \mathbb{Z}  \   \big\}
\       .
\]
Since  $  H_q  :   D_q  \to   \mathbb{T} $  is  supposed to be   an extension of  $  H_{ q-1 }  $, 
  we just    need to define    $  H_q $   on   
  $  \big\{  \left(   \mathrm{S}^{  \mathcal{M}  }  \right)^r   \mathsf{g}   :   \  r  \in  \mathbb{Z}  \   \big\}  $. 
  We proceed to do so. 
Let  $ q^{  \star }   $  be  the least  $  i  \in  \omega \setminus  \lbrace 0  \rbrace  $  such that   
 \[
 \big\{  \mathsf{d}^r_{  i   }  :   \   r   \in     \mathbb{Z}   \,  ,  \big\}  
   \cap     
    \mathcal{V}     \left(   H_{ q-1}  \left[  D_{q-1}   \right]   \right) 
  = 
   \emptyset  
  \     .
  \]  
We define   $  H_q  $ on  $  D_p  \setminus  D_{p-1}  $  as follows: 
\[
H_q  \left(      \left(   \mathrm{S}^{  \mathcal{M}  }  \right)^r   \mathsf{g}     \right)   
 := 
   \left(   \mathrm{S}^{  \mathcal{M}  }  \right)^r   H_q  \left(      \mathsf{g}     \right)   
   \     \mbox{  for  all }   r  \in  \mathbb{Z}   
\]
and    
\[
  H_q  \left(      \mathsf{g}     \right)   
:=  
\begin{cases}
  \langle H_{ q-1} ( g_0 )   \,  ,  \,   H_{ q-1} ( g_1 )   \rangle^{  \mathcal{B}  }
         &   \mbox{ if  }   \mathsf{g}   = \langle  g_0   ,      g_1   \rangle^{  \mathcal{B}  }  
                      \mbox{  and  }    g_0,  g_1   \in   D_{ q-1 }

 \\
  \mathcal{R}_m  \left( H_{ q-1} ( h )     \right)          &   \mbox{ if  }   \mathsf{g}   =  \mathcal{R}_m  \left(  h     \right) 
             \mbox{  and  }     h  \in   D_{ q-1 }   
 \\
   \mathcal{A}_n  \left( H_{ q-1} ( h )     \right)          &   \mbox{ if  }   \mathsf{g}   =  \mathcal{A}_n  \left(  h     \right) 
                \mbox{  and  }     h  \in   D_{ q-1 }   
 \\
\mathsf{d}^{   \chi \left(  \mathsf{g}  \right)   }_0                        
                                                    &   \mbox{ if  }   \mathsf{g}   =  \mathsf{d}_0^{    \chi \left(  \mathsf{g}  \right)  }
 \\
  \mathsf{d}_{ q^{  \star } }^{     \chi \left(  \mathsf{g}  \right)     }      &    \mbox{ otherwise}     
\end{cases}
\]
where  $  m  \in  \mathbb{Z}  $  and    $  n  \in  \omega  \setminus  \lbrace 0, 1 ,  2  \rbrace   $. 
Recall that  $  \mathcal{A}_2  (v)  =  \langle  v,  v  \rangle^{  \mathcal{B}  }  $  for all $  v  \in  \mathbb{T}  $.
Recall also that  $  \chi  \left(   \mathsf{d}^m_n \right)  =  m  $  for all  $  (n,m)  \in    \omega  \times  \mathbb{Z}      \,     $.

We need to show  that       $  H_q  $ is a well-defined   $ L_{  \mathbb{T} } $-embedding. 
This  will follow from the following   claims.

   \begin{quote} {\bf (Claim)} \;\;\;\;\;\;    
$  H_q  $  is  a well-defined    $ L_{  \mathbb{T} } $-homomorphism. 
 \end{quote}

 We just need to show that  $  H_q   \left(   \mathsf{g}   \right)  $  is well-defined.
First, we consider the case     $    \mathsf{g}   \in  \big\{   \mathsf{d}^m_0   :   \    m   \in    \mathbb{Z}  \,   \big\}  $. 
Assume    $   \mathsf{g}   =   \mathsf{d}^{  m }_0   $   where   $  m  \in  \mathbb{Z}        \,   $. 
We have  
\[
 H_q   \left(   \mathsf{g}   \right)   : =  
 \mathsf{d}_0^{    \chi \left(  \mathsf{g}  \right)  }        =     \mathsf{d}^{  m }_0
  \        .
\]
Since   $  \mathsf{g}   $  has only one presentation,    $  H_q   \left(   \mathsf{g}   \right)  $  is well-defined.

 Next,  we consider the case  
 \[
   \mathsf{g}   = \langle  g_0   ,      g_1   \rangle^{  \mathcal{B}  }    
   \     \mbox{  where   }    \     g_0, g_1  \in D_{p-1}  
   \      .
   \]
  It could be that  $  \mathsf{g}   $  has a different presentation. 
 Clearly,    $  \mathsf{g}   \not\in   \big\{  \mathsf{d}^m_n :   \   (n, m)  \in  \omega  \times  \mathbb{Z}  \,  \rbrace  $. 
 We thus have the following cases: 
 \begin{itemize}
 \item[(i)]  $  \mathsf{g}   = \langle  h_0   ,      h_1   \rangle^{  \mathcal{B}  }    $   where   $ h_0, h_1  \in D_{ q-1}  $

 \item[(ii) ]   $  \mathsf{g} =  \mathcal{G} \left( v  \right)  $  where  $  v  \in D_{ q-1}  $   and  
  $  \mathcal{G}  \in    \lbrace  \mathcal{R}_m   :  \   m  \in  \mathbb{Z} \,   \rbrace   \cup 
 \lbrace  \mathcal{A}_n :  \   n  \in    \omega  \setminus \lbrace 0, 1 ,  2   \rbrace  \,   \rbrace       \,   $.
 \end{itemize}
 We consider (i). 
 Since $  \langle  \cdot  ,  \cdot   \rangle^{  \mathcal{B} }  $  is one-to-one,  
from $  \langle  g_0   ,      g_1   \rangle^{  \mathcal{B}  }      = \langle  h_0   ,      h_1   \rangle^{  \mathcal{B}  }    $
we get  $  g_0 = h_0  $  and  $  g_1 = h_1  $.
We   consider (ii). 
Let  
\[
\mathcal{G}^{  \prime  }  :=   \begin{cases}
\mathcal{R}_{ m-1}      \mbox{  if  }     \mathcal{G} =  \mathcal{R}_m   \mbox{  where  }  m  \in  \mathbb{Z}
\\
\mathcal{A}_{ n-1}      \mbox{  if  }     \mathcal{G} =  \mathcal{A}_n   \mbox{  where  }  n  \in  \omega \setminus  \lbrace 0, 1 , 2  \rbrace 
\      .
\end{cases}
\]
A priori,  we may have  $ \mathcal{G}^{  \prime  }  \left(  v  \right)  \not\in  \mathcal{W}_k  \left(  X,  a  \right)     $
even though    $ \mathcal{G}  \left(  v  \right)    \in  \mathcal{W}_k  \left(  X,  a  \right)     $. 
However,   since $  \langle  \cdot  ,  \cdot   \rangle^{  \mathcal{B} }  $  is one-to-one,  
from   $   \langle  g_0   ,      g_1   \rangle^{  \mathcal{B}  }    =   \mathcal{G}  \left(  v  \right)  
=  \langle  \mathcal{G}^{  \prime  }  \left(  v  \right)    ,  v  \rangle^{  \mathcal{B}  }
$ 
we get   
\[
   g_1  =  v      \   \mbox{  and   }    \     g_0  =    \mathcal{G}^{  \prime  }  \left(  v  \right)   
   \     .    
   \]
 Since  $  H_{ q-1}  $     is  an  $  L_{  \mathbb{T}  }  $-homomorphism    and  $  g_1  =  v  $
 \[
   H_{q-1}  (g_0 )   =     \mathcal{G}^{  \prime  }    \left(    H_{ q-1} (v)  \right)   
   \     \mbox{  and   }    \   
   H_{ q-1 } (g_1)  =  H_{ q-1}  (v)  
   \       .
\]   
Using the presentation    $ \mathsf{g}   = \langle  g_0   ,      g_1   \rangle^{  \mathcal{B}  }      $  gives  
$  H_q  \left(  \mathsf{g}  \right)  =  \langle  H_{q-1}  (g_0 )   ,     H_{ q-1} ( g_1)   \rangle^{  \mathcal{B}  }   $. 
Using the presentation   $ \mathsf{g} =  \mathcal{G}   \left(  v  \right)   $  also  gives 
\begin{align*}
 H_q  \left(  \mathsf{g}  \right)   
 &= 
  \mathcal{G}   \left(    H_{ q-1} (v)  \right) 
  \\
  &=  
  \langle      \mathcal{G}^{  \prime  }    \left(    H_{ q-1} (v)  \right)   \,  ,  \,     H_{ q-1} (v)     \rangle^{  \mathcal{B}  }
  \\
  &=    \langle  H_{q-1}  (g_0 )   ,     H_{ q-1} ( g_1)   \rangle^{  \mathcal{B}  }  
  \       .
\end{align*}
Thus,  $  H_q  \left(  \mathsf{g}  \right)   $  is well-defined.

Finally, we consider the case 
\[
   \mathcal{G}   \left(  u  \right)     =  \mathsf{g} =     \mathcal{H}   \left(  v  \right)   
\]
where    $  u,  v  \in  D_{ q-1}    $   and  
$  \mathcal{G} ,   \mathcal{H}  \in    \lbrace  \mathcal{R}_m   :  \   m  \in  \mathbb{Z} \,   \rbrace   \cup 
 \lbrace  \mathcal{A}_n :  \   n  \in    \omega  \setminus \lbrace 0, 1 ,  2   \rbrace  \,   \rbrace    \,    $. 
Since  distinct maps in 
$   \lbrace  \mathcal{R}_m   :  \   m  \in  \mathbb{Z} \,   \rbrace   \cup 
 \lbrace  \mathcal{A}_n :  \   n  \in    \omega  \setminus \lbrace 0, 1   \rbrace  \,   \rbrace    $  
 have disjoint images,  
 from    $   \mathcal{G}   \left(  u  \right)    =     \mathcal{H}   \left(  v  \right)     $   we get     $     \mathcal{G}   =     \mathcal{H}      \,  $. 
Since  $     \mathcal{G}    $  is one-to-one,  
from    $     \mathcal{G}   \left(  u  \right)   =     \mathcal{G}   \left(  v  \right)     $  we get  $  u  = v  $. 
Thus,   $  H_q  \left(  \mathsf{g}  \right)   $  is well-defined  in this case also.

   \begin{quote} {\bf (Claim)} \;\;\;\;\;\;    
$  H_q  $   is  \textsf{Good}.
 \end{quote}

 By how   $  H_q  $   is defined,    we just need to show that    
 $  \chi   \left(      \left(   \mathrm{S}^{  \mathcal{M}  }  \right)^r   \mathsf{g}     \right)     =  
 \chi   \left(      \left(   \mathrm{S}^{  \mathcal{M}  }  \right)^r       H_q   \left( \mathsf{g}    \right)   \right)     $  for  all  $  r  \in  \mathbb{Z} $. 
 Since    $     \chi   \left(      \left(   \mathrm{S}^{  \mathcal{M}  }  \right)^r   w   \right)    =  r  +  \chi (w)  $  for all $  w  \in \mathbb{T} $, 
 it suffices to show that  
 $  \chi   \left(     H_q   \left( \mathsf{g}    \right)   \right)    =   \chi   \left(    \mathsf{g}  \right)       \,   $.
We have  
\begin{align*}
 \chi   \left(     H_q   \left( \mathsf{g}    \right)   \right)      
 &=  \begin{cases}
 \chi   \left(      H_{ q-1} ( g_1 )   \right) 
         &   \mbox{ if  }   \mathsf{g}   = \langle  g_0   ,      g_1   \rangle^{  \mathcal{B}  }  
 \\
 \chi   \left(      H_{ q-1} ( h )   \right) 
           &   \mbox{ if  }   \mathsf{g}   =  \mathcal{R}_m  \left(  h     \right) 
 \\
 \chi   \left(      H_{ q-1} ( h  )   \right) 
         &   \mbox{ if  }   \mathsf{g}   =  \mathcal{A}_n  \left(  h     \right)  
 \\
  \chi   \left(    \mathsf{g}  \right)         &    \mbox{ otherwise }     
\end{cases}
\\
 &=  \begin{cases}
 \chi   \left(     g_1  \right) 
         &   \mbox{ if  }   \mathsf{g}   = \langle  g_0   ,      g_1   \rangle^{  \mathcal{B}  }  
 \\
 \chi   \left(   h \right) 
           &   \mbox{ if  }   \mathsf{g}   =  \mathcal{R}_m  \left(  h     \right) 
 \\
 \chi   \left(    h \right) 
         &   \mbox{ if  }   \mathsf{g}   =  \mathcal{A}_n  \left(  h     \right)  
 \\
  \chi   \left(    \mathsf{g}  \right)         &    \mbox{ otherwise }     
\end{cases}
\\
&=   \chi  \left(  \mathsf{g}  \right)  
\        . 
\end{align*}
  where the second  equality uses that  $  H_{ q-1}  $  is \textsf{Good}  and  in the first equality we have used that  $  \chi \left(  \mathsf{d}^m_n  \right)  =  m  $. 
 Thus,  $  H_q  $  is  \textsf{Good.}

 The next claim in combination with  Lemma    \ref{SectionQNthClureLemma} will be used to show that  $  H_q  $  is  one-to-one. 
 Lemma      \ref{SectionQNthClureLemma}  says that  
we have a  \textsf{Good}   one-to-one  $  L_{  \mathbb{T} }^- $-homomorphism   
 $L  :    \mathsf{Cl}_N  \left(    \mathcal{W}_k  \left( X  ,  \mathsf{d}^0_0 \right)    \right)  \to  \mathbb{T} $ 
 that agrees with $F$ on   $  \mathcal{W}_k  \left( X  \right)  $.  
 We  define a rank function  $  \mathsf{rk}_q  $  on  $  D_q $. 
 \[
\mathsf{rk}_q  \left(  w   \right) :=      \min  \lbrace   j  \in  \lbrace 0, 1,  \ldots ,  q  \rbrace :   \   w   \in  D_j  \,  \rbrace 
\       .
\]

    \begin{quote} {\bf (Claim)} \;\;\;\;\;\;    
    Let   $  w  \in  D_q  $.
Assume there does not exist    $ c  \in  \big\{  \mathsf{d}^m_n :   \  (n, m  )  \in  \omega  \times  \mathbb{Z}  \,  \rbrace 
\setminus  \left(    \big\{  \mathsf{d}^m_0 :   \   m  \in   \mathbb{Z}  \,  \rbrace   \cup  \mathcal{V} \left(  F \left[ D_0  \right]  \right)    \right)   $ 
such that  $  c  \preceq   H_q (w)   $. 
Then,   $  H_q (w)   \in      \mathsf{Cl}_{   2\mathsf{rk}_q  \left(  w   \right) }  \left(    \mathcal{W}_k  \left( X   ,  \mathsf{d}^0_0  \right)    \right)   
\subseteq    \mathsf{Cl}_{   N }  \left(    \mathcal{W}_k  \left( X    ,  \mathsf{d}^0_0  \right)    \right)       $
and  $  H_q  (w)  =  L(w)  \,  $. 
 \end{quote}

 Recall that  in the first claim we showed that  $  2 \ell   +1  \leq  N   $; 
 in particular,  $  2q  <  N  $. 
 We prove by induction on   $    \mathsf{rk}_q  \left(  w   \right)    $  that  
  $  H_q (w)   \in      \mathsf{Cl}_{   2\mathsf{rk}_q  \left(  w   \right) }  \left(    \mathcal{W}_k  \left( X    ,  \mathsf{d}^0_0   \right)    \right)     $
  and  $  H_q (w)  =  L(w)  \,   $. 
  If  $  \mathsf{rk}_q  \left(   w   \right)  = 0 $,  then this clearly holds. 
We consider the case  $   \mathsf{rk}_q  \left(   w  \right)   >  0  $. 
By (2),  we   have the following cases: 
\begin{itemize}
\item[(a)]  $     w  =  \langle u,  v  \rangle^{  \mathcal{B}  }  $  where  
$  u,  v  \in  D_{q-1}     \,  $.

\item[(b)]  $    w  =  \left(  \mathrm{S}^{  \mathcal{M}  }  \right)^r    \mathcal{G}  \left(   v  \right)   $  where  
$  r  \in  \mathbb{Z} $, 
$   v  \in  D_{q-1}     $
and     $  \mathcal{G}  \in    \lbrace  \mathcal{R}_m   :  \   m  \in  \mathbb{Z} \,   \rbrace   \cup 
 \lbrace  \mathcal{A}_n :  \   n  \in    \omega  \setminus \lbrace 0, 1   \rbrace  \,   \rbrace       \,   $.
\end{itemize}
We consider (a). 
In this case we also have no    $ c  \in  \big\{  \mathsf{d}^m_n :   \  (n, m  )  \in  \omega  \times  \mathbb{Z}  \,  \rbrace 
\setminus  \left(    \big\{  \mathsf{d}^m_0 :   \   m  \in   \mathbb{Z}  \,  \rbrace   \cup  \mathcal{V} \left(  F \left[ D_0  \right]   \right)    \right)   $ 
such that  $  c  \preceq    H_q (u)   $  or  $  c  \preceq  H_q ( v )  $  
since we would otherwise get  $  c  \preceq   H_q (w)  $. 
Hence, by the induction hypothesis,  
$  u,  v  \in     \mathsf{Cl}_{ p  }  \left(    \mathcal{W}_k  \left( X   ,  \mathsf{d}^0_0   \right)     \right)   $  
where    
$  p  =  \max  \left(   2 \mathsf{rk}_q  \left(   u \right)  \,  ,  \,  2 \mathsf{rk}_q  \left(  v  \right)   \right)     $, 
  $  H_q (u) = L(u)  $  and  $  H_q (v)  =  L(v)  $. 
  Since  $  H_q  $  and  $  L$  are both  $  L_{  \mathcal{T} }^-$-homomorphisms 
  \[
H_q (w)  =    \langle H_q (u) ,  H_q ( v )   \rangle^{  \mathcal{B}  }    =      \langle  L (u) ,  L   ( v )   \rangle^{  \mathcal{B}  } 
=  L(w) 
\      .
  \]
Furthermore 
\[
w  \in    \mathsf{Cl}_{ p +1    }  \left(    \mathcal{W}_k  \left( X   ,  \mathsf{d}^0_0  \right)     \right)     
\subseteq     \mathsf{Cl}_{ p + 2    }  \left(    \mathcal{W}_k  \left( X    ,  \mathsf{d}^0_0   \right)     \right)    
\subseteq  
  \mathsf{Cl}_{   2  \mathsf{rk}_q  \left(  w   \right) }  \left(    \mathcal{W}_k  \left( X    ,  \mathsf{d}^0_0 \right)    \right) 
\]
since  $   \mathsf{rk}_q  \left(  w   \right)   \geq   1  +   \max  \left(   \mathsf{rk}_q  \left(   u \right)  \,  ,  \,   \mathsf{rk}_q  \left(  v  \right)   \right)      \,    $.

Finally, we consider (b). 
We have  $    w  =   \left(  \mathrm{S}^{  \mathcal{M}  }  \right)^r    \mathcal{G}  \left(     v  \right)   $. 
As above,  by the induction hypothesis,  
$  v   \in    \mathsf{Cl}_{   2  \mathsf{rk}_q  \left(  v   \right) }  \left(    \mathcal{W}_k  \left( X    ,  \mathsf{d}^0_0   \right)    \right)   $,  
and    $  H_q  (v)  =  L(v)     \,  $. 
  Since  $  D_q $  is closed under    $   \mathrm{S}^{  \mathcal{M}  }  $  and   $  \mathrm{P}^{  \mathcal{M}  } $, 
  we also have  
 $    \left(  \mathrm{S}^{  \mathcal{M}  }  \right)^r  v   \in    
 \mathsf{Cl}_{   2  \mathsf{rk}_q  \left(  v   \right) }  \left(    \mathcal{W}_k  \left( X   ,  \mathsf{d}^0_0  \right)    \right)   $  
  and  $  H_q  \left(   \left(  \mathrm{S}^{  \mathcal{M}  }  \right)^r  v   \right)  = 
  L \left(   \left(  \mathrm{S}^{  \mathcal{M}  }  \right)^r  v   \right)    \,   $
  where the equality of $  \mathsf{rk}_q  $ follows from (1). 
  Since $  v  \in     \mathsf{Cl}_{   2  \mathsf{rk}_q  \left(  v   \right) }  \left(    \mathcal{W}_k  \left( X   ,  \mathsf{d}^0_0  \right)    \right)   $, 
  \[
  \mathcal{G}^{  \prime }  \left(  v  \right)   \in  
    \mathsf{Cl}_{   2  \mathsf{rk}_q  \left(  v   \right) +1  }  \left(    \mathcal{W}_k  \left( X    ,  \mathsf{d}^0_0   \right)    \right)
  \    \mbox{  and  }    \  
  L  \left(      \mathcal{G}^{  \prime }  \left(  v  \right)      \right)  =     \mathcal{G}^{  \prime }  \left(   L(v)  \right)   
  \]
  where  
\[
\mathcal{G}^{  \prime  }  :=   \begin{cases}
\mathcal{R}_{ m-1}      \mbox{  if  }     \mathcal{G} =  \mathcal{R}_m   \mbox{  where  }  m  \in  \mathbb{Z}
\\
\mathcal{A}_{ n-1}      \mbox{  if  }     \mathcal{G} =  \mathcal{A}_n   \mbox{  where  }  n  \in  \omega \setminus  \lbrace 0, 1  \rbrace 
\      .
\end{cases}
\]
Recall that  $  \mathcal{A}_1 $  is the identity map. 
Since   
 $   w  =    \left(  \mathrm{S}^{  \mathcal{M}  }  \right)^r       \mathcal{G} \left(  v  \right)   
  =  \langle      \mathcal{G}^{  \prime }  \left(  v  \right)    \,  ,  \, 
  \left(  \mathrm{S}^{  \mathcal{M}  }  \right)^r   v  \rangle^{  \mathcal{B}  }  $
\[
w  \in     \mathsf{Cl}_{   2  \mathsf{rk}_q  \left(  v   \right) + 2 }  \left(    \mathcal{W}_k  \left( X    ,  \mathsf{d}^0_0   \right)    \right)
\subseteq    
 \mathsf{Cl}_{   2  \mathsf{rk}_q  \left(  w   \right)  }  \left(    \mathcal{W}_k  \left( X    ,  \mathsf{d}^0_0    \right)    \right)
\]
and  
\begin{align*}
L(w)    
&=  
 \langle      \mathcal{G}^{  \prime }  \left(   L(v)  \right)    \,  ,  \,  L(  \left(  \mathrm{S}^{  \mathcal{M}  }  \right)^r v)  \rangle^{  \mathcal{B}  } 
\\
&=  \langle      \mathcal{G}^{  \prime }  \left(   H_q (v)  \right)    \,  ,  \,  H_q (  \left(  \mathrm{S}^{  \mathcal{M}  }  \right)^r v)  \rangle^{  \mathcal{B}  } 
\\
&= \langle      \mathcal{G}^{  \prime }  \left(   H_q (v)  \right)    \,  ,  \,  \left(  \mathrm{S}^{  \mathcal{M}  }  \right)^r     H_q (  v)  \rangle^{  \mathcal{B}  } 
\\
&=   \left(  \mathrm{S}^{  \mathcal{M}  }  \right)^r      \mathcal{G}  \left(   H_q (v)  \right)  
\\
&=    H_q  \left(    \mathcal{G}  \left(  \left(  \mathrm{S}^{  \mathcal{M}  }  \right)^r  v   \right)    \right) 
\end{align*}
where the last equality uses that  $  H_q  $ is an $L_{ \mathbb{T}  }  $-homomorphism. 
This completes the proof of the claim.

   \begin{quote} {\bf (Claim)} \;\;\;\;\;\;    
$  H_q  $  is  one-to-one. 
 \end{quote}  

Assume for the sake of a contradiction  $  H_q  $  is not one-to-one. 
For  each $  v  \in  \mathbb{T}  $,  we clearly have  $  \left(  \mathrm{S}^{  \mathcal{M}  }  \right)^r   v  \neq 
  \left(  \mathrm{S}^{  \mathcal{M}  }  \right)^s   v    $
  for distinct  $  r, s  \in  \mathbb{Z} $. 
  Hence,   since  $  H_{ q- 1  }  $  is one-to-one, $  D_{ q-1}  $  is closed under   $  \mathrm{S}^{  \mathcal{M}  }   $  
  and   $ \mathrm{P}^{  \mathcal{M}  }  $   and   $H_q   $  is an  $  L_{  \mathbb{T}  } $-homomorphism, 
  this means that  there   exists   $  w  \in D_{ q-1}  $  such that    $  H_{ q-1 } (w)  \neq  H_q  \left(  \mathsf{g} \right)   $. 
We have two cases: 
\begin{itemize}
\item[(I)]   $  w  \in  \mathcal{W}_k  \left( X  \right)  $;

\item[(II)]   $  w    \not\in  \mathcal{W}_k  \left( X  \right)  $. 
\end{itemize} 
We consider (I). 
By Lemma    \ref{SectionQNthClureLemma}, 
we have a  \textsf{Good}   one-to-one  $  L_{  \mathbb{T} }^- $-homomorphism   
 $L  :    \mathsf{Cl}_N  \left(    \mathcal{W}_k  \left( X    ,  \mathsf{d}^0_0   \right)    \right)  \to  \mathbb{T} $ 
 that agrees with $F$ on   $  \mathcal{W}_k  \left( X  \right)  $.  
 By (2),  
 since  $  H_q  \left(  \mathsf{g} \right)  =  H_q  (w)  = F(w)  $, 
 there does not exist   $ c  \in  \big\{  \mathsf{d}^m_n :   \  (n, m  )  \in  \omega  \times  \mathbb{Z}  \,  \rbrace 
\setminus  \left(    \big\{  \mathsf{d}^m_0 :   \   m  \in   \mathbb{Z}  \,  \rbrace   \cup  \mathcal{V} \left(  F \left[ D_0  \right]  \right)    \right)   $ 
such that  $  c  \preceq   H_q (   \mathsf{g} )  $. 
Hence, by the preceding claim 
\[
\mathsf{g}  \in    \mathsf{Cl}_{   2  \mathsf{rk}_q  \left(   \mathsf{g}   \right)  }  \left(    \mathcal{W}_k  \left( X   ,  \mathsf{d}^0_0  \right)    \right)
\subseteq    \mathsf{Cl}_{   N }  \left(    \mathcal{W}_k  \left( X    ,  \mathsf{d}^0_0  \right)    \right)
\   \mbox{  and  }   H_q  \left(  \mathsf{g} \right)  =  L  \left(   \mathsf{g} \right) 
\    .
\]
Since  $L$ agrees with   $F$ on   $   \mathcal{W}_k  \left( X  \right)    $,  we also have  $  H_q  (w)  =  F(w)  =   L(w)  $. 
Hence  $  L(w)   =  L \left(  \mathsf{g}  \right)   \,   $.
But since  $  L  $  is one-to-one,  from    $    L(w)   =  L \left(  \mathsf{g}  \right)     $  we get 
$  w  =  \mathsf{g}  $, 
which contradicts the assumption that  $  w  \neq  \mathsf{g} $.

Finally, we consider (II). 
By  (2),  we may assume  $  \mathsf{g} ,  w  \not\in  \big\{  \mathsf{d}^m_n  :   \   (n,m)  \in  \omega \times  \mathbb{Z} \,  \big\}   $. 
Since  $ H_q (w)  =  H_q  \left(  \mathsf{g}  \right)  $, 
$  H_q  $  is  an    $  L_{  \mathbb{T} }  $-homomorphism  
and     distinct maps in    $   \lbrace  \mathcal{R}_m   :  \   m  \in  \mathbb{Z} \,   \rbrace   \cup 
 \lbrace  \mathcal{A}_n :  \   n  \in    \omega  \setminus \lbrace 0, 1   \rbrace  \,   \rbrace       $
 have disjoint images,  we have the following cases: 
 \begin{itemize}
\item[(IIa)]   $  w  =  \langle  w_0,  w_1  \rangle^{  \mathcal{B}  }  $   and 
 $ \mathsf{g}   =  \langle   w^{  \prime  }_0,     w^{  \prime  }_1  \rangle^{  \mathcal{B}  }  $
 where  $ w_0, w_1,   w^{  \prime  }_0,   w^{  \prime  }_1  \in D_{q-1} $;

\item[(IIb)]    $  w  =  \langle  w_0,  w_1  \rangle^{  \mathcal{B}  }  $   and 
 $  \mathsf{g}    =  \mathcal{G}  \left( u \right)    $
 where  $ w_0, w_1,   u   \in    D_{q-1}  $
 and    $    \mathcal{G}   \in   \lbrace  \mathcal{R}_m   :  \   m  \in  \mathbb{Z} \,   \rbrace   \cup 
 \lbrace  \mathcal{A}_n :  \   n  \in    \omega  \setminus \lbrace 0, 1   \rbrace  \,   \rbrace       $;

\item[(IIc)]   $  w  =  \mathcal{G} \left( v  \right)   $  and     $  \mathsf{g}  =  \langle   w^{  \prime  }_0,     w^{  \prime  }_1  \rangle^{  \mathcal{B}  }  $
 where  $v ,   w^{  \prime  }_0,   w^{  \prime  }_1  \in D_{q-1}$  and    $    \mathcal{G}   \in   \lbrace  \mathcal{R}_m   :  \   m  \in  \mathbb{Z} \,   \rbrace   \cup 
 \lbrace  \mathcal{A}_n :  \   n  \in    \omega  \setminus \lbrace 0, 1   \rbrace  \,   \rbrace       $;

\item[(IId)]  $  w  =  \mathcal{G} \left( v  \right)   $  and     $\mathsf{g}    =  \mathcal{G} \left( u  \right)     $
 where  $ u,  v   \in D_{q-1} $  and    $    \mathcal{G}   \in   \lbrace  \mathcal{R}_m   :  \   m  \in  \mathbb{Z} \,   \rbrace   \cup 
 \lbrace  \mathcal{A}_n :  \   n  \in    \omega  \setminus \lbrace 0, 1   \rbrace  \,   \rbrace         \,    $.
 \end{itemize}
  From these cases we get 
\[
 \begin{cases}
 \langle  H_{q-1} (w_0) ,  H_{q-1} (w_1)   \rangle^{  \mathcal{B}  } 
 = 
 \langle  H_{q-1}   \left( w^{ \prime }_0 \right)  , H_{q-1}   \left( w^{ \prime }_0 \right)   \rangle^{  \mathcal{B}  } 
 &      \mbox{ in case of (IIa)}
\\
 \langle  H_{q-1} (w_0) ,  H_{q-1} (w_1)   \rangle^{  \mathcal{B}  }    =  \mathcal{G}  \left(  H_{q-1}   (u)  \right)  
  &      \mbox{ in case of (IIb)}
\\
   \mathcal{G}  \left(    H_{q-1}   (v)   \right)    = 
 \langle  H_{q-1}   \left( w^{ \prime }_0 \right)  , H_{q-1}   \left( w^{ \prime }_0 \right)   \rangle^{  \mathcal{B}  } 
 &      \mbox{ in case of (IIc)} 
 \\
  \mathcal{G}  \left(  H_{q-1}   (v)  \right)    =   \mathcal{G}  \left(    H_{q-1}   (u)  \right)     &   \mbox{ in case of  (IId).}
\end{cases}
\]
 Since  $  \langle  \cdot ,  \cdot  \rangle^{  \mathcal{B} }   $ and maps    in    $   \lbrace  \mathcal{R}_m   :  \   m  \in  \mathbb{Z} \,   \rbrace   \cup 
 \lbrace  \mathcal{A}_n :  \   n  \in    \omega  \setminus \lbrace 0, 1   \rbrace  \,   \rbrace       $ 
 are one-to-one
  \[
 \begin{cases}
  H_{q-1} (w_0)  =   H_{q-1}   \left( w^{ \prime }_0 \right)   \;  \wedge  \;     H_{q-1} (w_1)   =  H_{q-1}   \left( w^{ \prime }_0 \right)

 &      \mbox{ in case of (IIa)}
\\
  H_{q-1} (w_0)  = \mathcal{G}^{  \prime }   \left(  H_{q-1}   (u)  \right)    
    \;  \wedge  \;     H_{q-1} (w_1)   =  H_{q-1}   \left( u  \right)
    &      \mbox{ in case of (IIb)}
\\
\mathcal{G}^{  \prime }     \left(    H_{q-1}   (v)   \right)    =    H_{q-1}   \left( w^{ \prime }_0 \right) 
\;  \wedge   \;   
   H_{q-1} (v)   =  H_{q-1}   \left(   w^{ \prime }_1  \right)
 &      \mbox{ in case of (IIc)} 
 \\
 H_{q-1}   (v)     =    H_{q-1}   (u)     &   \mbox{ in case of  (IId).}
\end{cases}
\]
Since  $ H_{ q-1}  $  is  an   $  L_{  \mathbb{T}  } $-embedding, we get 
 \[
 \begin{cases}
w_0 =  w^{ \prime }_0   \;  \wedge  \;   w_1 =  w^{ \prime }_0 

 &      \mbox{ in case of (IIa)}
\\
w_0  = \mathcal{G}^{  \prime }   \left( u   \right)    
    \;  \wedge  \;    w_1  =  u  
    &      \mbox{ in case of (IIb)}
\\
\mathcal{G}^{  \prime }     \left(    v \right)    =      w^{ \prime }_0 
\;  \wedge   \;   
v   =  w^{ \prime }_1 
 &      \mbox{ in case of (IIc)} 
 \\
v =  u    &   \mbox{ in case of  (IId).}
\end{cases}
\]
 But this implies  $  w  = \mathsf{g}   $, which  contradicts the assumption that   $  w  \neq  \mathsf{g}  $.

This completes the proof that  $  H_q  $  is one-to-one.

 The next two claims show that    $  H_q  $   is an  $  L_{  \mathbb{T}  }  $-embedding.

     \begin{quote} {\bf (Claim)} \;\;\;\;\;\;    
     Let   $   w,  u,  v   \in D_p $  be such that  $  H_q (w)  =  \langle  H_q(u) , H_q (v)  \rangle^{  \mathcal{B}  }  $.
 Then,  $  w  =   \langle u, v  \rangle^{  \mathcal{B}  }  $.
 \end{quote}  
  
  We  have two cases: 
  \begin{itemize}
\item[(III)]   $  w  \in  \mathcal{W}_k  \left( X  \right)  $;

\item[(IV)]   $  w    \not\in  \mathcal{W}_k  \left( X  \right)  $. 
\end{itemize} 
  We consider (III). 
  By Lemma    \ref{SectionQNthClureLemma}, 
we have a  \textsf{Good}   one-to-one  $  L_{  \mathbb{T} }^- $-homomorphism   
 $L  :    \mathsf{Cl}_N  \left(    \mathcal{W}_k  \left( X    ,  \mathsf{d}^0_0 \right)    \right)  \to  \mathbb{T} $ 
 that agrees with $F$ on   $  \mathcal{W}_k  \left( X  \right)  $.  
 By (2),  
 since  $  \langle  H_q(u) , H_q (v)  \rangle^{  \mathcal{B}  }    =  H_q  (w) = F(w)   $, 
 there does not exist   $ c  \in  \big\{  \mathsf{d}^m_n :   \  (n, m  )  \in  \omega  \times  \mathbb{Z}  \,  \rbrace 
\setminus  \left(    \big\{  \mathsf{d}^m_0 :   \   m  \in   \mathbb{Z}  \,  \rbrace   \cup  \mathcal{V} \left(  F  \left[ D_0  \right]  \right)    \right)   $ 
such that  $  c  \preceq  H_q (u)     $  or  $  c  \preceq    H_q (v)       \,    $. 
Hence, by  one of the claims above 
\[
u, v  \in    \mathsf{Cl}_{  p  }  \left(    \mathcal{W}_k  \left( X   ,  \mathsf{d}^0_0 \right)    \right)
\subseteq    \mathsf{Cl}_{   N -1 }  \left(    \mathcal{W}_k  \left( X   ,  \mathsf{d}^0_0  \right)    \right)
\,   ,  \    \   
H_q (u)  =  L(u)  
\   \mbox{  and  }   H_q  (v)  =  L(v) 
\]
where   $  p   =  \max \left(  2  \mathsf{rk}_q  \left(  u \right)  \,  ,  \,    2  \mathsf{rk}_q  \left(  u \right)     \right)       \,     $. 
Since  $  L  $  is an  $  L_{  \mathbb{T}  }^- $-homomorphism, 
from $  u, v  \in    \mathsf{Cl}_{  p  }  \left(    \mathcal{W}_k  \left( X   ,  \mathsf{d}^0_0  \right)    \right)  $  we get  
\[
 \langle u, v  \rangle^{  \mathcal{B}  }  \in      \mathsf{Cl}_{   N  }  \left(    \mathcal{W}_k  \left( X   ,  \mathsf{d}^0_0  \right)    \right)
 \     \mbox{  and  }   \  
 L  \left(     \langle u, v  \rangle^{  \mathcal{B}  }    \right)   =  
  \langle  L(u) ,  L(v)   \rangle^{  \mathcal{B}  }   = H_q  (w)  
  \     .
\] 
Since  $L$  and  $F$ agree on   $    \mathcal{W}_k  \left( X  \right)     $,  
  $  H_q (w)  = F(w)  =    L(w)    \,  $.
Since  $    L$  is one-to-one,  
from    $   L  \left(     \langle u, v  \rangle^{  \mathcal{B}  }    \right)   =  L(w)  $  we get  
$  w  =      \langle u, v  \rangle^{  \mathcal{B}  }   \,       $.

Finally,  we consider (IV):    $  w    \not\in  \mathcal{W}_k  \left( X  \right)  $   and  
$  H_q (w)  =  \langle  H_q(u) , H_q (v)  \rangle^{  \mathcal{B}  }  $.
By (2) and the fact that  $  H_q  $  fixes  $  \big\{  \mathsf{d}^m_0  :   \   m  \in  \mathbb{Z} \,      \big\} $, 
we have  $  w  \not\in   \big\{  \mathsf{d}^m_n  :   \   (n, m)     \in    \omega  \times  \mathbb{Z} \,      \big\}   $,  
since we would otherwise  get  $  H_q  (w)  \in  \big\{  \mathsf{d}^m_n  :   \   (n, m)    \in    \omega  \times  \mathbb{Z} \,      \big\}   $.
We thus have the following cases:  
\begin{itemize}
\item[(IVa)]  $     w  =  \langle w_0 ,  w_1   \rangle^{  \mathcal{B}  }  $  where  
$  w_0, w_1  \in  D_{q-1}      \,  $.

\item[(IVb)]  $    w  =  \left(  \mathrm{S}^{  \mathcal{M}  }  \right)^r    \mathcal{G}  \left(   h   \right)   $  where  
$  r  \in  \mathbb{Z} $, 
$   h  \in  D_{q-1}      $
and     $  \mathcal{G}  \in    \lbrace  \mathcal{R}_m   :  \   m  \in  \mathbb{Z} \,   \rbrace   \cup 
 \lbrace  \mathcal{A}_n :  \   n  \in    \omega  \setminus \lbrace 0, 1   \rbrace  \,   \rbrace       \,   $.
\end{itemize}
We consider (IVa). 
Since  $ H_q  $  is an $  L_{  \mathbb{T} }  $-homomorphism 
\[
\langle H_q (w_0)  ,   H_q  (w_1)    \rangle^{  \mathcal{B}  }   =  H_q (w)  = 
 \langle  H_q(u) , H_q (v)  \rangle^{  \mathcal{B}  }       
 \    .  
\]
Since    $  \langle  \cdot  ,  \cdot   \rangle^{  \mathcal{B}  }  $  and  $  H_q  $  are both one-to-one, 
$  w_0  =  u $  and  $  w_ 1 = v  $, 
and hence     $  w  =   \langle u, v  \rangle^{  \mathcal{B}  }  $.

We consider (IVb):   $    w  =  \left(  \mathrm{S}^{  \mathcal{M}  }  \right)^r    \mathcal{G}  \left(   h   \right)   $. 
We have 
\[
\left(  \mathrm{S}^{  \mathcal{M}  }  \right)^r    \mathcal{G}  \left(  H_{q}  \left( h \right)   \right) 
= H_q (w)   = 
 \langle  H_q(u) , H_q (v)  \rangle^{  \mathcal{B}  }       
\]
which implies  
\[
\left(  \mathrm{S}^{  \mathcal{M}  }  \right)^r     H_q  (h)  =  H_q (v)  
\   \mbox{  and  }    \  
\mathcal{G}^{  \prime }  \left(  H_{q}  \left( h \right)   \right)   =  H_q  (u) 
\       .
\]
Since  $  H_q  $ is a one-to-one  $  L_{  \mathbb{T} }  $-homomorphism 
\[
\left(  \mathrm{S}^{  \mathcal{M}  }  \right)^r     h  =  v  
\   \mbox{  and  }    \  
\mathcal{G}^{  \prime }  \left(  H_{q}  \left( h \right)   \right)   =  H_q  (u) 
\       .
\]
We thus just need to show that  $   \mathcal{G}^{  \prime  }    \left(   h   \right)   =  u  $.
Since  $  h   \in D_{ q-1} $,  
if  $  u  \in  D_{  q-1}  $,  then this follows from the fact that $ H_{ q-1}  $  is an $  L_{  \mathbb{T}  }  $-embedding. 
So, assume  $  u \in D_q  \setminus  D_{ q-1}  $. 
Then 
\[
\mathcal{G}^{  \prime }  \left(  H_{q-1}  \left( h \right)   \right)   =  H_q  (u) 
=  \begin{cases} 
\langle   H_{q-1}   \left(  u_0  \right)  ,    H_{q-1}   \left(  u_1  \right)  \rangle^{  \mathcal{B}  }  
&     \mbox{  if  }    u  =  \langle  u_0,  u_1    \rangle^{  \mathcal{B}  }   
\\
\mathcal{G}^{  \prime }  \left(  H_{ q-1} ( u^{  \prime }  )  \right)      &   \mbox{  if   }   u  =   \mathcal{G}^{  \prime }  \left(     u^{  \prime }   \right)   
\end{cases}
\]
where we have used that distinct maps in   $   \lbrace  \mathcal{R}_m   :  \   m  \in  \mathbb{Z} \,   \rbrace   \cup 
 \lbrace  \mathcal{A}_n :  \   n  \in    \omega  \setminus \lbrace 0, 1   \rbrace  \,   \rbrace       $
 has disjoint images. 
 If    $  \mathcal{G}^{  \prime }  \left(  H_{q-1}  \left( h \right)   \right)   =    \mathcal{G}^{  \prime }  \left(  H_{ q-1} ( u^{  \prime }  )  \right)   $,   then $  h = u^{  \prime }  $  since    $   \mathcal{G}^{  \prime }   $  and  $  H_{ q-1} $ are both one-to-one. 
 This in turn implies    $    w  =   \langle u, v  \rangle^{  \mathcal{B}  }     \,    $.
 Assume 
 $  \mathcal{G}^{  \prime }  \left(  H_{q-1}  \left( h \right)   \right)    =  
 \langle   H_{q-1}   \left(  u_0  \right)  ,    H_{q-1}   \left(  u_1  \right)  \rangle^{  \mathcal{B}  }      \,     $. 
 Then 
 \[
  \mathcal{G}^{  \prime  \prime  }  \left(  H_{q-1}  \left( h \right)   \right)     =   H_{q-1}   \left(  u_0  \right)   
  \    \mbox{  and  }    \   
   H_{q-1}  \left( h \right)   =   H_{q-1}   \left(  u_1  \right) 
   \     .
 \]
Since  $  H_{ q-1}  $  is an $  L_{  \mathbb{T} }  $-embedding,   we get 
 $ u_0  =   \mathcal{G}^{  \prime   \prime  }    \left(   h   \right)   $ and  $  u_1 = h $,  
 and hence   $  u =   \mathcal{G}^{  \prime  }    \left(   h   \right)   $, which in turn implies 
   $    w  =   \langle u, v  \rangle^{  \mathcal{B}  }     \,    $.

       \begin{quote} {\bf (Claim)} \;\;\;\;\;\;    
       Let   $  \mathcal{G}  \in  \lbrace  \mathcal{R}_m  :   \   m  \in  \mathbb{Z}  \,   \rbrace   \cup  
       \lbrace  \mathcal{A}_n  :   \     n   \in  \omega  \setminus  \lbrace 0, 1,  2  \rbrace    \,   \rbrace      \,      $.
     Let   $   w,  v   \in D_q $  be such that  $  H_q (w)  = \mathcal{G}  \left(    H_q (v)   \right)   $.
 Then,  $  w  =  \mathcal{G}  \left(    v  \right)     $.
 \end{quote}  
  
  We prove the claim by induction on  $  \mathsf{rk}_q  \left(  w  \right)  $. 
  We consider the base  case    $   \mathsf{rk}_q  \left(  w  \right)    = 0  $,  that is, 
  $  w  \in  \mathcal{W}_k  \left(  X   \right)   $. 
  We have  $  H_q (w)  =  F(w)  $. 
  Let    $  v  \in D_q  $  such that   $  H_q (w)  = \mathcal{G}  \left(    H_q (v)   \right)   $. 
  Since  $F :  \mathcal{W}_{k+N }  \left(  X   \right)     \to  \mathbb{T}  $  is an  $ L_{  \mathbb{T}  }  $-embedding
  and  $  N   >  0  $,  
  from  $   F(w)  = \mathcal{G}  \left(    H_q (v)   \right)   $  we get that   there must also exist  
  $  u  \in    \mathcal{W}_{k+1 }  \left(  X   \right)   $  such that  
  $  w  =   \mathcal{G}  \left(    u   \right)   $. 
  Hence 
  \[
 \mathcal{G}  \left(   F (u)   \right)   =  F(w)   =   H_q (w)  = \mathcal{G}  \left(    H_q (v)   \right) 
 \      .
  \]
  Since  $   \mathcal{G}   $  is one-to-one
\[
 F(u)  =  H_q (v)   
 \        .
 \]
    We just need to show that  $  u  =  v  $, and it will then follow that   $  w  =  \mathcal{G}  \left(    v  \right)     $.
    By Lemma     \ref{SectionQNthClureLemma},  
    we have a  \textsf{Good} one-to-one  $  L_{  \mathbb{T} }^-$-homomorphism 
    $  L  :   \mathsf{Cl}_{  N-1 }   \left(    \mathcal{W}_{k+1}   \left(  X   ,  \mathsf{d}^0_0  \right)      \right)      \to  \mathbb{T}  $  
    that agrees with $F$ on  $   \mathcal{W}_{k+1}   \left(  X   \right)    $.
    By    the  claim we used to show that  $  H_q  $  is one-to-one
    \[
    v  \in    \mathsf{Cl}_{  2  \mathsf{rk} \left(  v \right)  }   \left(    \mathcal{W}_k  \left(  X   ,  \mathsf{d}^0_0  \right)      \right)   
    \subseteq  
    \mathsf{Cl}_{N-1}  \left(    \mathcal{W}_k  \left(  X   ,  \mathsf{d}^0_0  \right)      \right)   
    \subseteq    \mathsf{Cl}_{  N-1 }   \left(    \mathcal{W}_{k+1}   \left(  X   ,  \mathsf{d}^0_0  \right)      \right)   
    \      .
    \]
Since  $  L$  and  $  H_q  $  are both       $  L_{  \mathbb{T} }^-$-homomorphisms, 
we must have  $  L(v)  = H_q  (v)  $. 
Since  $L$ and  $F$ agree on   $   \mathcal{W}_{k+1}   \left(  X   \right)    $, we have  $ F(u)  =  L(u)  $. 
That is 
\[
L(u)  =  F(w)  =  H_q ( v)  =  L(v) 
\   .
\]
Since  $L$ is one-to-one,  $  u =  v  $, and hence  $  w  =  \mathcal{G}  \left(    v  \right)     $.

  We consider the case    $    \mathsf{rk}_q  \left(  w  \right)   >  0 $. 
  We have the following cases: 
  \begin{itemize}
  \item[(a)]  $  w  =  \langle  w_0,  w_1  \rangle^{  \mathcal{B}  }  $  where  $  w_0,  w_1  \in  D_{ q-1}  $

  \item[(b)]  $  w  =  \mathcal{H} \left( u  \right)   $   where  $  u \in D_{ q-1}  $  and  
  $  \mathcal{H}  \in     \lbrace  \mathcal{R}_m  :   \   m  \in  \mathbb{Z}  \,   \rbrace   \cup  
       \lbrace  \mathcal{A}_n  :   \     n   \in  \omega  \setminus  \lbrace 0, 1,  2  \rbrace    \,   \rbrace      \,      $.
\end{itemize}    
We consider (a). 
\[
\mathcal{G}^{  \prime  }  =  \begin{cases}
\mathcal{R}_m    &    \mbox{  if  }   \mathcal{G}  =  \mathcal{R}_m  \mbox{  where  }   m  \in  \mathbb{Z}
\\
\mathcal{A}_n    &    \mbox{  if  }   \mathcal{G}  =  \mathcal{A}_n   \mbox{  where  }   n  \in  \omega  \setminus  \lbrace 0, 1,  2  \rbrace 
\     .
\end{cases}
\]
We have 
\[
\langle  H_q  (w_0)  ,  H_q (w_1)   \rangle^{  \mathcal{B}  }    =  
H_q  (w)    =  
\langle   \mathcal{G}^{  \prime  }   \left(  H_q (v)  \right)    \,   ,  \,    H_q (v)    \rangle^{  \mathcal{B}  } 
\       .
\]
Since  $  \langle  \cdot  ,  \cdot  \rangle^{  \mathcal{B}  }   $  is one-to-one
\[
 H_q  (w_0)    =    \mathcal{G}^{  \prime  }   \left(  H_q (v)  \right)       \     \mbox{  and  }    \   
 H_q (w_1)     =   H_q (v)   
 \   .
\]
Since  $  H_q  $  is one-to-one,  from   $   H_q (w_1)     =   H_q (v)    $  we get  $  w_1  =  v  $. 
By the induction hypothesis,  from  $   H_q  (w_0)    =    \mathcal{G}^{  \prime  }   \left(  H_q (v)  \right)    $  we get 
$  w_0  =    \mathcal{G}^{  \prime  }   \left(  v  \right)   $. 
It follows from this that    $  w  =  \mathcal{G}  \left(    v  \right)     $.

We consider (b). 
We have 
\[
\mathcal{H}  \left(  H_q (u)  \right)    =  H_q (w)   =     \mathcal{G}  \left(    H_q (v)   \right)    
\      .
\]
We have  $  \mathcal{H}   =  \mathcal{G}   $  since distinct maps in 
$  \lbrace  \mathcal{R}_m  :   \   m  \in  \mathbb{Z}  \,   \rbrace   \cup  
       \lbrace  \mathcal{A}_n  :   \     n   \in  \omega  \setminus  \lbrace 0, 1,  2  \rbrace    \,   \rbrace    $
       have disjoint images. 
 Since      $  \mathcal{G}   $ is one-to-one, 
 from     $  \mathcal{G}  \left(  H_q (u)  \right)    =  \mathcal{G}  \left(    H_q (v)   \right)      $  
 we get  $  H_q (u)   =  H_q (v)   $.
 Since  $  H_q  $  is one-to-one,   $  u  =  v  $. 
 It follows from this that    $  w  =  \mathcal{G}  \left(    v  \right)     $.
\end{proof}

\section{Proof of Theorem \ref{QPLUSCANTORPAIRING}}

By   looking at the statement of  Lemma  \ref{SectionQFirstExtensionLemma}, 
we will need to start  with very large sets in order to apply the lemma  $n$  times. 
For each  $ 0 <   n  \in  \omega $,  we define the function  $  \rho_n  :   \omega  \to  \omega $  by recursion  
\[
\rho_n (0) =   0
\     \mbox{  and  }       \    
\rho_n (k+1)  =  2  ( n +2)      \left(   \rho_n ( k ) +1       \right)     2^{   \rho_n ( k ) +2}   
\]
The only property of  $  \rho_n $  we will use is the following lemma.

\begin{lemma}   
Let  $  0 <  n  <   \omega   $.
If   $X  \subseteq   \mathbb{T} $  is  a  finite set of cardinality $  \leq  n +2  $,  then 
\[
2 \vert   \mathcal{V}_{  \rho_n ( k ) +1   }        \left(   X\right)    \vert   \leq    \rho_n (k+1)   -   \rho_n ( k )  - 1  
\   \mbox{  for  all }  k  \in  \omega   
\      .
\]
\end{lemma}
\begin{proof}

We have  
\begin{align*}
2  \vert     \mathcal{V}_{  \rho_n ( k ) +1   }     \left(   X  \right)    \vert  
&  \leq  2    \vert   X  \vert       \left(  2^{   \rho_n ( k ) +2} - 1 \right)  
  \\
    &  \leq  2  (n+2)       \left(  2^{   \rho_n ( k ) +2} - 1 \right)  
  \\
  &  \leq  2   (n+2)          \left(   \rho_n ( k ) +1       \right)     2^{   \rho_n ( k ) +2}   -  \left(   \rho_n (k) + 1  \right) 
  \\
  &   =        \rho_n ( k+1 )    -  \left(   \rho_n (k) + 1  \right) 
\        .    & \qedhere
\end{align*}

\end{proof}

We finally have everything we need to complete the proof of   Theorem  \ref{QPLUSCANTORPAIRING}.

\begin{proof}[Proof of  Theorem  \ref{QPLUSCANTORPAIRING}]

We define the sets  $  A_n  $  as follows: 
\[
A_n  :=   \big\{   \mathsf{d}^0_{ k }  :   \   k \in  \lbrace 1, 2,  \ldots ,  \rho_n  \left( n  +1 \right)    \rbrace   \   \big\} 
\    \    \mbox{  for   }    n  \in  \omega 
\         . 
\]
We don't include   $  \mathsf{d}^0_0 $   since  \textsf{Good} maps fix  
$  \big\{   \mathsf{d}^m_{ 0 }  :    \    m  \in  \mathbb{Z}   \,   \big\}    \,   $.

Fix $ n  \in  \omega  $.   
Fix  an element $ w  $  of   $  \mathcal{M} $. 
To complete the proof,  we need to show that there exists  $  a_0   \in  A_n  $  and  $  b_0  \in M \setminus A_n $  such that 
$  \left( \mathcal{M}   \, , \, w,     a_0    \right)  \equiv_n    \left( \mathcal{M}   \, , \,  w   ,   b_0    \right)    $.  
We have two  cases: 
 \begin{enumerate}
\item     $  w     \in  \mathbb{N}  $

\item $   w     \in   \mathbb{T}    \,       $. 
\end{enumerate}
We consider case (1). 
Let  $  a_0  :=   \mathsf{d}^0_1 $  and  let  $  a_0  :=   \mathsf{d}^0_p  $  where  
$ p  >   \rho_n  \left( n+1  \right)  $. 
Then,    $  a_0   \in  A_n  $  and  $  b_0    \not\in  A_n  $. 
We  define by recursion an   automorphism  $    \tau :  \mathcal{M}   \to   \mathcal{M} $  
such that  $  \tau  \left(   a_0  \right)   =  b_0   $: 
\[
\tau  (  g  )  =  \begin{cases}
g        &   \mbox{ if  }      g  \in  \mathbb{N}  \cup  \big\{ \mathsf{d}^m_n  :   \   n\neq  1,  p   \,  ,    \    (n,m)  \in  \omega  \times  \mathbb{Z}  \,   \rbrace  
\\
 \mathsf{d}^m_p     &     \mbox{  if   }  g  =    \mathsf{d}^m_1     \    \mbox{  and   }    m  \in  \mathbb{Z}
 \\
 \mathsf{d}^m_1     &     \mbox{  if   }  g  =    \mathsf{d}^m_p      \    \mbox{  and   }    m  \in  \mathbb{Z}  
 \\
 \langle  \tau \left(  g_0  \right)    \,  ,  \,    \tau \left(  g_1  \right)   \rangle^{  \mathcal{B} }  
         &   \mbox{  if   }     \mathbb{T}  \ni    g  =  \langle  g_0,  g_1     \rangle^{  \mathcal{B} }   
  \\
\mathcal{R}_m  \left(  \tau \left(  v  \right)     \right)  
         &   \mbox{  if   }   \mathbb{T}  \ni    g  =  \mathcal{R}_m  \left(  v    \right)      \    \mbox{  and   }    m  \in  \mathbb{Z}
         \        .
\end{cases}
\]
This is clearly a bijection that  fixes   $  \mathbb{N} $. 
Furthermore, the restriction of  $  \tau  $  to  $  \mathbb{T} $  is  a  \textsf{Good} $  L_{  \mathbb{T} } $-homomorphism.
This shows that 
 $  \left( \mathcal{M}   \, , \,  w  , a_0     \right)   \simeq   
   \left( \mathcal{M}   \, , \,   w  , b_0    \right)    $, 
 and hence  $  \left( \mathcal{M}   \, , \,    w  , a_0   \right)  \equiv_n     \left( \mathcal{M}   \, , \,   w , b_0   \right)    $.

We consider case (2).
Recall that    $  \mathrm{S}^{  \mathcal{M}  }   \left(  \mathsf{d}^m_n  \right)  =   \mathsf{d}^{m+1}_n  $  and 
$  \mathrm{P}^{  \mathcal{M}  }   \left(  \mathsf{d}^m_n  \right)  =   \mathsf{d}^{m-1}_n  $ 
for all $  (n,m )  \in  \omega  \times  \mathbb{Z} $. 
Since  $  \vert  \mathcal{V}_{   \rho_n (n  ) + 1 }  \left(   w  \right)    \vert + 1   <   \rho_n  (n+1)    $, 
there exist  $  q^{  \star }  \in   \lbrace 1, 2,  \ldots ,  \rho_n  \left( n  + 1 \right)     $  
and  $  p^{  \star }  \in  \omega  \setminus   \lbrace 0, 1, 2,  \ldots ,  \rho_n  \left( n  +1 \right)    \rbrace  $  such that  
\[
\left(   \forall  m  \in  \mathbb{Z} \right)   \left[   \ 
   \mathsf{d}^m_{  q^{  \star }  }    \not\in    \mathcal{V}_{   \rho_n (n  ) + 1 }  \left(   w \right)  
   \   \wedge    \   
      \mathsf{d}^m_{  p^{  \star }  }    \not\in    \mathcal{V}_{   \rho_n (n  ) + 1 }  \left(   w  \right)  
\         \right]
\         .
\]
Since     $  \mathcal{W}_{   \rho_n (n  ) + 1 }  \left(   w  \right)   $  is the closure of  
$   \mathcal{V}_{   \rho_n (n  ) + 1 }  \left(   w  \right)    $  under   
$  \mathrm{S}^{  \mathcal{M}  }  $  and     $  \mathrm{P}^{  \mathcal{M}  }  $
\[
\left(   \forall  m  \in  \mathbb{Z} \right)   \left[   \ 
   \mathsf{d}^m_{  q^{  \star }  }    \not\in    \mathcal{W}_{   \rho_n (n  ) + 1 }  \left(  w  \right)  
   \   \wedge    \   
      \mathsf{d}^m_{  p^{  \star }  }    \not\in    \mathcal{W}_{   \rho_n (n  ) + 1 }  \left(   w   \right)  
\         \right]
\         .   \tag{*}
\]
Observe that  $  \chi  \left(  \mathsf{d}^m_{  q^{  \star }  }    \right)  =  m =   \chi  \left(  \mathsf{d}^m_{  p^{  \star }  }    \right)      $
for all $  m \in  \mathbb{Z}  $. 
Let  
\[
a_0  :=      \mathsf{d}^0_{  q^{  \star }  }    
\       \         \mbox{  and   }     \       \  
b_0  :=      \mathsf{d}^0_{  p^{  \star }  }   
\       .
\]
Then,    $  a_0  \in A_n  $  and    $ b_0   \not\in  A_n     \,     $.
It follows from (*) that we have  a \textsf{Good}  $ L_{  \mathbb{T} }  $-isomorphism 
\[
H  :     \mathcal{W}_{ \rho_n (n)    }   \left(   w , a_0  \right)   \to     \mathcal{W}_{ \rho_n (n)   }   \left(  w, b_0  \right)    
 \   \mbox{  where  }   \   
H  \left(  w\right)  =  w
\, ,     \   
H  \left(  a_0  \right)   =    b_0 
\]  
since  
$
 \mathcal{W}_{ \rho_n (n)   }   \left(  w  ,   a_0   \right)   =  
  \mathcal{W}_{ \rho_n (n)   }   \left(  w  \right)    \cup   
  \big\{  \mathsf{d}^m_{  q^{  \star }  }     :    \    m  \in  \mathbb{Z}   \,    \big\}    
$
and 
$
 \mathcal{W}_{ \rho_n (n)   }   \left(  w  ,   b_0   \right)   =  
  \mathcal{W}_{ \rho_n (n)   }   \left(  w  \right)    \cup   
  \big\{  \mathsf{d}^m_{  p^{  \star }  }     :    \    m  \in  \mathbb{Z}   \,    \big\}    
\,     $. 
We  define  $H  $   as follows  
\[
H  \left(  g   \right)    =  \begin{cases} 
g      &     \mbox{  if   }     g  \in      \mathcal{W}_{ \rho_n (n)   }   \left(   w   \right)    
\\
 \mathsf{d}^m_{  p^{  \star }  }       &     \mbox{  if   }   w  =   \mathsf{d}^m_{  q^{  \star }  } 
 \       .
\end{cases}
\]
This map is an $L_{  \mathbb{T} }  $-embedding  since  
\[
  \mathcal{W}_{ \rho_n (n) +1   }   \left(  w  \right)    \cap   
  \big\{  \mathsf{d}^m_{  q^{  \star }  }     :    \    m  \in  \mathbb{Z}   \,    \big\}    =  \emptyset
  \       \mbox{  and   }     \    
    \mathcal{W}_{ \rho_n (n) +1   }   \left(  w  \right)    \cap   
  \big\{  \mathsf{d}^m_{  p^{  \star }  }     :    \    m  \in  \mathbb{Z}   \,    \big\}    =  \emptyset
  \        .
\]

We prove  $  \left( \mathcal{M}   \, , \,  w  , a_0      \right)  \equiv_n     
\left( \mathcal{M}   \, , \,   w , b_0     \right)    $ 
by showing that  player  $ \exists $ has a winning strategy for the       Ehrenfeucht-Fra\"iss\'e game
$   \mathsf{EF}_{ n  }  \left[   \,  
  \left(   \,     \mathcal{M}    \, , \,  w    \,   ,  \,       a_0     \,   \right)         \,   ,   \,   
  \left(   \,     \mathcal{M}    \, , \,    w     \,   ,  \,      b_0      \,  \right)        \,        \right]        \,  $.
Player  $  \exists  $  ensures that if 
$  (a_1, b_1) ,  (a_2,  b_2)   \ldots ,  (a_n,  b_n)    $  is the play after  $  n  $  rounds,  
then the following holds:  
\begin{enumerate}
\item   For all  $  i   \in  \lbrace  1, 2,  \ldots ,  n \rbrace  $,  if   $  a_i  \in  \mathbb{N}  $,  then $  a_i  =  b_i     \,  $.

\item   For all  $  i   \in  \lbrace  1, 2,  \ldots ,  n \rbrace  $,  if   $  b_i  \in  \mathbb{N}  $,  then $  a_i  =  b_i     \,  $.

\item  For each  $  m  \in  \lbrace 1, 2,  \ldots ,  n  \rbrace  $, 
    let   $ J_m    \subseteq   \lbrace  0, 1,  \ldots ,  m  \rbrace  $  be the set of all indexes  $  j  \in     \lbrace 0, 1,  \ldots ,  m  \rbrace $  such that  $  a_j  \in  \mathbb{T} $   (and hence  also $  b_j  \in    \mathbb{T}  $).
    For each  $ m    \in   \lbrace 0, 1,  \ldots ,  n  \rbrace   $, 
  there exists  a  \textsf{Good} $  L_{  \mathbb{T} } $-isomorphism 
\[
   F_m   : \mathcal{W}_{  \rho_n  \left(  n -m  \right)  }  \left(   w ,      \lbrace  a_j  :   \  j  \in  J_m  \rbrace   \right) 
  \to 
 \mathcal{W}_{  \rho_n  \left(  n -m  \right)  }  \left(    w, ,      \lbrace  b_j  :   \  j  \in  J_m  \rbrace   \right)    
     \] 
such that    $    F_m  \left(   w \right)  =  w    $  and    $   F_m  \left(   a_j    \right)  =  b_j   $  for all   $  j  \in J_m   $. 
\end{enumerate}
At stage  $m+1  \leq  n $, if player  $  \forall  $  picks  $  a_{ m+1}  \in  \mathbb{T} $, 
then  player  $  \exists  $  applies  $  F_m  $  to Lemma \ref{SectionQFirstExtensionLemma}   and gets  $  b_{ m+1} \in \mathbb{T} $  and  $  F_{m+1} $. 
If  player  $  \forall  $  picks  $  b_{ m+1}  \in  \mathbb{T} $, 
then  player  $  \exists  $  applies  $  F_m^{-1}  $  to Lemma  \ref{SectionQFirstExtensionLemma}  and gets  $  a_{ m+1}  \in \mathbb{T}  $  and  $  F_{m+1}^{-1} $. 
Player $  \exists  $  wins the game since  
the map  $  G  :  \mathbb{N}  \cup   \mathcal{W}_{ 0 }  \left(     w,   \lbrace  a_j  :   \  j  \in  J_n  \rbrace    \right)   \to  M  $  
that  fixes    $  \mathbb{N}    $  and agrees with  $  F_n  $  on 
$   \mathcal{W}_{ 0 }  \left(    w,      \lbrace  a_j  :   \  j  \in  J_n  \rbrace   \right)   =  
\lbrace  w  \rbrace  \cup   \lbrace  a_j  :   \  j  \in  J_n  \rbrace        $
  is a partial embedding   into   $  \mathcal{M}  $.

This completes the proof of    Theorem  \ref{QPLUSCANTORPAIRING}. 
\end{proof}

\section*{ }

\end{document}